\documentclass[12pt]{amsart}
\usepackage[margin=1.0in]{geometry}
\usepackage{amsmath,amssymb,amsfonts,graphicx,color,fancyhdr,psfrag,comment,enumerate}
\usepackage[latin1]{inputenc}
\usepackage[hyperpageref]{backref}
\usepackage[colorlinks=true, pdfstartview=FitV, linkcolor=blue, citecolor=blue, urlcolor=blue]{hyperref}
\allowdisplaybreaks

\usepackage{epstopdf}
\usepackage{mathrsfs}
\usepackage{epsfig}
\usepackage{hyperref} 
\usepackage{ifthen}
\usepackage{float}
\usepackage{enumerate}
\usepackage{mathtools}
\usepackage[bottom]{footmisc}
\usepackage{tikz}
\usetikzlibrary{matrix,shapes,arrows,positioning,chains}
\usepackage{amssymb,amsmath,amsfonts}
\usepackage{bbm}

\numberwithin{equation}{section}

\newtheorem{theorem}{Theorem}[section]

\newtheorem{lemma}[theorem]{Lemma}

\newtheorem{example}[theorem]{Example}
\newtheorem{remark}[theorem]{Remark}
\newtheorem{corollary}[theorem]{Corollary}

\newcommand{\X}{\mathbb{R}^{n_1+n_2}}

\DeclarePairedDelimiter\floor{\lfloor}{\rfloor}

\begin{document}
\title[General Grushin Pseudo-multipliers]{On some operator-valued Fourier pseudo-multipliers associated to Grushin operators}

\author[S. Bagchi, R. Basak, R. Garg and A. Ghosh]
{Sayan Bagchi \and Riju Basak \and Rahul Garg \and Abhishek Ghosh}

\address[S. Bagchi]{Department of Mathematics and Statistics, Indian Institute of Science Education and Research Kolkata, Mohanpur--741246, West Bengal, India.}
\email{sayan.bagchi@iiserkol.ac.in}

\address[R. Basak]{Department of Mathematics, Indian Institute of Science Education and Research Bhopal, Bhopal--462066, Madhya Pradesh, India.}
\curraddr{Department of Mathematics, Indian Institute of Science Education and Research Mohali, S.A.S. Nagar- 140306, Punjab, India.}
\email{rijubasak52@gmail.com}

\address[R. Garg]{Department of Mathematics, Indian Institute of Science Education and Research Bhopal, Bhopal--462066, Madhya Pradesh, India.}
\email{rahulgarg@iiserb.ac.in}

\address[A. Ghosh]{Institute of Mathematics, Polish Academy of Sciences, Ul. \'Sniadeckich 8, 00-656 Warsaw, Poland.} 
\email{agosh@impan.pl}

\subjclass[2020]{58J40, 43A85, 42B25}

\keywords{Hermite operator, Grushin operator, spectral multipliers, pseudo-differential operators, maximal operator, sparse operators, weighted boundedness}

\begin{abstract}
This is a continuation of our work \cite{BBGG-1, BBGG-2} where we have initiated the study of sparse domination and quantitative weighted estimates for Grushin pseudo-multipliers. In this article, we further extend this analysis to study analogous estimates for a family of operator-valued Fourier pseudo-multipliers associated to Grushin operators $G = - \Delta_{x^{\prime}} - |x'|^2 \Delta_{x^{\prime \prime}}$ on $\mathbb{R}^{n_1+n_2}.$ 
\end{abstract}
\maketitle

\tableofcontents

\section{Introduction}
Multiplier theorems are always of deep interest in analysis due to their vast applicability in various problems pertaining to harmonic analysis and partial differential equations. In the present article we study general multiplier operators in the context of Grushin operators. Denoting points in $\X$ by $x=(x^{\prime}, x^{\prime\prime}) \in \mathbb{R}^{n_1} \times \mathbb{R}^{n_2}$, consider the Grushin operator
$$G = - \Delta_{x^{\prime}} - |x'|^2 \Delta_{x^{\prime \prime}}$$
on $\X$. Given a bounded function $m$ on $\mathbb{R}_{+}$, spectral multiplier $m(G)$ is the operator defined via the spectral integration associated to the spectral measure of $G$. Grushin operators were introduced in \cite{Grushin70}. The operator $G$ is degenerate elliptic along the $n_2$-dimensional plane $\{0\}\times \mathbb{R}^{n_2}$. It is studied in various contexts related to Dirichlet problems in weighted Sobolev spaces, free boundary problems in partial differential equations etc. In particular, when $n_2=1$, it is closely connected to the sub-Laplacian on the Heisenberg group (see, for example, \cite{Dziubanski-Jotsaroop-Hardy-BMO-Grushin}). In fact, it is same as the sub-Laplacian acting on cylindrically symmetric functions. Denoting by $d$ the Grushin control distance (see \eqref{def:grushin-control-metric}), we write $B(x, r) := \{y \in \mathbb{R}^{n_1+n_2} : d(x, y)< r\}$ for the ball of radius r, centered at $x$, and $|B(x, r)|$ stands for the (Lebesgue) measure of the ball $B(x, r)$. It is known that $\left|B(x,r)\right| \sim r^{n_1 + n_2} \max\{r, |x^{\prime}|\}^{ n_2}$, which implies that $|B(x, sr)| \lesssim_{n_1, n_2} (1+s)^{Q} |B(x,r)|$ for every $s>0$, where $Q = n_{1}+2 n_{2}$. Thus, $(\mathbb{R}^{n_1+n_2}, d, |\cdot|)$ is a homogeneous metric space with homogeneous dimension $Q=n_1+2n_2$. It then follows from the work of Duong, Ouhabaz and Sikora  \cite{DuongOuhabazSikoraWeightedPlancherel2002JFA} on more general homogeneous spaces that if 
\begin{align} \label{Mult-condn}
\|\eta\, \delta_{r}m\|_{W^{2, s}}<\infty
\end{align}
with $s>Q/2$, where $\eta$ is a non-zero cut-off function, then $m(G)$ is bounded on $L^p(\X)$ for $1<p<\infty$ and is of weak-type $(1, 1)$. In \cite{JotsaroopSanjayThangaveluRieszTransformsGrushin}, multiplier theorem was studied using $R$-boundedness principle. Sharp multiplier theorem for the Grushin operator (that is, $L^p$ for $1<p<\infty$ and weak $(1,1)$ boundedness of $m(G)$ with $m$ satisfying condition \eqref{Mult-condn} for only $s>\frac{n_1+n_2}{2}$) was established in the works of Martini--Sikora \cite{MartiniSikoraGrushinMRL} and Martini--M\"uller \cite{MartiniMullerGrushinRevistaMath}, see also \cite{Martini-JFAA-2022, Martini-IMRN-2022} for other interesting results related to multiplier theorems for the Grushin operator. Let us also mention that related multiplier theorems in the context of weighted $L^p$ spaces were studied in \cite{DuongSikoraYanJFA2011}. 

The purpose of this article is to study sharp weighted estimates for a large class of interesting operators related to Grushin operators but which are not covered by the methods developed in \cite{DuongOuhabazSikoraWeightedPlancherel2002JFA, DuongSikoraYanJFA2011}. In order to motivate the class of operators that we study, let us first recall the relation between Grushin (pseudo) multipliers and Hermite (pseudo) multipliers. Let $f^\lambda$ denote the inverse Fourier transform of a nice function $f$ on $\X$ in $x^{\prime \prime}$-variable, given by $$f^\lambda(x^{\prime}) = (2 \pi)^{- n_2} \int_{\mathbb{R}^{n_2}} f (x^{\prime}, x^{\prime \prime}) e^{i \lambda \cdot x^{\prime \prime}} \, d x^{\prime \prime}.$$
Then one can write the Grushin operator as follows
\begin{align*} 
Gf(x) = \int_{\mathbb{R}^{n_2}} e^{-i \lambda \cdot x^{\prime \prime}} \left(H(\lambda) f^{\lambda} \right) (x^{\prime}) \, d\lambda, 
\end{align*} 
where $H(\lambda) = - \Delta_{x'} + |\lambda|^2 |x'|^2$ is the scaled Hermite operator. As in \cite{Bagchi-Garg-1}, one then defines Grushin pseudo-multipliers by 
\begin{align*} 
m(x, G)f(x) = \int_{\mathbb{R}^{n_2}} e^{-i \lambda \cdot x^{\prime \prime}} \left(m(x, H(\lambda)) f^{\lambda} \right) (x^{\prime}) \, d\lambda,
\end{align*} 
where $m(x, H(\lambda))$ is the Hermite pseudo-multiplier introduced in \cite{EppersonHermitePseudo} and subsequently studied in \cite{BagchiThangaveluHermitePseudo}. Recently, Grushin pseudo-multipliers and their quantitative weighted estimates are investigated systematically in \cite{Bagchi-Garg-1,BBGG-1,BBGG-2}. 

\medskip 
\begin{itemize}
\item The above can be thought of our point of departure for this article and it is natural to consider more general operators of the following form: 
\begin{align} \label{def:Theta-operator} 
\Theta f(x) = \int_{\mathbb{R}^{n_2}} e^{-i \lambda \cdot x^{\prime \prime}} \left( \Theta(x'', \lambda) f^{\lambda} \right) (x^{\prime}) \, d\lambda.
\end{align} 
Our analysis of these ``\textit{operator-valued Fourier pseudo-multipliers}" is also motivated by the works of \cite{MauceriWeylTransformJFA80,BagchiThangaveluHermitePseudo} which we shall explain shortly. In this direction, we study quantitative weighted estimates with some natural ``derivative" conditions on $\Theta(x'', \lambda)$ by employing the modern machinery of sparse domination.  

\medskip 
\item In particular, given any $\mu_{0}\in \mathbb{N}^{n_1},$ consider the following ``\textit{shifted Grushin pseudo-multipliers}" associated to the joint functional calculus of $G$ (for details, see Subsection \ref{subsec:main-results-joint-functional-calculus}),
\begin{align} \label{def:pseudo-mult-op-with-conj-shifts} 
T f(x)  = \int_{\mathbb{R}^{n_2}} e^{-i\lambda \cdot x''} \sum_{\mu} m(x, (2 (\mu - \mu_0) + \tilde{1}) |\lambda|, \lambda) \left( f^{\lambda}, \Phi^{\lambda}_{\mu} \right) \Phi^{\lambda}_{\mu}(x') \, d\lambda.
\end{align} 
These operators played a very important role in the analysis of Grushin (pseudo) multipliers in \cite{Bagchi-Garg-1, MartiniMullerGrushinRevistaMath}. Since the shifted Hermite multipliers do not behave well in the $L^p$ scale for $p\neq 2$, it is a difficult task to obtain $L^p$-estimates for these operators and studying this is one of our primary goals in this article. In a limited case, we accomplish this as an application of our analysis of general operators of the form \eqref{def:Theta-operator}. It is not obvious to obtain these estimates for shifted Grushin pseudo-multipliers from those of Grushin pseudo-multipliers, see Remark~\ref{rem:shift-vs-joint-spectral} for more details.
\end{itemize}

In order to motivate our results, let us first recall some relevant literature. In \cite{MauceriWeylTransformJFA80}, Mauceri studied Mihlin--H\"ormander-type multiplier theorem for Weyl multipliers with conditions given in terms of certain non-commutative derivatives of the operator $M \in B(L^2(\mathbb{R}^n))$. On the other hand, for general Fourier multipliers on the Heisenberg group, analogous Mihlin--H\"ormander-type conditions were studied in \cite{De-Michele-Mauceri-Heisenberg79,Chin-Cheng-Lin-Heisenberg,BagchiFourierMultipliersHeisenbergStudia}. Inspired by the work of \cite{MauceriWeylTransformJFA80}, the first author and S. Thangavelu \cite{BagchiThangaveluHermitePseudo} studied some weighted $L^p$-boundedness results for a general class of operators on $L^2(\mathbb{R}^n)$. Our assumptions are in fact motivated in part by those of \cite{BagchiThangaveluHermitePseudo}, and therefore before stating our assumptions and results in detail let us briefly recall those of \cite{BagchiThangaveluHermitePseudo}. For the same, we need to introduce the non-commutative derivatives $\delta_j$ and $\bar{\delta}_j$ of operators $M$ on $L^2 (\mathbb{R}^n) $ which are defined by 
$$ \delta_j M = [M, A_j] \quad \text{and} \quad \bar{\delta}_j M = [A_j^*, M], $$
for $1 \leq j \leq n$, where $A_j$ and $A_j^*$ are the annihilation and creation operators 
$$A_j = \frac{\partial}{\partial x_j} +  x_j,\,\,\text{and}\,\, A_j^* = - \frac{\partial}{\partial x_j} + x_j.$$
Let $P_k$ denote the Hermite projection of $L^2 (\mathbb{R}^n) $ onto the eigenspace of the Hermite operator $H = - \Delta + |x|^2$ with eigenvalues $2k+n$. For $j \in \mathbb{N}$, let $\chi_j$ stand for the dyadic projection $\chi_j = \sum_{2^{j} \leq 2 k + n < 2^{j+1}} P_k$. Denoting the kernel of a Hilbert-Schmidt operator $M$ on $L^2 (\mathbb{R}^n)$ by $M(x,y)$, following results were proved in \cite{BagchiThangaveluHermitePseudo}: 

Let $M \in \mathcal{B} \left( L^2 (\mathbb{R}^n) \right)$ be such that $\delta^{\nu} \bar{\delta}^{\gamma} M \in \mathcal{B} \left( L^2 (\mathbb{R}^n) \right)$ and 
\begin{align} \label{cond-Bagchi-Thangavelu-1}
\sup_{x \in \mathbb{R}^n} \int_{\mathbb{R}^n} \left| \left( \left( \delta^{\nu} \bar{\delta}^{\gamma} M \right) \chi_j \right) (x,y) \right|^2 dy \lesssim 2^{j \left( \frac{n}{2} - |\nu| - |\gamma| \right)}
\end{align}
for all $|\nu| + |\gamma| \leq N$, for some $N \in \mathbb{N}$. Then, $M$ extends to a bounded operator on $L^p(\mathbb{R}^{n}, w)$ for every $p \in (2, \infty)$ and $w \in A_{p/2} (\mathbb{R}^n)$ if $N \geq \floor*{n/2} + 1$, and if we increase the number of derivatives in assumption \eqref{cond-Bagchi-Thangavelu-1}, that is, $N \geq n + 1$, then $M$ maps boundedly $L^p(\mathbb{R}^{n}, w)$ to itself for every $p \in (1, \infty)$ and $w \in A_p(\mathbb{R}^n)$. This provides an analogue of the Mihlin-H\"ormander multiplier theorem in the operator valued setting, and as a consequence of this analysis, authors in \cite{BagchiThangaveluHermitePseudo} also obtained weighted estimates for Hermite pseudo-multipliers under suitable smoothness conditions on the symbol function. Hermite pseudo-multipliers were originally introduced in \cite{EppersonHermitePseudo}. Also, we refer the work of L. Weis \cite{Lutz-Weis-R-Boundedness-2001} where more general operator valued Fourier multipliers were studied. In particular, when $n_2=1$, L. Weis studied $L^p$-boundedness results with conditions given in terms of R-boundedness of the operator family $\Theta(\lambda)$.

\subsection{Main results on operator-valued Fourier pseudo-multipliers} \label{subsec:main-results-general-family}
In order to state our results, we need some more notations. For $\nu, \gamma \in \mathbb{N}^{n_1}$ and $\vartheta_1, \vartheta_2 \in \mathbb{N}^{n_2}$, let us denote by $\mathcal{L}(\lambda)_{\vartheta_1, \vartheta_2}^{\nu, \gamma}$ the differential operator given by 
\begin{align} \label{def:derivative-theta-family} 
\mathcal{L}(\lambda)_{\vartheta_1, \vartheta_2}^{\nu, \gamma} \Theta(x'', \lambda) =  |\lambda|^{- \left(\frac{|\nu|+|\gamma|}{2} + |\vartheta_2| \right)} \delta^{\nu}(\lambda) \bar{\delta}^{\gamma} (\lambda) \partial_{\lambda}^{\vartheta_1} \Theta(x'', \lambda), 
\end{align}
where the non-commutative derivatives $\delta_j(\lambda)$ and $\bar{\delta}_j(\lambda)$ are the scaled versions of earlier defined derivatives with $\lambda \neq 0$. More precisely, 
$$ \delta_j (\lambda) M = |\lambda|^{-1/2} [M, A_j(\lambda)] \quad \text{and} \quad \bar{\delta}_j (\lambda) M = |\lambda|^{-1/2} [A_j(\lambda)^*, M], $$ 
with $$A_j (\lambda) = \frac{\partial}{\partial x'_j} + |\lambda| x'_j,\,\, \text{and}\,\, A_j(\lambda)^* = - \frac{\partial}{\partial x'_j} + |\lambda| x'_j.$$ 

We assume that for some $L_0, N_0 \in \mathbb{N}$, operators $ \mathcal{L}(\lambda)_{\vartheta_1, \vartheta_2}^{ \nu, \gamma} x'^{\alpha_0} \partial^{\beta_0}_{x''} \Theta(x'', \lambda) \in \mathcal{B} \left( L^2 \left( \mathbb{R}^{n_1} \right) \right)$ for all $|\nu + \gamma| + |\vartheta_1 + \vartheta_2| + |\alpha_0| \leq L_{0}$ and $|\alpha_0| = |\beta_0| \leq N_0$, and that they satisfy conditions \eqref{def:grushin-kernel-Maucheri-hormander-cond1-L0-N0} and \eqref{def:grushin-kernel-Maucheri-hormander-cond2-L0-N0} which we shall shortly describe. We would like to point out that $x'_l \partial_{x''_k}$ are part of the first order gradient vector field associated to the Grushin operator (see \eqref{first-order-grad} and \eqref{first-order-grad-vector}), and therefore the appearance of $x'^{\alpha_0} \partial^{\beta_0}_{x''}$ (for various choices of $|\alpha_0| = |\beta_0|$) in the derivative assumptions is natural. The exact choice of $L_0$ and $N_0$ will appear in the statement of the main results. With $\lambda \neq 0$, for each $j \in \mathbb{N}$, let us denote by $\chi_{j}(\lambda)$ the scaled Hermite Projection operator defined by 
\begin{align} \label{def:Hermite-projection-operator-chi-lambda}
\chi_{j}(\lambda) = 
\begin{cases}
\sum\limits_{(2 k + n_1) |\lambda| < 1} P_k (\lambda) & \textup{ if } j = 0 \\
\sum\limits_{2^{j-1} \leq (2 k + n_1) |\lambda| < 2^j} P_k (\lambda) & \textup{ if } j \in \mathbb{N} \setminus \{0\}, 
\end{cases}
\end{align}
where $P_k (\lambda)$ stands for the Hermite projection of $L^2 (\mathbb{R}^{n_1}) $ onto the eigenspace of the scaled Hermite operator $H (\lambda) = - \Delta_{x'} + |\lambda|^2 |x'|^2$ with eigenvalues $(2k+n_1) |\lambda|$. Since, $\chi_j (\lambda)$ are finite rank operators on $L^2 (\mathbb{R}^{n_1})$ for any $\lambda \neq 0$, hence they are Hilbert-Schmidt operators. Consequently, for any bounded operator $T$ on $L^2 \left( \mathbb{R}^{n_1}\right)$, $T \chi_{j}(\lambda)$ is Hilbert-Schmidt, and therefore has an integral kernel. Now, for each fixed $x^{\prime} \in \mathbb{R}^{n_1}$ and $j \in \mathbb{N}$, set
\begin{align} \label{def:set-Cj(x-prime)}
\mathfrak{E}_{j} (x^{\prime}) = \left\{ y^{\prime} \in \mathbb{R}^{n_1} : |x'|+|y'| \geq 2^{-j/2} \right\}. 
\end{align} 

We consider the following conditions on kernels of the above mentioned operators: 
\newcommand{\KCKCaa}{
\begin{align} 
& |B(x, 2^{-j/2})| \int_{\mathfrak{E}_{j} (x^{\prime})} \int_{\mathbb{R}^{n_2}} \left| \frac{ \left( \left( \mathcal{L}(\lambda)_{\vartheta_1, \vartheta_2}^{ \nu, \gamma} {x'}^{\alpha_0} \partial^{\beta_0}_{x''} \Theta(x'', \lambda) \right) \chi_{j}(\lambda)\right)(x^{\prime}, y^{\prime})}{(|x^{\prime}|+|y^{\prime}|)^{|\vartheta_1 + \vartheta_2|}} \right|^2 d\lambda \, dy^{\prime} \tag{$KC1$-$L_0$-$N_0$} \label{def:grushin-kernel-Maucheri-hormander-cond1-L0-N0} \\ 
\nonumber & \quad \lesssim_{L_0, N_0} 2^{j |\beta_0|} 2^{-j\left( |\nu + \gamma| + |\vartheta_1 + \vartheta_2| \right)}, \\ 
\nonumber & \textup{for all } |\nu + \gamma| + |\vartheta_1 + \vartheta_2| + |\beta_0| \leq L_0 \, \textup{ and } |\alpha_0| = |\beta_0| \leq N_0, \\ 
& |B(x, 2^{-j/2})| \int_{\mathfrak{E}_{j} (x^{\prime})^c} \int_{\mathbb{R}^{n_2}} \left| \left( \left( \mathcal{L}(\lambda)_{\vartheta_1, \vartheta_2}^{\nu, \gamma} {x'}^{\alpha_0} \partial^{\beta_0}_{x''} \Theta(x'', \lambda) \right)  \chi_{j}(\lambda) \right) (x^{\prime}, y^{\prime}) \right|^2 d\lambda \, dy' \label{def:grushin-kernel-Maucheri-hormander-cond2-L0-N0} \tag{$KC2$-$L_0$-$N_0$} \\ 
\nonumber & \quad \lesssim_{L_0, N_0} 2^{j |\beta_0|} 2^{-j\left( |\nu + \gamma| + 2 |\vartheta_1 + \vartheta_2| \right)}, \\ 
\nonumber & \textup{for all } |\nu + \gamma| + 2 |\vartheta_1 + \vartheta_2| + |\beta_0| \leq L_0 \, \textup{ and } |\alpha_0| = |\beta_0| \leq N_0.
\end{align}
}
\begin{KCKCaa}
\end{KCKCaa}

Before proceeding further, we have the following remark on above conditions, and their usefulness. 

\begin{remark} \label{rem:main-conditions-discussion}
In proving multiplier theorems for Grushin operators, authors of \cite{MartiniSikoraGrushinMRL,MartiniMullerGrushinRevistaMath} made use of the general approach developed in \cite{DuongOuhabazSikoraWeightedPlancherel2002JFA} for spectral multipliers on Homogeneous spaces. One of the most important tools in these works is establishing suitable weighted Plancherel estimates. Now, the control distance function related to the Grushin operator is very complicated, and therefore it is extremely challenging to establish weighted Plancherel theorem for the Grushin multiplier operators via the asymptotic estimates of the control distance. A main obstacle is the increasing singularity near origin as we take $\lambda$-derivatives of the multiplier symbol. This fact can also be seen in the analysis of \cite{Bagchi-Garg-1}. If the operator is a spectral (pseudo) multiplier, this issue can be bypassed using the method of \cite{DuongOuhabazSikoraWeightedPlancherel2002JFA} (as also done in \cite{BBGG-1,BBGG-2}). Since we are dealing with a more general operator, there seems to be no obvious way to avoid the use of precise asymptotics of the control distance function. This indeed forces us to consider only those operators which behave well near $\lambda=0$, and we do so by taking suitable inverse powers of $|\lambda|$ in the definition \eqref{def:derivative-theta-family} of the differential operator $\mathcal{L}(\lambda)_{\vartheta_1, \vartheta_2}^{\nu, \gamma}$. Hence, while we cannot fully recover the results for spectral (pseudo) multipliers from our results, our technique allows us to analyse some interesting class of operators which are not covered in the set-up of spectral (pseudo) multipliers. One such class is of shifted Grushin pseudo-multipliers, given by  \eqref{def:pseudo-mult-op-with-conj-shifts}. 
\end{remark}

Moving on, the asymptotic behaviour of the Grushin control distance is well understood, and can be described in two different regimes of $\mathbb{R}^{n_1 + n_2}$ (see \eqref{def:grushin-control-metric}). With that, the proof of Theorem $1.15$ (based on estimates of Proposition $4.1$, Lemma $4.5$ and Proposition $5.1$) of \cite{Bagchi-Garg-1} motivates us to consider the set of conditions \eqref{def:grushin-kernel-Maucheri-hormander-cond1-L0-N0} and \eqref{def:grushin-kernel-Maucheri-hormander-cond2-L0-N0} with respect to $\mathfrak{E}_{j} (x^{\prime})$ as given in \eqref{def:set-Cj(x-prime)}. It shall become more transparent in Section \ref{sec:mauceri-type-weighted-kernel-estimates} that the formulation of operator $\Theta$ (given by \eqref{def:Theta-operator}) is inspired by the work of \cite{Bagchi-Garg-1}, where the action of the Grushin control distance on integral kernels associated to Grushin (pseudo) multipliers was studied. We now state our main results. Our first result addresses end-point weak type estimates for the operator $\Theta$.
\begin{theorem} \label{thm:Mauceri-kernel-a=0-weak-11}
Let $\Theta$ satisfies conditions \eqref{def:grushin-kernel-Maucheri-hormander-cond1-L0-N0} and \eqref{def:grushin-kernel-Maucheri-hormander-cond2-L0-N0} for $L_0 = Q + 2$ and $N_0 = 0$. Then the operator $\Theta$ given by \eqref{def:Theta-operator} is of weak type $(1,1)$. 
\end{theorem}
We focus on stating sparse domination results at this point.  Let $\mathcal{S}$ denote a family of generalised-dyadic cubes (Christ's cubes) on the homogeneous space $(\mathbb{R}^{n_1+n_2}, d, |\cdot|)$. We say a collection of measurable sets $S \subset \mathcal{S}$ to be an $\eta$-sparse family, for some $0<\eta<1$, if for every member $\mathcal{Q}\in {S}$ there exists a set $E_{\mathcal{Q}} \subseteq \mathcal{Q}$ such that $|E_{Q}|\geq \eta |\mathcal{Q}|$. Corresponding to a sparse family ${S}$ and $1\leq r<\infty$, we define the sparse operator as follows: 
\begin{align} \label{def:Sparse-operator} 
\mathcal{A}_{r, S}f(x) = \sum_{\mathcal{Q}\in S} \left( \frac{1}{|\mathcal{Q}|} \int_{\mathcal{Q}}|f|^r \right)^{1/r} \chi_{\mathcal{Q}}(x).
\end{align} 
We simply write $\mathcal{A}_{S}$ for $\mathcal{A}_{1, S}$.

\begin{theorem} \label{thm:Mauceri-kernel-a=0}
Let $\Theta$ satisfies conditions \eqref{def:grushin-kernel-Maucheri-hormander-cond1-L0-N0} and \eqref{def:grushin-kernel-Maucheri-hormander-cond2-L0-N0} for $L_0 = \floor*{Q/2} + 3$ and $N_0 = 1$. Then, for every compactly supported bounded measurable function $f$ on $\mathbb{R}^{n_1+n_2}$, there exists a sparse family $S\subset \mathcal{S}$ such that
\begin{align*} 
|\Theta f(x)| \lesssim_{T} \mathcal{A}_{2, S}f(x),
\end{align*}
for almost every $x \in \mathbb{R}^{n_1+n_2}$. 
\end{theorem}

Increasing the number of derivatives result in an improved pointwise sparse domination. 
\begin{theorem} \label{thm:Mauceri-kernel-a=0-more-derivative}
Let $\Theta$ satisfies conditions \eqref{def:grushin-kernel-Maucheri-hormander-cond1-L0-N0} and \eqref{def:grushin-kernel-Maucheri-hormander-cond2-L0-N0} for $L_0 = Q + 3$ and $N_0 = 1$. Then, for every compactly supported bounded measurable function $f$ on $\mathbb{R}^{n_1+n_2}$, there exists a sparse family $S\subset \mathcal{S}$ such that 
\begin{align*} 
| \Theta f(x)| \lesssim_{T} \mathcal{A}_{S} f(x),
\end{align*} 
for almost every $x \in \mathbb{R}^{n_1+n_2}$. 
\end{theorem}

Next, we discuss the particular family of operators as given in \eqref{def:pseudo-mult-op-with-conj-shifts}, which also provides us an example of operators of the type $\Theta(x'', \lambda)$. For the same, we need to first define the joint functional calculus of the Grushin operator $G$.

\subsection{Applications to shifted Grushin pseudo-multipliers} \label{subsec:main-results-joint-functional-calculus}

We briefly review the joint functional calculus of $G$, and refer to \cite{MartiniJointFunctionalCalculiMathZ,Bagchi-Garg-1,BBGG-2} for more details. Consider the following family of operators: 
\begin{align} \label{eq:operatorsLandU}
L_{j} = (-i \partial_{x_j'})^2 + {x_j^\prime}^{2} \sum_{k=1}^{n_2} (-i\partial_{x_k''})^2 \quad \textup{and} \quad U_k = - i \partial_{x_k''}, 
\end{align} 
for $j= 1, 2, \ldots, n_1$ and $k = 1, 2, \ldots, n_2$. The operators $L_{j}$ and $U_k$ are essentially self-adjoint on $C_c^\infty(\mathbb{R}^{n_1 + n_2})$ and their spectral resolutions commute. Hence they admit a joint functional calculus on $L^2(\mathbb{R}^{n_1 + n_2})$ in the sense of the spectral theorem. 

Let us write $\boldsymbol{L} = (L_1, L_2,\ldots, L_{n_1})$ and $\boldsymbol{U} = (U_1, U_2, \ldots, U_{n_2})$ and $ \tilde{1} = (1, 1, \ldots, 1) \in \mathbb{R}^{n_1}$. Given a function $m \in L^\infty \left( \mathbb{R}^{n_1 + n_2} \times (\mathbb{R}_+)^{n_1} \times (\mathbb{R}^{n_2} \setminus \{0\}) \right)$, one can (densely) define the multiplier operator $m( \boldsymbol{L}, \boldsymbol{U})$ by 
\begin{align} \label{def:joint-Gru-pseudo}
m(\boldsymbol{L}, \boldsymbol{U}) f(x) : = \int_{\mathbb{R}^{n_2}} e^{-i \lambda \cdot x^{\prime \prime}} \sum_{\mu \in \mathbb{N}^{n_1}} m \left( (2 \mu + \tilde{1}) |\lambda| \right) \left( f^{\lambda}, \Phi^{\lambda}_{\mu} \right) \Phi^{\lambda}_{\mu} (x^\prime) \, d\lambda
\end{align} 
where $\Phi^{\lambda}_{\mu}$ are the scaled Hermite functions. 

We adapt the convention that $\Phi^{\lambda}_{\mu} = 0$ if any of the coordinates $\mu_j$ of $\mu = \left( \mu_1, \ldots, \mu_{n_1} \right) $ is negative. Now, for fixed $\mu_{0}\in \mathbb{N}^{n_1}$, define the following shift operators: 
\begin{align}
\notag\mathfrak{S}_{\mu_{0}} f(x) & = \int_{\mathbb{R}^{n_2}}e^{-i\lambda \cdot x''} \sum_{\mu} \left( f^{\lambda}, \Phi^{\lambda}_{\mu} \right) \Phi^{\lambda}_{\mu + \mu_{o}}(x')\, d\lambda \\
\notag \mathfrak{S}^{*}_{\mu_{0}} f(x) & = \int_{\mathbb{R}^{n_2}} e^{-i\lambda\cdot  x''} \sum_{\mu} \left( f^{\lambda}, \Phi^{\lambda}_{\mu} \right) \Phi^{\lambda}_{\mu - \mu_{o}}(x')\, d\lambda, 
\end{align} 
for $f \in S (\mathbb{R}^{n_1+n_2})$, and given a bounded measurable function $m$ on $\left( \mathbb{R}_+ \right)^{n_1} \times (\mathbb{R}^{n_2}\setminus\{0\})$, consider the conjugation of the multiplier operator $m(\boldsymbol{L},\boldsymbol{U})$ by shift operators: 
\begin{align*} 
T^{\mu_{0}}_{m(\boldsymbol{L},\boldsymbol{U})} f(x) & = \mathfrak{S}_{\mu_{0}} \circ m(\boldsymbol{L},\boldsymbol{U}) \circ \mathfrak{S}^{*}_{\mu_{0}}f(x) \\ 
\nonumber & = \int_{\mathbb{R}^{n_2}} e^{-i \lambda \cdot  x''} \sum_{\mu} m(2 \mu + \tilde{1}) |\lambda|, \lambda)) \left( f^{\lambda}, \Phi^{\lambda}_{\mu + \mu_{0}} \right) \Phi^{\lambda}_{\mu + \mu_{0}}(x')\, d\lambda \\ 
\nonumber & = \int_{\mathbb{R}^{n_2}} e^{-i\lambda \cdot x''} \sum_{\mu} m((2 (\mu - \mu_0) + \tilde{1}) |\lambda|, \lambda) \left( f^{\lambda}, \Phi^{\lambda}_{\mu} \right) \Phi^{\lambda}_{\mu}(x') \, d\lambda.
\end{align*}

While there are Mihlin--H\"ormander-type boundedness results for the operator $m(\boldsymbol{L}, \boldsymbol{U})$ under suitable conditions on $m$ (see, for example, Section 6 of \cite{MartiniJointFunctionalCalculiMathZ}), it is because of the absence of $L^p$-boundedness results (with $p \neq 2$) for the shift operators $\mathfrak{S}_{\mu_{0}}$ and $\mathfrak{S}^{*}_{\mu_{0}}$ that we can not perhaps directly deduce $L^p$-boundedness results for the operators $T^{\mu_{0}}_{m(\boldsymbol{L}, \boldsymbol{U})}$. 

In fact, we are going to consider Grushin pseudo-multiplier operators with shifts of the above type. More precisely, given a bounded measurable function $m$ on $\mathbb{R}^{n_1 + n_2} \times \left( \mathbb{R}_+ \right)^{n_1} \times (\mathbb{R}^{n_2} \setminus\{0\})$, we define  
\begin{align} \label{def-repeat:pseudo-mult-op-with-conj-shifts} 
T^{\mu_{0}}_{m(x, \boldsymbol{L}, \boldsymbol{U})} f(x) & : = \int_{\mathbb{R}^{n_2}} e^{-i\lambda \cdot x''} \sum_{\mu} m(x, (2 (\mu - \mu_0) + \tilde{1}) |\lambda|, \lambda) \left( f^{\lambda}, \Phi^{\lambda}_{\mu} \right) \Phi^{\lambda}_{\mu}(x') \, d\lambda, 
\end{align}
and study analogues of Theorems \ref{thm:Mauceri-kernel-a=0-weak-11}, \ref{thm:Mauceri-kernel-a=0} and \ref{thm:Mauceri-kernel-a=0-more-derivative} for these operators. 

Note that we have 
$$ T^{\mu_{0}}_{m(x, \boldsymbol{L}, \boldsymbol{U})} f(x) = \int_{\mathbb{R}^{n_2}} e^{-i \lambda \cdot x^{\prime \prime}} \left( \Theta(x'', \lambda) f^{\lambda} \right) (x^{\prime}) \, d\lambda, $$ 
where the operator family $\Theta(x'', \lambda)$ is densely defined by 
\begin{align} \label{def:theta-lambda-example-with-shifts} 
\Theta(x'', \lambda) f^{\lambda} (x') = \sum_{\mu} m(x, (2 (\mu - \mu_0) + \tilde{1}) |\lambda|, \lambda) \left( f^{\lambda}, \Phi^{\lambda}_{\mu} \right) \Phi^{\lambda}_{\mu} (x').
\end{align} 
We study these operators under the following natural assumption on the symbol function $m$:
\begin{align} \label{assumption:decay-grushin-joint-symb}
\left|X^\Gamma \partial_{\tau}^{\alpha} \partial_{\kappa}^{\beta} m(x, \tau, \kappa) \right| \leq_{\Gamma, \alpha, \beta} (1 + |\tau| + |\kappa|)^{-  (|\alpha|+ |\beta|) + \frac{|\Gamma|}{2}} 
\end{align} 
for some $\Gamma \in \mathbb{N}^{n_1 + n_1 n_2}$, $\alpha \in \mathbb{N}^{n_1}$ and $\beta \in \mathbb{N}^{n_2}$.  

Together with the decay assumption \eqref{assumption:decay-grushin-joint-symb} on the symbol function $m(x, \tau, \kappa)$ and its gradients, we also require a cancellation condition in $\kappa$-variable of the following type 
\begin{align} \label{def:grushin-symb-vanishing-0-condition}
\lim_{\kappa \to 0} \partial_{\kappa}^{\beta'} m(x, \tau, \kappa) = 0, 
\tag{CancelCond} 
\end{align} 
for appropriate $\beta' \in \mathbb{N}^{n_2}$. Making use of condition \eqref{def:grushin-symb-vanishing-0-condition}, several $L^p$-boundedness results for various classes of pseudo-multipliers associated to joint functional calculus of $G$ are studied in \cite{Bagchi-Garg-1,BBGG-2}. Before moving further, let us write a simple example of a symbol function satisfying assumptions of our interest. 

\begin{example} \label{Example:grushin-symb-vanishing-0-condition}
Consider the following symbol function
\begin{align}
m(x, \tau, \kappa) =\exp (- (1 + |\tau|^2)/|\kappa|^2).
\end{align}
It is straightforward to verify that $m$ satisfies \eqref{assumption:decay-grushin-joint-symb} for all $\Gamma \in \mathbb{N}^{n_1 + n_1 n_2}$, $\alpha \in \mathbb{N}^{n_1}$ and $\beta \in \mathbb{N}^{n_2}$, as well as \eqref{def:grushin-symb-vanishing-0-condition} for all $\beta' \in \mathbb{N}^{n_2}$.
\end{example}

The following remark is also in order. 
\begin{remark} \label{rem:shift-vs-joint-spectral}
One should note that the cancellation condition \eqref{def:grushin-symb-vanishing-0-condition} is not really sufficient to convert the shift type case to the one for the joint spectral one. Namely, if we define $\tilde{m}(x, \tau, \kappa) = m(x, \tau + |\kappa| \mu_0, \kappa)$, then one may ask whether the shift type case gets converted into the spectral pseudo one for $\tilde{m}(x, \boldsymbol{L}, \boldsymbol{U})$ which is covered in our earlier spectral pseudo-multiplier theorem. No, there is an issue, and we explain it below. 

When one takes a derivative in $\kappa_l$-variable, one of the terms that one gets would be of the form $\frac{\kappa_l}{|\kappa|} \partial_{{\tau_j}} \tilde{m}(x, \tau, \kappa)$. Another derivative in $\kappa_l$-variable will then produce a term of the form $|\kappa|^{-1} \partial_{{\tau_j}} \tilde{m}(x, \tau, \kappa)$. Thanks to the assumed cancellation condition, we have no issue of singularity near $\kappa = 0$, yet we may not get the expected rate of decay. To be more precise, while we need  
$$\left|\partial^2_{{\kappa_l}}\tilde{m}(x, \tau, \kappa)\right| \lesssim (1 + |\tau| + |\kappa|)^{-2},$$ 
we are only assured of the following bound 
$$\left|\partial^2_{{\kappa_l}}\tilde{m}(x, , \tau, \kappa)\right| \lesssim (1 + |\kappa|)^{-1} (1 + |\tau| + |\kappa|)^{-1}.$$
\end{remark}

We now state our results for operators $T^{\mu_{0}}_{m(x, \boldsymbol{L}, \boldsymbol{U})}$ as an application of Theorem~\ref{thm:Mauceri-kernel-a=0-weak-11}, Theorem~\ref{thm:Mauceri-kernel-a=0} and Theorem~\ref{thm:Mauceri-kernel-a=0-more-derivative}. 
\begin{theorem} \label{thm:joint-shift-multi-weak-type}
Let $m\in  L^\infty \left(\mathbb{R}^{n_1 + n_2} \times (\mathbb{R}_+)^{n_1} \times (\mathbb{R}^{n_2} \setminus \{0\}) \right)$ be such that it satisfies the condition 
\eqref{assumption:decay-grushin-joint-symb} for all $|\alpha|+ |\beta| \leq Q +2$, and the condition \eqref{def:grushin-symb-vanishing-0-condition} for all $|\beta'| \leq Q + 2$. If the operator $T^{\mu_{0}}_{m(x, \boldsymbol{L}, \boldsymbol{U})}$ is bounded on $L^2 ( \mathbb{R}^{n_1 + n_2})$ then $T^{\mu_{0}}_{m(x, \boldsymbol{L}, \boldsymbol{U})}$ is of weak type $(1, 1)$. 
\end{theorem}

\begin{theorem} \label{thm:joint-shift-multi-sparse-less-derivative} 
Let $m\in  L^\infty \left(\mathbb{R}^{n_1 + n_2} \times (\mathbb{R}_+)^{n_1} \times (\mathbb{R}^{n_2} \setminus \{0\}) \right)$ be such that it satisfies the condition 
\eqref{assumption:decay-grushin-joint-symb} for all $|\alpha|+ |\beta| + |\Gamma| \leq \floor*{Q/2} + 3$, $|\Gamma| \leq 1$, and the condition \eqref{def:grushin-symb-vanishing-0-condition} for all $|\beta'| \leq \floor*{Q/2} + 3$. Assume also that the operator $T^{\mu_{0}}_{m(x, \boldsymbol{L}, \boldsymbol{U})}$ is bounded on $L^2 ( \mathbb{R}^{n_1 + n_2})$. Then for every compactly supported bounded measurable function $f$ there exists a sparse family $S\subset \mathcal{S}$ such that
\begin{align*} 
|T^{\mu_{0}}_{m(x, \boldsymbol{L}, \boldsymbol{U})}f(x)| \lesssim_{T^{\mu_{0}}_{m}} \mathcal{A}_{2, S}f(x),
\end{align*}
for almost every $x \in \mathbb{R}^{n_1+n_2}$.
\end{theorem}

\begin{theorem} \label{thm:joint-shift-multi-sparse-more-derivative}
Let $m\in  L^\infty \left(\mathbb{R}^{n_1 + n_2} \times (\mathbb{R}_+)^{n_1} \times (\mathbb{R}^{n_2} \setminus \{0\}) \right)$ be such that it satisfies the condition 
\eqref{assumption:decay-grushin-joint-symb} for all $|\alpha|+ |\beta| + |\Gamma| \leq Q + 3$, $|\Gamma| \leq 1$, and the condition \eqref{def:grushin-symb-vanishing-0-condition} for all $|\beta'| \leq Q + 3$. Assume also that the operator $T^{\mu_{0}}_{m(x, \boldsymbol{L}, \boldsymbol{U})}$ is bounded on $L^2 ( \mathbb{R}^{n_1 + n_2})$. Then for every compactly supported bounded measurable function $f$ there exists a sparse family $S\subset \mathcal{S}$ such that
\begin{align*} 
|T^{\mu_{0}}_{m(x, \boldsymbol{L}, \boldsymbol{U})}f(x)| \lesssim_{T^{\mu_{0}}_{m}} \mathcal{A}_{S}f(x),
\end{align*}
for almost every $x \in \mathbb{R}^{n_1+n_2}$.
\end{theorem} 

\subsection{Methodology of the proof} 
\label{subsec:intro-methodology-proofs} 
As explained earlier, in this paper our main goal is to study pointwise sparse domination results for a family of operator valued Fourier pseudo-multipliers in the context of the Grushin operator. We shall use the sparse domination technique developed in \cite{Lerner-Ombrosi-pointwaise-sparse2020, Lorist-pointwaise-sparse2021} to prove our results. In \cite{BBGG-1}, using the technique developed in \cite{Lerner-Ombrosi-pointwaise-sparse2020, Lorist-pointwaise-sparse2021}, we studied sparse domination results for pseudo-multipliers associated to Grushin operator $G$. To prove sparse domination results in \cite{BBGG-1}, we used the weighted Plancherel estimates of the associated kernels of the pseudo-multipliers. But for the case of family of operator valued Fourier pseudo-multipliers the weighted estimates of the associated kernels do not follow from the known spectral theory. Instead, we build it with a very detailed use of the pointwise action of the Grushin control distance (for the operator $G$) developed in \cite{Bagchi-Garg-1}. Once we establish these estimates in Subsections \ref{subsec:mauceri-type-L2-weighted-kernel-estimates} and \ref{subsec:mauceri-type-L-infty-weighted-kernel-estimates}, Theorems \ref{thm:Mauceri-kernel-a=0-weak-11}, \ref{thm:Mauceri-kernel-a=0} and \ref{thm:Mauceri-kernel-a=0-more-derivative} follow from Theorems \ref{thm:weak-type-bound-operator}, \ref{thm:main-sparse} and \ref{thm:main-sparse-more-derivative} of Subsection \ref{subsec:results-for-Grsuhin-pseudo} respectively. We would however like to point out that our assumptions on number of derivatives for $\Theta(x'', \lambda)$ family in these theorems are $Q + 2$, $\floor*{Q/2} + 3$ or $Q + 3$, meaning that we need one/two extra derivatives in comparison to what we need for results of \cite{BBGG-1}. This is because our analysis depends on the kernel estimates of \cite{Bagchi-Garg-1}, and therefore we have to necessarily consider the action of the control distance function as well as its powers in the multiples of $4$ only, which leads to the mentioned extra derivative requirement.

We discuss some preliminary results related to our work in Section \ref{sec:prelim}. In Section \ref{sec:mauceri-type-weighted-kernel-estimates}, we study the weighted estimates of the kernels. Using the weighted estimates of kernel we shall prove our main results Theorems \ref{thm:Mauceri-kernel-a=0-weak-11}, \ref{thm:Mauceri-kernel-a=0} and \ref{thm:Mauceri-kernel-a=0-more-derivative}. Finally, as examples of such family of operators, we discuss the ones given by \eqref{def-repeat:pseudo-mult-op-with-conj-shifts}, duly noting that these operators do not fall into the category of the spectral pseudo-multipliers studied in \cite{BBGG-1}, and give the proof of Theorems  \ref{thm:joint-shift-multi-weak-type}, \ref{thm:joint-shift-multi-sparse-less-derivative} and \ref{thm:joint-shift-multi-sparse-more-derivative} in Section \ref{Sec:Shifted-Grushin-pseudo-multipliers}.

\subsection{Notations and parameters} \label{subsec:notations} 
We shall use the following list of notations and parameters in this paper.
\begin{itemize} 

\item $\mathcal{B} \left( Y \right) $   denotes the space of all bounded linear operators on the Banach space $Y$ and $\| T \|_{op}$ denotes the operator norm of $T \in \mathcal{B} \left( Y \right)$. 

\item $\mathbb{R}_+ = [0, \infty)$, $\mathbb{N} = \{0, 1, 2, 3, \ldots\}$ and $\mathbb{N}_+ = \{ 1, 2, 3, \ldots\}$. 

\item We write $x = (x^{\prime}, x^{\prime \prime}) \in \mathbb{R}^{n_1} \times \mathbb{R}^{n_2} \, = \mathbb{R}^{n_1 + n_2}$. 

\item We  write $|\mu| = \sum_{j=1}^{n_1}\mu_j$ for $\mu = (\mu_1, \ldots,\mu_{n_1}) \in \mathbb{N}^{n_1}$. Whenever $c = (c_1, \ldots, c_{n_1}) \in \mathbb{R}^{n_1}$, we write $|c| = \left( \sum_{j=1}^{n_1} |c_j|^2 \right)^{1/2}$ and $|c|_1 = \sum_{j=1}^{n_1} |c_j|$.

\item  $ \mathfrak{r}, R_0, \eta$ etc always represent elements of $\mathbb{R}_+$. 

\item  $ j, j^{\prime}, q, l, N, N_j, L, L_j$ etc always represent elements of $\mathbb{N}$. 

\item $\alpha, \alpha_j, \gamma, \gamma_j, \mu, \tilde{\mu}, \tilde{\tilde{\mu}}, \nu, \tilde{\nu}, \tilde{\tilde{\nu}}, \nu_1, \nu_2, \theta, \theta_j$ etc always represent elements of  $\mathbb{N}^{n_1}$. 

\item $\beta, \beta', \beta_j, \vartheta, \vartheta_j, \tilde{\vartheta}$ etc always represents elements of $\mathbb{N}^{n_2}$. 

\item $\Gamma, \Gamma_j $ etc always represent elements of $\mathbb{N}^{n_0}$, where $n_0 = n_1 + n_1 n_2$.

\item $ \vec{c}, \vec{c}$ and $\tilde{c}$ etc represent elements of $\mathbb{R}^{n_1}$. 

\item $\tau$ corresponds to an element of $ \left( \mathbb{R}_+ \right)^{n_1}$ and  $\lambda, \kappa $ etc correspond to elements of $\mathbb{R}^{n_2}$. We write
$\partial_{\tau}^{\alpha}$, $\partial_\lambda^{\vartheta}$ and $\partial_\kappa^{\beta}$ for partial differential operators.  

\item $\vec{c}(s)$ denotes a vector $(c_1(s), \ldots, c_{n_1}(s))$ where each $c_j(s)$ is a linear function on $[0,1]^{N}$, for some positive integer $N$. 

\item $\Omega$ stands for a compact subset of $\mathbb{R}^2$ which may differ at various places. 

\item $\vec{c}(w)$ denotes a real-valued linear function on $[0,1]^{N_1} \times \Omega^{N_2} \times [0,1]^{N_3}$ for some non-negative integers $N_1$,$N_2$ and $N_3$.

\item $g(w)$'s are continuous functions on compact sets.

\item For a multi-index $\gamma = (\gamma^{(j)})_{j=1}^{n}$, we write $\gamma! =  \prod_{j = 1}^{n} \gamma^{(j)}!$ and $\tau^{\frac{1}{2} \gamma} = \prod_{j = 1}^{n} \tau_j^{\frac{1}{2} \gamma^{(j)}}$.

\item For $A, B >0$, by the expression $A \lesssim B$ we mean that there exists a constant $C>0$ such that $A \leq C B$. We write $A \lesssim_{\epsilon} B$, whenever the implicit constant $C$ may depend on $\epsilon$. We write $A \sim B$, when $A \lesssim B$ and $B \lesssim A$ both true together.
\end{itemize}  

\section{Preliminaries} \label{sec:prelim}

In this section, we talk over some definitions and  basic results which will be essential to discuss the proofs of main results in the next subsections.
$G$ is a hypoelliptic operator, see \cite{Grushin70}, and the control distance $\tilde{d} (x,y)$ associated with the Grushin operator $G$ is given by 
\begin{align} \label{def:grushin-control-metric}
\tilde{d} (x,y) = \sup_{\Psi \in \mathcal{F}} \left| \Psi(x) - \Psi(y) \right|, 
\end{align}
where $\mathcal{F} = \left\{ \Psi \in W^{1, \infty} (\mathbb{R}^{n_1 + n_2}) : \sum\limits_{1 \leq j \leq n_1} \left| X_{j} \Psi \right|^2 + \sum\limits_{1 \leq j \leq n_1} \sum\limits_{1 \leq k \leq n_2} \left| X_{j, k} \Psi \right|^2 \leq 1 \right\}$ where $X_{j}$ and $X_{j, k}$ are the first order gradient vector field associated to the Grushin operator $G$ and are defined, for $1 \leq j \leq n_1$, $1 \leq k \leq n_2$, by
\begin{align} \label{first-order-grad}
X_{j} = \frac{\partial}{\partial x_j^{\prime}} \quad \textup{and} \quad X_{j, k} = x^{\prime}_{j} \frac{\partial}{\partial x_k^{\prime \prime}}.
\end{align}

It is known that $\tilde{d}$ asymptotically equivalent to a quasi metric $d$, where $d$ is given by  
\begin{align} \label{def:grushin-metric-asymp}
 d(x,y) := \left|x^{\prime} - y^\prime \right| + 
\begin{cases}
\frac{\left|x^{\prime \prime} - y^{\prime \prime}\right|}{\left|x^{\prime} \right| + \left|y^\prime \right|} &\textup{ if } \left|x^{\prime \prime} - y^{\prime \prime}\right|^{1/2} \leq \left|x^{\prime} \right| + \left|y^\prime \right| \\
\left|x^{\prime \prime} - y^{\prime \prime}\right|^{1/2} &\textup{ if } \left|x^{\prime \prime} - y^{\prime \prime}\right|^{1/2} \geq \left|x^{\prime} \right| + \left|y^\prime \right|. 
\end{cases}
\end{align}

From now on, by abuse of notation, we take $d(x,y)$ of \eqref{def:grushin-metric-asymp} as the Grushin metric. Let us denote by $X$ the first order gradient vector field 
\begin{align} \label{first-order-grad-vector}
X := (X_j, X_{j,k})_{1 \leq j \leq n_1, \, 1 \leq k \leq n_2}. 
\end{align}
To know more about control distance associated to Grushin operator, see the work \cite{RobinsonSikoraDegenerateEllipticOperatorsGrushinTypeMathZ2008}.

As mentioned in the introduction, our analysis of kernels of operators $\Theta$ corresponding to conditions \eqref{def:grushin-kernel-Maucheri-hormander-cond1-L0-N0} and \eqref{def:grushin-kernel-Maucheri-hormander-cond2-L0-N0} is based on the pointwise kernel estimates developed in \cite{Bagchi-Garg-1}. In particular, we shall make a frequent use of the two estimates from \cite{Bagchi-Garg-1}, but before stating those estimates, we have the following important remark. 

\begin{remark} \label{rem:convention-symbol-support} 
While making kernel estimates, involved symbol functions $m (\tau)$ (respectively $m (\tau, \kappa)$) are always compactly supported in $\left( \mathbb{R}_+ \right)^{n_1}$ (respectively in $\left( \mathbb{R}_+ \right)^{n_1} \times \left( \mathbb{R}^{n_2} \setminus \{0\} \right)$). This is to ensure the validity of various pointwise infinite sums in the discrete variable (in $\mathbb{R}^{n_1}$) and the integrability in continuous variable (in $\mathbb{R}^{n_2}$). However, the constants or the number of terms appearing in any of the claimed finite linear combination do not depend on the assumed compact support. These number of terms in-general depend on the power of $(x' - y')$ and $(x'' - y'')$ and the ambient dimensions $n_1$ and $n_2$. 

As usual, we can extend these symbol functions to $ \mathbb{R}^{n_1}$ (respectively on $\mathbb{R}^{n_1} \times \left( \mathbb{R}^{n_2} \setminus \{0\} \right)$) by defining them to be $0$ outside their support. 
\end{remark}

\begin{lemma}[Lemma 4.5, \cite{ Bagchi-Garg-1}] \label{first-layer-lem}
Let $\tilde{c} \in \mathbb{R}^{n_1}$ be fixed. For any $\alpha \in \mathbb{N}^{n_1}$, we can express 
$$ \left(x^{\prime} - y^{\prime} \right)^{\alpha} \sum_{\mu} m((2\mu + \tilde{c})|\lambda|) \Phi_{\mu}^{\lambda}(x^{\prime}) \Phi_{\mu}^{\lambda}(y^{\prime})$$
is a finite linear combination of terms of the form 
\begin{align*}
\int_{[0,1]^{|\alpha|}} \sum_{\mu} C_{\mu, \tilde{c}, \vec{c}} \left(\tau^{\frac{1}{2} \gamma_2} \partial_{\tau}^{ \gamma_1 } m\right) ((2\mu  + \tilde{c} + \vec{c}(s))|\lambda|) \Phi_{\mu + \tilde{\mu}}^{\lambda}(x^{\prime}) \Phi_{\mu}^{\lambda}(y^{\prime}) \, ds, 
\end{align*} 
where $|\gamma_2| \leq |\gamma_1| \leq |\alpha|$, $|\gamma_1| - \frac{1}{2} |\gamma_2| = \frac{|\alpha|}{2}$, $|\tilde{\mu}| \leq |\alpha|$, and $C_{\mu, \tilde{c}, \vec{c}}$ is a bounded function of $\mu$, $\tilde{c}$ and $\vec{c}$. 
\end{lemma}

Also we require the following kernel estimate which follows by combining Lemma $4.7$ and Remark $4.8$ of \cite{Bagchi-Garg-1}. 
\begin{lemma}[\cite{Bagchi-Garg-1}] 
\label{weighted-kernel-estimate-3} 
Let $\tilde{c} \in \mathbb{R}^{n_1}$ be fixed and $m$ is a compactly supported smooth function on $(\mathbb{R}_+)^{n_1} \times (\mathbb{R}^{n_2} \setminus \{0\})$. For any $L_1, L_2\in \mathbb{N}$, we can write 
$$ \left|x^\prime - y^\prime \right|^{4 L_1} \left|x^{\prime \prime} - y^{\prime \prime} \right|^{2 L_2} \int_{\mathbb{R}^{n_2}} \sum_{\mu} m((2\mu + \tilde{c})|\lambda|, \lambda) E_{\mu, \lambda} (x, y) \, d \lambda $$ 
as a finite linear combination of terms of the form 
\begin{align*}
& {x^{\prime}}^{\alpha_{1}}\int_{[0,1]^{N_1} \times \Omega^{N_2} \times [0,1]^{4 L_1}} \int_{\mathbb{R}^{n_2}} \mathfrak{A}_l(\lambda) \sum_{\mu} C_{\mu, \tilde{c}, \vec{c}} \left(\tau^{\frac{1}{2} \theta_1} \partial_\tau^{\theta_2}\partial_{\kappa}^{\beta} m\right) ((2\mu + \tilde{c} + \vec{c}(\omega)) |\lambda|,\lambda) \\ 
\nonumber & \quad \quad \quad \quad \quad \quad \quad \quad \quad \quad \quad \quad \quad \quad \quad \quad \quad \Phi_{\mu + \tilde{\mu}}^{\lambda}(x^{\prime}) \Phi_{\mu}^{\lambda}(y^{\prime}) e^{i \lambda \cdot (x^{\prime \prime} - y^{\prime \prime})} \, g(\omega) \, d \lambda \, d\omega, 
\end{align*}
where $N_1, N_2 \leq 2 L_2 - (|\beta| + l)$, $| \alpha_1 | \leq 2 L_2 - (|\beta|+l)$, $|\mathcal{\theta}_1| \leq |\mathcal{\theta}_2| \leq 4 L_1+ 4 L_2 - 2  (|\beta|+l)$, $|\mathcal{\theta}_2| - \frac{|\mathcal{\theta}_1|}{2} = 2 L_1 + 2 L_2 -  (|\beta|+l)- \frac{|\alpha_1|}{2} \geq 2 L_1 + L_2 - \frac{1}{2} (|\beta|+l)$, $C_{\mu, \tilde{c}, \vec{c}}$ is a bounded function of $\mu$ and $\vec{c}$, and $\mathfrak{A}_l$ is a continuous function on $\mathbb{R}^{n_2} \setminus \{0\}$ which is homogeneous of degree $-l$.
\end{lemma} 

\subsection{Kernel estimates and pointwise sparse domination}
\label{subsec:results-for-Grsuhin-pseudo}

In this subsection we will state sparse domination results from \cite{BBGG-1} for general operators which satisfy various kernel estimates. 

For an operator $T \in \mathcal{B} \left( L^2({\mathbb{R}^{n_1+n_2}}) \right)$, and each $j \in \mathbb{N}$, we define 
\begin{align} \label{def:operator-countable-break}
T_{j} = T \circ \chi_{j} (G)
\end{align} 
where $\chi_{j} (G)$ is the Grushin multipliers corresponding to Hermite projection operators $\chi_{j} (\lambda)$ defined in \eqref{def:Hermite-projection-operator-chi-lambda}. Then we can decompose the operator $T = \sum_{j} T_j$,
with the convergence taken in the strong operator topology norm on $L^2({\mathbb{R}^{n_1+n_2}})$. It is clear that  $\chi_{j} (G)$ and $T_{j} = T \circ \chi_{j} (G)$ is a Hilbert-Schmidt operator on $L^2({\mathbb{R}^{n_1+n_2}})$. This implies each $T_j$ has an integral kernel and we denote it by $T_j(x,y)$.

\medskip
\noindent \textbf{\underline{$L^2$-conditions on the kernel}:} There exists some $R_{0} \in (0, \infty)$ such that for all $ \mathfrak{r} \in [0, R_{0}]$ and for every positive real number $\mathcal{K}_0$, we have 
\begin{align} 
\sup_{x\in \mathbb{R}^{n_1+n_2}} |B(x, 2^{-j/2})| \int_{\mathbb{R}^{n_1+n_2}} d(x,y)^{2\mathfrak{r}} |T_{j}(x,y)|^2 \, dy & \lesssim_{R_0} 2^{-j\mathfrak{r}}, \label{cond:General-hypo} \\ 
\sup_{x\in \mathbb{R}^{n_1+n_2}} |B(x, 2^{-j/2})| \int_{d(x,y)<\mathcal{K}_0} d(x,y)^{2\mathfrak{r}} |X_{x} T_{j}(x,y)|^2 \, dy & \lesssim_{R_0, \mathcal{K}_0} 2^{-j \mathfrak{r}} 2^{j} \label{cond:General-hypo-grad}. 
\end{align}

\noindent \textbf{\underline{$L^{\infty}$-conditions on the kernel}:} There exists some $R_{0} \in (0, \infty)$ such that for all $\mathfrak{r} \in [0, R_{0}]$ and for every positive real number $\mathcal{K}_0$, we have
\begin{align} 
\sup_{x\in \mathbb{R}^{n_1+n_2}} \sup_{y \in \mathbb{R}^{n_1+n_2}} |B(x, 2^{-j/2})|^{1/2} |B(y, 2^{-j/2})|^{1/2} d(x,y)^{\mathfrak{r}} |T_{j}(x,y)| & \lesssim_{R_0} 2^{-j\mathfrak{r}/2}, \label{cond:General-hypo-sup} \\ 
\sup_{x \in \mathbb{R}^{n_1+n_2}} \sup_{y \in \mathbb{R}^{n_1+n_2}} |B(x, 2^{-j/2})|^{1/2} |B(y, 2^{-j/2})|^{1/2} d(x,y)^{\mathfrak{r}} |X_{y}T_{j}(x,y)| & \lesssim_{R_0} 2^{-j\mathfrak{r}/2} 2^{j/2}. \label{cond:General-hypo-y-grad-sup} \\ 
\sup_{x \in \mathbb{R}^{n_1+n_2}} \sup_{d(x,y)<\mathcal{K}_0} |B(x, 2^{-j/2})|^{1/2} |B(y, 2^{-j/2})|^{1/2} d(x,y)^{\mathfrak{r}} |X_{x}T_{j}(x,y)| & \lesssim_{R_0, \mathcal{K}_0} 2^{-j\mathfrak{r}/2} 2^{j/2}, \label{cond:General-hypo-grad-sup} 
\end{align}
Using the above conditions on the kernel, in our earlier work \cite{BBGG-1} we have proved the following general sparse domination principles which will be crucial to our purpose here.
\begin{theorem}[Theorem 3.7, \cite{BBGG-1}] \label{thm:weak-type-bound-operator}
Let $T \in \mathcal{B} \left( L^2({\mathbb{R}^{n_1+n_2}}) \right)$ and $T_j$ be as defined by \eqref{def:operator-countable-break}. Suppose that the integral kernels $T_j(x,y)$ satisfy condition \eqref{cond:General-hypo-sup} and \eqref{cond:General-hypo-y-grad-sup} for some $R_0 \geq Q + \frac{1}{2}$, then $T$ is weak type $(1,1)$. 
\end{theorem}

\begin{theorem}[Theorem 3.2, \cite{BBGG-1}] \label{thm:main-sparse}
Let $T \in \mathcal{B} \left( L^2({\mathbb{R}^{n_1+n_2}}) \right)$ and $T_j$ be as defined by \eqref{def:operator-countable-break}. Suppose that the integral kernels $T_j(x,y)$ satisfy conditions \eqref{cond:General-hypo} and \eqref{cond:General-hypo-grad} for some $R_0 > Q/2$. Then for every compactly supported bounded measurable function $f$ there exists a sparse family $S\subset \mathcal{S}$ such that
\begin{align} \label{sparseHormander1-general} 
|Tf(x)| \lesssim_{T} \mathcal{A}_{2, S} f(x),
\end{align} 
for almost every $x \in \mathbb{R}^{n_1 + n_2}$. 
\end{theorem}

\begin{theorem}[Theorem 3.3, \cite{BBGG-1}] \label{thm:main-sparse-more-derivative}
Let $T \in \mathcal{B} \left( L^2({\mathbb{R}^{n_1+n_2}}) \right)$ and $T_j$ be as defined by \eqref{def:operator-countable-break}. Suppose that the integral kernels $T_j(x,y)$ satisfy conditions \eqref{cond:General-hypo-sup} and \eqref{cond:General-hypo-y-grad-sup} for some $R_0 \geq Q + \frac{1}{2}$, and condition \eqref{cond:General-hypo-grad-sup} for some $R_0 > Q$. Then for every compactly supported bounded measurable function $f$ there exists a sparse family $S\subset \mathcal{S}$ such that
\begin{align}
\label{sparseHormander2-general}
|Tf(x)| \lesssim_{T} \mathcal{A}_{S} f(x),
\end{align}
for almost every $x \in \mathbb{R}^{n_1+n_2}$. 
\end{theorem}

\begin{remark}
\label{choice:partition-of-unity}
It is easy to see that the above results hold true if we choose any arbitrary smooth partition of unity $\{ \psi_j \}_{j \in \mathbb{N}}$ of $(0, \infty)$ and decompose operator $T$ into a countable sum $\widetilde{T}_j = T \circ \psi_j (G)$, and assume all the conditions \eqref{cond:General-hypo}--\eqref{cond:General-hypo-grad-sup} for $\widetilde{T}_j$.
\end{remark}

\section{Weighted kernel estimates and proofs of results of Subsection \ref{subsec:main-results-general-family}} \label{sec:mauceri-type-weighted-kernel-estimates}

Choose $\psi_0 \in C_c^\infty((-2,2))$ and $\psi_1 \in C_c^\infty((1/2,2))$ such that $0 \leq \psi_0 (\eta), \psi_1 (\eta) \leq 1$, and 
\begin{align} \label{def:spectral-break} 
\sum_{j=0}^\infty \psi_j(\eta) = 1 \quad \text{and} \quad \sum_{j=-\infty}^\infty \psi_1(2^j \eta) = 1, 
\end{align} 
for all $\eta \geq 0$, where $\psi_j (\eta) = \psi_1 \left( 2^{-(j-1)} \eta \right)$ for $j \geq 2.$  

We shall also make use of the following notation: for each $q \in \mathbb{N}$, we define $\Upsilon^q$ on $\mathbb{R}^{n_2} \setminus \{0\}$ by 
\begin{align} \label{def:spectral-break-kappa-variable} 
\Upsilon^{q}(\kappa) = \sum_{j = -\infty}^q \psi_1(2^j |\kappa|). 
\end{align} 

Given $\Theta \in \mathcal{B} \left( L^2({\mathbb{R}^{n_1+n_2}}) \right)$, we break it into a countable sum of operators as follows. For each $j \in \mathbb{N}$, we define 
\begin{align} \label{def:operator-countable-break-smooth}
\Theta_j = \Theta \circ S_j 
\end{align} 
with $S_{j}$ being the Grushin multiplier operator $S_{j} = \psi_{j} (G).$ 

We first establish the following basic technical result showing that the derivatives of operators $S_j(\lambda)$ would produce finitely many shifts with the number of such shifts depending only on the order of the derivative and the ambient space dimension. Let us denote by $D_j$ the commutator on $\mathbb{R}^{n_1}$ by the multiplication operator $x'_j$, that is, $D_j (M) = [x_j, M]$ for any operator $M$ on functions on $\mathbb{R}^{n_1}$. Recall also that $\chi_{j}(\lambda)$ denotes the (scaled) Hermite Projection operator $\sum_{2^j \leq (2k+ n_1)|\lambda| < 2^{j+1}} P_k (\lambda)$. 
\begin{lemma} \label{lem:finite-shifts}
There exists a constant $C>0$ (depending only on $n_1$ and $n_2$) such that for any $\nu \in \mathbb{N}^{n_1}$ and $\vartheta \in \mathbb{N}^{n_2}$, 
\begin{align*}
\chi_{j'}(\lambda) \left( D^{\nu} \partial_{\lambda}^{\vartheta} S_{j}(\lambda) \right) \equiv 0 
\end{align*}
whenever $|j'-j| > C \left( |\nu|+ 2 |\vartheta| \right)$.
\end{lemma}
\begin{proof}
Note that operators $D^{\nu}$ and $\partial_{\lambda}^{\vartheta}$ commute. So, let us first examine $D^{\nu} S_{j}(\lambda)$. It follows from Lemma \ref{first-layer-lem} that 
$$\left(z'-y'\right)^{\nu} \sum_{\mu} \psi_j((2|\mu|+n_1)|\lambda|) \Phi^{\lambda}_{\mu}(z') \Phi^{\lambda}_{\mu} (y')$$ 
is finite linear combination of 
\begin{align} \label{eq:finite-shifts-1}
\int_{[0,1]^{|\nu|}} \sum_{\mu} C_{\mu, \vec{c}} \left(\tau^{ \frac{1}{2} \gamma_2} \partial_{\tau}^{\gamma_1} \psi_j\right) \left( \left(2\mu  + \tilde{1} + \vec{c}(s) \right) |\lambda| \right) \Phi_{\mu + \tilde{\mu}}^{\lambda} (z^{\prime}) \Phi_{\mu}^{\lambda}(y^{\prime}) \, ds, 
\end{align} 
where $|\gamma_2| \leq |\gamma_1| \leq |\nu|$, $|\gamma_1| - \frac{1}{2} |\gamma_2| = \frac{|\nu|}{2}$, $|\tilde{\mu}| \leq |\nu|$, and $C_{\mu, \vec{c}}$ is bounded in $\mu$ and $\vec{c}$. 

Now, recall that 
\begin{align*} 
\frac{\partial}{\partial |\lambda|} \Phi_{\mu}^{\lambda} & = \frac{1}{4|\lambda|} \sum_{j=1}^{n_1} \left\{-\sqrt{(\mu_j + 1)(\mu_j + 2)} \Phi_{\mu + 2e_j}^{\lambda} +  \sqrt{\mu_j(\mu_j - 1)} \Phi_{\mu - 2e_j}^{\lambda}\right\},
\end{align*}
and therefore for a nice function $m(\tau)$ we have 
\begin{align} \label{lambda-der}
& \frac{\partial}{\partial |\lambda|} \left\{ \sum_{\mu} m \left( \left(2\mu  + \tilde{1} + \vec{c}(s) \right) |\lambda| \right) \Phi_{\mu + \tilde{\mu}}^{\lambda} (z^{\prime}) \Phi_{\mu}^{\lambda}(y^{\prime}) \right\} \\ 
\nonumber & \quad = \sum_{\mu} \sum_{j=1}^{n_1} (2 \mu_j + 1 + c_j(s) ) \left(\partial_{\tau_j} m\right) \left( \left(2\mu  + \tilde{1} + \vec{c}(s) \right) |\lambda| \right) \Phi_{\mu + \tilde{\mu}}^{\lambda} (z^{\prime}) \Phi_{\mu}^{\lambda}(y^{\prime}) \\
\nonumber & \quad \quad - \frac{1}{4|\lambda|} \sum_{\mu} \sum_{j=1}^{n_1} \sqrt{(\mu_j + \tilde{\mu}_j + 1)(\mu_j + \tilde{\mu}_j + 2)} \, m \left( \left(2\mu  + \tilde{1} + \vec{c}(s) \right) |\lambda| \right) \Phi_{\mu + \tilde{\mu} + 2 e_j}^{\lambda} (z^{\prime}) \Phi_{\mu}^{\lambda}(y^{\prime}) \\ 
\nonumber & \quad \quad +  \frac{1}{4|\lambda|}  \sum_{\mu} \sum_{j=1}^{n_1} \sqrt{(\mu_j + \tilde{\mu}_j) (\mu_j + \tilde{\mu}_j - 1)} \, m \left( \left(2\mu  + \tilde{1} + \vec{c}(s) \right) |\lambda| \right) \Phi_{\mu + \tilde{\mu} - 2 e_j}^{\lambda} (z^{\prime}) \Phi_{\mu}^{\lambda}(y^{\prime}) \\
\nonumber & \quad \quad -  \frac{1}{4|\lambda|} \sum_{\mu} \sum_{j=1}^{n_1} \sqrt{(\mu_j + 1)(\mu_j + 2)} \, m \left( \left(2\mu  + \tilde{1} + \vec{c}(s) \right) |\lambda| \right) \Phi_{\mu + \tilde{\mu}}^{\lambda} (z^{\prime}) \Phi_{\mu + 2 e_j}^{\lambda}(y^{\prime}) \\
\nonumber & \quad \quad +  \frac{1}{4|\lambda|} \sum_{\mu} \sum_{j=1}^{n_1} \sqrt{\mu_j(\mu_j - 1)} \, m \left( \left(2\mu  + \tilde{1} + \vec{c}(s) \right) |\lambda| \right) \Phi_{\mu + \tilde{\mu}}^{\lambda} (z^{\prime}) \Phi_{\mu - 2 e_j}^{\lambda}(y^{\prime}).
\end{align}
Notice from \eqref{lambda-der} that each application of $\partial_{\lambda}$ introduces a shift of at the most two places in $\mu$-index of $\Phi_{\mu}^{\lambda}$. So, if we apply $\partial_{\lambda}^{\vartheta}$ then there would be at the most $2 |\vartheta|$ number of shifts in $\mu$-index. Applying this on $C_{\mu, \vec{c}} \left(\tau^{ \frac{1}{2} \gamma_2} \partial_{\tau}^{\gamma_1} \psi_j\right) \left( \left(2\mu  + \tilde{1} + \vec{c}(s) \right) |\lambda| \right)$ of \eqref{eq:finite-shifts-1}, we see that the kernel of $D^{\nu} \partial_{\lambda}^{\vartheta} S_{j}(\lambda)$ is a finite linear combination of terms of the form 
\begin{align} \label{eq:finite-shifts-2}
\int_{[0,1]^{|\nu|}} \sum_{\mu} \tilde{C}_{\mu, \vec{c}} \left(\tau^{ \frac{1}{2} \theta_2} \partial_{\tau}^{\theta_1} \psi_j \right) \left( \left(2 \mu  + \tilde{1} + \vec{c}(s) \right) |\lambda| \right) \Phi_{\mu + \tilde{\mu}}^{\lambda} (z^{\prime}) \Phi_{\mu + \tilde{\tilde{\mu}}}^{\lambda}(y^{\prime}) \, ds, 
\end{align} 
with $|\tilde{\mu}|, |\tilde{\tilde{\mu}}| \leq |\nu| + 2 |\vartheta|$. 

Since $\chi_{j'}$ is supported in $[2^{j'}, 2^{j'+1})$, the conclusion of the lemma follows from \eqref{eq:finite-shifts-2}. 
\end{proof}

We now prove several kernel estimates explaining how conditions of the type \eqref{def:grushin-kernel-Maucheri-hormander-cond1-L0-N0} and \eqref{def:grushin-kernel-Maucheri-hormander-cond2-L0-N0} imply conditions \eqref{cond:General-hypo} to \eqref{cond:General-hypo-grad-sup} in different ranges. 

\subsection{\texorpdfstring{$L^2$}{}-weighted Plancherel (gradient) estimates} \label{subsec:mauceri-type-L2-weighted-kernel-estimates} Here, we shall discuss $L^2$-weighted Plancherel estimates arising out of conditions of the type \eqref{def:grushin-kernel-Maucheri-hormander-cond1-L0-N0} and \eqref{def:grushin-kernel-Maucheri-hormander-cond2-L0-N0}. 

\begin{lemma} \label{lem:General-Kernel-estimate} 
Suppose $\Theta$ satisfies the conditions \eqref{def:grushin-kernel-Maucheri-hormander-cond1-L0-N0} and \eqref{def:grushin-kernel-Maucheri-hormander-cond2-L0-N0} for some $L_0 \in \mathbb{N}$ and $N_0 = 0$. Then for all $L \leq \frac{L_0}{2}$, we have 
\begin{align} \label{General-square-estimate}
\sup_{x \in \mathbb{R}^{n_1 + n_2}} |B(x, 2^{-j/2})| \int_{\mathbb{R}^{n_1 + n_2}} d(x,y)^{4L} |\Theta_{j}(x,y)|^2 \, dy \lesssim_{L_0} 2^{-2j L}.
\end{align}
\end{lemma}

In order to prove Lemma \ref{lem:General-Kernel-estimate}, we  introduce another approximate identity, this time in $\lambda$-variable, to handle (possible) singularities as $\lambda \to 0$. This is important as we would see that our proof requires an integration by parts in $\lambda$-variable, and such an approximate identity helps in handling the boundary terms. For each $q \in \mathbb{N}$, let us define 
\begin{align} \label{def:Theta-family-lambda-cut}
\Theta_{j}^{q}(x'', \lambda) = \Theta_{j}(x'', \lambda) \Upsilon^{q}(\lambda), 
\end{align} 
where $\Upsilon^{q}(\lambda)$ are defined in \eqref{def:spectral-break-kappa-variable}. 

Let  $\Theta_{j}^q$ denote the operator on $L^2 (\mathbb{R}^{n_1 + n_2})$ corresponding to the family  $\left\{ \Theta_{j}^q(x'', \lambda) \right\}$.

\begin{lemma} \label{lem:General-Kernel-estimate-lambda-singularity}
Suppose $\Theta$ satisfies the conditions \eqref{def:grushin-kernel-Maucheri-hormander-cond1-L0-N0} and  \eqref{def:grushin-kernel-Maucheri-hormander-cond2-L0-N0} for some $L_0 \in \mathbb{N}$ and $N_0 = 0$. Then for all $L \leq \frac{L_0}{2}$ and $q \in \mathbb{N}$, we have 
\begin{align} \label{General-square-estimate-singularity-with-q}
\sup_{x\in \mathbb{R}^{n_1 + n_2}} |B(x, 2^{-j/2})| \int_{\mathbb{R}^{n_1 + n_2}} d(x,y)^{4L} |\Theta_{j}^{q}(x,y)|^2 \, dy \lesssim_{L_0}  2^{-2j L}.
\end{align}
\end{lemma}

Note that estimate \eqref{General-square-estimate-singularity-with-q} is uniform in $q \in \mathbb{N}$. Assuming Lemma \ref{lem:General-Kernel-estimate-lambda-singularity} for now, let us see how to deduce Lemma \ref{lem:General-Kernel-estimate} from it. 

\begin{proof}[Proof of Lemma \ref{lem:General-Kernel-estimate}]
Note that we have 
$$ \int_{\mathbb{R}^{n_2}} \left| \Theta_{j}(x'', \lambda)(x^{\prime}, y^{\prime}) \right| d\lambda < \infty, $$
which holds true in view of conditions \eqref{def:grushin-kernel-Maucheri-hormander-cond1-L0-N0} and  \eqref{def:grushin-kernel-Maucheri-hormander-cond2-L0-N0}, and the support condition on $\lambda$ arising from the operator $S_j = \psi_{j} (G)$. In fact, we have here that $|\lambda| \leq 2^{j+1}$ (see \eqref{def:spectral-break}). 

Therefore, it follows from the dominated convergence theorem that 
\begin{align} \label{eq:General-square-estimate-step1}
\lim_{q \to \infty} \Theta_{j}^{q} (x,y) & = \lim_{q \to \infty}  \int_{\mathbb{R}^{n_2}} e^{-i \lambda \cdot (x^{\prime\prime} - y^{\prime\prime})} \Upsilon^{q}(\lambda) \Theta_{j}(x'', \lambda)(x^{\prime},y^{\prime}) \, d\lambda  \\ 
\nonumber & = \int_{\mathbb{R}^{n_2}} e^{-i \lambda \cdot (x^{\prime\prime} - y^{\prime\prime})}  \Theta_{j}(x'', \lambda)(x^{\prime},y^{\prime}) \, d\lambda \\ 
\nonumber & = \Theta_{j} (x,y). 
\end{align} 

Finally, we can make use of Fatou's Lemma to conclude that 
 \begin{align*}
& |B(x, 2^{-j/2})| \int_{\mathbb{R}^{n_1 + n_2}} d(x,y)^{4L} |\Theta_{j}(x,y)|^2 \, dy\\
& \leq \liminf_{q\rightarrow \infty} |B(x, 2^{-j/2})| \int_{\mathbb{R}^{n_1 + n_2}} d(x,y)^{4L} |\Theta^{q}_{j}(x,y)|^2 \, dy \lesssim 2^{-2j L}, 
\end{align*}
where the last inequality follows from Lemma \ref{lem:General-Kernel-estimate-lambda-singularity}. 

This completes the proof of Lemma \ref{lem:General-Kernel-estimate}. 
\end{proof}

\begin{proof}[Proof of Lemma \ref{lem:General-Kernel-estimate-lambda-singularity}]
In view of the definition of the metric $d(x,y)$, it suffices to decompose the integral in the claimed estimate \eqref{General-square-estimate-singularity-with-q} into two regions, namely, 
\begin{align} \label{eq:first-break-lem:General-Kernel-estimate-lambda-singularity}
I & = \left| B(x,2^{-j/2}) \right| \int_{\mathbb{R}^{n_2}} \int_{\mathfrak{E}_{j} (x^\prime)} |x'-y'|^{4L_1}\frac{|x''-y''|^{4 L_2}}{(|x'|+|y'|)^{4 L_2}} |\Theta_{j}^{q}(x,y)|^2 \, dy^{\prime} \, dy^{\prime \prime}, \\ 
\nonumber \textup{and} \quad II & = \left| B(x,2^{-j/2}) \right| \int_{\mathbb{R}^{n_2}} \int_{\mathfrak{E}_{j} (x^{\prime})^c} |x'-y'|^{4L_1} |x''-y''|^{2 L_2} |\Theta_{j}^{q}(x,y)|^2 \, dy^{\prime} \, dy^{\prime \prime}, 
\end{align}
with arbitrary $L_1$ and $L_2$ such that $L_1 + L_2=L$, and $\mathfrak{E}_{j} (x^{\prime})$ are sets as defined in \eqref{def:set-Cj(x-prime)}. In the following estimation, we write ${\mathbbm{1}}_{j, x'}$ for the characteristic function of the set $\mathfrak{E}_{j} (x')$. 

Let us analyse $|x'-y'|^{2L_1}|x''-y''|^{2 L_2}\Theta_{j}^{q}(x,y)$. Making use of the support condition on $\lambda$-variable (which is a consequence of support conditions of $\Upsilon^{q}(\lambda)$ and $S_j(\lambda)$), we can perform integration by parts in $\lambda$-variable to conclude that $|x'-y'|^{2L_1}|x''-y''|^{2 L_2} \Theta_{j}^{q}(x,y)$ can be written as a finite linear combination of terms of the following form 
\begin{align*}
\int_{\mathbb{R}^{n_2}} e^{-i\lambda\cdot (x^{\prime\prime}-y^{\prime\prime})} ~
\partial^{\vartheta _1}_{\lambda} \Upsilon^{q}(\lambda) \, D^{\alpha} \partial^{\vartheta_2}_{\lambda}\Theta_{j}(x'', \lambda)(x^\prime,y^\prime)~ d\lambda
\end{align*}
where $|\alpha|=2L_1$ and $|\vartheta_1 + \vartheta_2| = 2L_2.$ 
  
In view of the Euclidean Plancherel theorem for $\lambda$-variable, we get that the $L^2$ norm of the above kernel in $y$-variable is dominated by 
\begin{align} \label{est:Mauceri-type-lambda-singularity}
&\int_{\mathbb{R}^{n_1}}\int_{\mathbb{R}^{n_2}} |\partial^{\vartheta_1}_{\lambda} \Upsilon^{q}(\lambda)|^{2} \left| D^{\alpha} \partial^{\vartheta_2}_{\lambda} \Theta_{j}  (x'', \lambda) (x^\prime,y^\prime)\right|^{2} d\lambda \, dy^{\prime}\\
\nonumber & \lesssim \int_{\mathbb{R}^{n_1}}\int_{\mathbb{R}^{n_2}} \frac{\left| D^{\alpha} \partial^{\vartheta_2}_{\lambda}\Theta_{j}(x'', \lambda) (x^\prime,y^\prime)\right|^{2}}{|\lambda|^{2|\vartheta_1|}} d\lambda \, dy^{\prime},
\end{align}
where the last inequality follows from the fact that  $\displaystyle \sup_{\lambda, q} |\lambda|^{|\vartheta_1|}|\partial^{\vartheta_1}_{\lambda} \Upsilon^{q}(\lambda)| \leq C_{\vartheta_1} < \infty$. 

For a fixed $\lambda \neq 0$, by Leibniz formula, the operator $|\lambda|^{-|\vartheta_1|}  D^{\alpha} \partial_{\lambda}^{\vartheta_2} \Theta_{j}(x'', \lambda)$ can be expressed as a finite sum of operators of the forms 
$$ \left(|\lambda|^{-|\vartheta_1|} D^{\nu_1} \partial_{\lambda}^{\vartheta_3} \Theta(x'', \lambda) \right) \left( D^{\nu_2} \partial_{\lambda}^{\vartheta_4}S_{j}(\lambda) \right), $$
where $\nu_1 + \nu_2 = \alpha$ and $\vartheta_3 + \vartheta_4 = \vartheta_2$. 

We take one such term and decompose it as follows: 
\begin{align} \label{est:Mauceri-type-lambda-singularity-decompose}
& \left(|\lambda|^{-|\vartheta_1|} D^{\nu_1} \partial_{\lambda}^{\vartheta_3} \Theta(x'', \lambda) \right) \left( D^{\nu_2} \partial_{\lambda}^{\vartheta_4}S_{j}(\lambda) \right) \\ 
\nonumber & \quad \quad = \sum_{j'=0}^{\infty} \left(|\lambda|^{-|\vartheta_1|} D^{\nu_1} \partial_{\lambda}^{\vartheta_3} \Theta(x'', \lambda) \right) \chi_{j'}(\lambda)\left( D^{\nu_2} \partial_{\lambda}^{\vartheta_4}S_{j}(\lambda) \right). 
\end{align}
But, Lemma \ref{lem:finite-shifts} implies that the infinite sum on the right hand side of \eqref{est:Mauceri-type-lambda-singularity-decompose} will survive only when $|j-j'| \leq C L$ where the constant $C > 0$ depends only on $n_1$ and $n_2$. 

Let us also consider the following open subsets of $\mathbb{R}^{n_1}$, for $1 \leq k \leq n_1$: 
\begin{align} \label{def:x-prime-sets-Mauceri-calc-1}
\mathcal{D}_k^+ & = \left\{ z^{\prime} \in \mathbb{R}^{n_1} : z_k^{\prime} > 0 \, \, \textup{and} \, z_k^{\prime} > |z_i^{\prime}| \, \, \textup{for all} \, i \neq k \right \}, \\ 
\nonumber \textup{and} \quad \mathcal{D}_k^- & = \left\{ z^{\prime} \in \mathbb{R}^{n_1} : z_k^{\prime} < 0 \, \, \textup{and} \, - z_k^{\prime} > |z_i^{\prime}| \, \, \textup{for all} \, i \neq k \right\}, 
\end{align}
Clearly, all the sets $\mathcal{D}_k^+$ and $\mathcal{D}_j^-$ are pairwise disjoint and the set $\mathbb{R}^{n_1} \setminus \cup_{k =1}^{n_1} \left( \mathcal{D}_k^+ \cup \mathcal{D}_k^- \right)$ is of Lebesgue measure zero. 

Let us now estimate terms $I$ and $II$ of \eqref{eq:first-break-lem:General-Kernel-estimate-lambda-singularity}. 

\medskip \noindent \textbf{\underline{Estimation of term $I$ of \eqref{eq:first-break-lem:General-Kernel-estimate-lambda-singularity}}:}
It suffices to estimate terms of the following form
\begin{align} \label{kernel-decom}
& \frac{{\mathbbm{1}}_{j,x^{\prime}}(y^{\prime})}{(|x'|+|y'|)^{2 L_2}} \int\limits_{\mathbb{R}^{n_1}} \left( |\lambda|^{- |\vartheta_1 + \vartheta_4|} \left( D^{\nu_1} \partial_{\lambda}^{\vartheta_3} \Theta(x'', \lambda) \right) \chi_{j'}(\lambda) \right) (x^{\prime}, z^{\prime}) \, \left( |\lambda|^{|\vartheta_4|} D^{\nu_2} \partial_{\lambda}^{\vartheta_4} S_{j}(\lambda) \right) (z^{\prime}, y^{\prime}) \, dz^{\prime} \\ 
& \nonumber =: \sum_{k=1}^{n_1} I^{1, \lambda, x''}_{k,+} (x^{\prime}, y^{\prime}) + \sum_{k=1}^{n_1} I^{1, \lambda, x''}_{k,-} (x^{\prime}, y^{\prime}) + I^{2, \lambda, x''} (x^{\prime}, y^{\prime}), 
\end{align} 
where 
\begin{align*} 
I^{1, \lambda, x''}_{k,+} (x^{\prime}, y^{\prime}) & = \frac{{\mathbbm{1}}_{j,x^{\prime}}(y^{\prime})}{(|x'|+|y'|)^{2 L_2}} \int_{\mathfrak{E}_{j} (x^{\prime}) \cap \mathcal{D}_k^+} \frac{\left(|\lambda|^{- |\vartheta_1 + \vartheta_4|} \left( D^{\nu_1} \partial_{\lambda}^{\vartheta_3} \Theta(x'', \lambda) \right) \chi_{j'}(\lambda) \right) (x^{\prime}, z^{\prime})}{(|x'|+z_k')^{2 L_2}} \\ 
& \quad \quad \quad \quad \quad \quad \quad \quad \quad \quad \quad  (|x'|+z_k')^{2 L_2} \left( |\lambda|^{|\vartheta_4|} D^{\nu_2} \partial_{\lambda}^{\vartheta_4} S_{j}(\lambda) \right) (z^{\prime}, y^{\prime}) \, dz^{\prime}, \\ 
I^{1, \lambda, x''}_{k,-} (x^{\prime}, y^{\prime}) & = \frac{{\mathbbm{1}}_{j,x^{\prime}}(y^{\prime})}{(|x'|+|y'|)^{2 L_2}} \int_{\mathfrak{E}_{j} (x^{\prime}) \cap \mathcal{D}_k^-} \frac{\left( |\lambda|^{- |\vartheta_1 + \vartheta_4|} \left( D^{\nu_1} \partial_{\lambda}^{\vartheta_3} \Theta(x'', \lambda) \right) \chi_{j'}(\lambda) \right) (x^{\prime}, z^{\prime})}{(|x'|-z_k')^{2 L_2}} \\ 
& \quad \quad \quad \quad \quad \quad \quad \quad \quad \quad \quad  (|x'|-z_k')^{2 L_2} \left( |\lambda|^{|\vartheta_4|} D^{\nu_2} \partial_{\lambda}^{\vartheta_4} S_{j}(\lambda) \right) (z^{\prime}, y^{\prime}) \, dz^{\prime}, \\ 
I^{2, \lambda, x''} (x^{\prime}, y^{\prime}) & = \frac{{\mathbbm{1}}_{j,x^{\prime}}(y^{\prime})}{(|x'|+|y'|)^{2 L_2}} \int_{\mathfrak{E}_{j} (x^{\prime})^c} \left(|\lambda|^{- |\vartheta_1 + \vartheta_4|} \left( D^{\nu_1} \partial_{\lambda}^{\vartheta_3} \Theta(x'', \lambda) \right) \chi_{j'} (\lambda) \right) (x^{\prime}, z^{\prime}) \\ 
& \quad \quad \quad \quad \quad \quad \quad \quad \quad \quad \quad \left( |\lambda|^{|\vartheta_4|} D^{\nu_2} \partial_{\lambda}^{\vartheta_4} S_{j} (\lambda) \right) (z^{\prime}, y^{\prime}) \, dz^{\prime}.
\end{align*}

\medskip \textbf{\underline{Analysis of the term $I^{1, \lambda, x''}_{k,+} (x^{\prime}, y^{\prime}) $ in \eqref{kernel-decom}}:} 

We can rewrite $I^{1, \lambda, x''}_{k,+} (x^{\prime}, y^{\prime})$ as follows: 
\begin{align*}
I^{1, \lambda, x''}_{k,+} (x^{\prime}, y^{\prime}) = \int_{\mathbb{R}^{n_1}} T^{1, \lambda, x'}_{k,+} (z^{\prime}, y^{\prime}) {\mathbbm{1}}_{j,x^{\prime}}(z^{\prime}) f^{1, \lambda, x}_{k, +} (z^{\prime}) \, dz^{\prime},  
\end{align*} 
where 
$$f^{1, \lambda, x}_{k,+} (z^{\prime}) = \mathbbm{1}_{\mathcal{D}_k^+} (z^{\prime}) \frac{\left( |\lambda|^{-|\vartheta_1 + \vartheta_4|} \left( D^{\nu_1} \partial_{\lambda}^{\vartheta_3} \Theta(x'', \lambda) \right) \chi_{j'}(\lambda) \right) (x^{\prime}, z^{\prime})}{(|x'|+z_k')^{2 L_2}},$$ 
and 
$$T^{1, \lambda, x'}_{k,+} (z^{\prime}, y^{\prime}) = {\mathbbm{1}}_{j,x^{\prime}}(y^{\prime}) \frac{(|x'|+z_k')^{2 L_2}}{(|x'|+|y'|)^{2 L_2}} \left( |\lambda|^{|\vartheta_4|} D^{\nu_2} \partial_{\lambda}^{\vartheta_4} S_{j} (\lambda) \right)(z^{\prime}, y^{\prime}).$$

Since $|\lambda|^{-|\vartheta_1 + \vartheta_4|} D^{\nu_1} \partial_{\lambda}^{\vartheta_3} \Theta(x'', \lambda) $ can be expressed as a finite linear combination of operators of the form $|\lambda|^{-\left(\frac{|\gamma_1 + \gamma_2|}{2} + |\vartheta_1 + \vartheta_4| \right)} \delta^{\gamma_1}(\lambda) \bar{\delta}^{\gamma_2}(\lambda) (\partial_{\lambda}^{\vartheta_3} \Theta(x'', \lambda))$ with $\gamma_1 + \gamma_2 = \nu_1$. Thus, it suffices to consider one such piece in the definition of the function $f^{1, \lambda, x}_{k,+}$. By abuse of notation, we use $f^{1, \lambda, x}_{k,+}$ to denote one such associated piece, namely, 
\begin{align*}
f^{1, \lambda, x}_{k,+}(z^{\prime}) & = \mathbbm{1}_{\mathcal{D}_k^+} (z^{\prime}) \frac{\left(|\lambda|^{-\left(\frac{|\gamma_1 + \gamma_2|}{2} + |\vartheta_1 + \vartheta_4| \right)} \left( \delta^{\gamma_1}(\lambda) \bar{\delta}^{\gamma_2}(\lambda)  \partial_{\lambda}^{\vartheta_3} \Theta(x'', \lambda) \right) \chi_{j'}(\lambda) \right) (x^{\prime}, z^{\prime})}{(|x'|+z_k')^{2 L_2}} \\
& = \mathbbm{1}_{\mathcal{D}_k^+} (z^{\prime}) \frac{\left( \left(\mathcal{L}(\lambda)_{ \vartheta_3, \vartheta_1 + \vartheta_4}^{\gamma_1, \gamma_2} \Theta(x'', \lambda) \right) \chi_{j'} (\lambda) \right) (x^{\prime}, z^{\prime})}{(|x'|+z_k')^{2 L_2}}.
\end{align*} 
 
Therefore, 
\begin{align*}
\int_{\mathbb{R}^{n_1}} \int_{\mathbb{R}^{n_2}} \left| I^{1, \lambda, x''}_{k,+} (x^{\prime}, y^{\prime}) \right|^2 d\lambda \, dy^{\prime} & \lesssim \int_{\mathbb{R}^{n_1}} \int_{\mathbb{R}^{n_2}} \left| \int_{\mathbb{R}^{n_1}} T^{1, \lambda, x'}_{k,+}(z^{\prime}, y^{\prime}) {\mathbbm{1}}_{j,x^{\prime}}(z^{\prime}) f^{1, \lambda, x}_{k, +} (z^{\prime}) \, dz^{\prime} \right|^2 d\lambda \, dy^{\prime} \\ 
& = \int_{\mathbb{R}^{n_1+n_2}} \left| \left( \left( A^{1, x}_{k, +} \right)^* F^{1, x}_{k, +} \right) (y) \right|^2 \, dy, 
\end{align*}
where $F^{1, x}_{k, +}$ is the function on $\mathbb{R}^{n_1 + n_2}$ given by 
$$ F^{1, x}_{k, +} (z) = {\mathbbm{1}}_{j,x^{\prime}}(z^{\prime}) \int_{\mathbb{R}^{n_2}} e^{ - i \lambda \cdot z^{\prime \prime}} f^{1, \lambda, x}_{k, +} (z^{\prime}) \, d\lambda $$
and $\left( A^{1, x}_{k, +} \right)^*$ is the adjoint of the operator $A^{1, x}_{k, +}$ and the kernel $A^{1, x}_{k, +} (z,y)$ of the operator $A^{1, x}_{k, +}$ given by 
$$ A^{1, x}_{k, +} (z,y) = \int_{\mathbb{R}^{n_2}} e^{ - i \lambda \cdot ( z^{\prime \prime} - y^{\prime \prime}) } {\mathbbm{1}}_{j,x^{\prime}}(y^{\prime}) \frac{(|x'|+z_k')^{2 L_2}}{(|x'|+|y'|)^{2 L_2}} \left( |\lambda|^{|\vartheta_4|} D^{\nu_2} \partial_{\lambda}^{\vartheta_4}S_{j}(\lambda) \right) (z^{\prime}, y^{\prime}) \, d \lambda.$$

We therefore have 
\begin{align} \label{est:mauceri-type-term1-calc-op-norm}
\int_{\mathbb{R}^{n_1}} \int_{\mathbb{R}^{n_2}} \left| I^{1, \lambda, x''}_{k,+} (x^{\prime}, y^{\prime}) \right|^2 d\lambda \, dy^{\prime}  \lesssim \left\| A^{1, x}_{k, +} \right\|^2_{op} \left\| F^{1, x}_{k, +} \right\|^2_2.  
\end{align}

Now, the $L^2$-norm of the function $F^{1, x}_{k, +}$ comes from the assumption \eqref{def:grushin-kernel-Maucheri-hormander-cond1-L0-N0}. More precisely, 
\begin{align} \label{est:mauceri-type-term1-calc-op-norm-easy-part} 
\left\| F^{1, x}_{k, +} \right\|^2_2 & = \int_{\mathfrak{E}_{j} (x^{\prime}) \cap \mathcal{D}_k^+ } \int_{\mathbb{R}^{n_2}} \left| \frac{\left( \left( \mathcal{L}(\lambda)_{ \vartheta_3, \vartheta_1 + \vartheta_4}^{\gamma_1, \gamma_2} \Theta(x'', \lambda) \right) \chi_{j'} (\lambda) \right) (x^{\prime}, z^{\prime})}{(|x'|+z_k')^{2 L_2}}\right|^{2} d\lambda \, dz^{\prime} \\ 
\nonumber & \lesssim \int_{\mathfrak{E}_{j} (x^{\prime})}\int_{\mathbb{R}^{n_2}} \left| \frac{\left( \left( \mathcal{L}(\lambda)_{ \vartheta_3, \vartheta_1 + \vartheta_4}^{\gamma_1, \gamma_2} \Theta(x'', \lambda) \right) \chi_{j'} (\lambda) \right) (x^{\prime}, z^{\prime})}{(|x'|+ |z'|)^{2 L_2}}\right|^{2} d\lambda \, dz^{\prime} \\ 
\nonumber & \lesssim 2^{-j\left( |\gamma_1 + \gamma_2|+ 2 L_2 \right)} |B(x, 2^{-j/2})|^{-1}.
\end{align} 

We now turn to the main estimation, namely the operator norm $\left\| A^{1, x}_{k, +} \right\|_{op}$. Note that the kernel $A^{1, x}_{k, +} (z,y)$ is a finite linear combination of terms of the following form: 
$$\int_{\mathbb{R}^{n_2}} e^{- i \lambda \cdot ( z^{\prime \prime} - y^{\prime \prime})}{\mathbbm{1}}_{j,x^{\prime}}(y^{\prime})\frac{|x^{\prime}|^{b_1} {z^{\prime}_{k}}^{b_2}}{(|x'|+|y'|)^{2 L_2}} (z'-y')^{\nu_2} |\lambda|^{|\vartheta_4|}
\left( \partial_{\lambda}^{\vartheta_4} S_{j}(\lambda) \right) (z^{\prime}, y^{\prime}) d\lambda, $$ 
with $b_1 + b_2 = 2 L_2$.

Writing $z^{\prime}_{k} = (z^{\prime}_{k} - y^{\prime}_{k}) + y^{\prime}_{k}$, with arbitrary $b_3 + b_4 = b_2$ and $\tilde{\nu}_2 = \nu_2 + b_4 e_{k}$, it suffices to estimate terms of the following form: 
\begin{align*}
\int_{\mathbb{R}^{n_2}} e^{- i \lambda \cdot ( z^{\prime \prime} - y^{\prime \prime})}{\mathbbm{1}}_{j,x^{\prime}}(y^{\prime})\frac{|x^{\prime}|^{b_1}{y^{\prime}_{k}}^{b_3}} {(|x'|+|y'|)^{2 L_2}}  (z'-y')^{\tilde{\nu}_2} |\lambda|^{|\vartheta_4|}
\left( \partial_{\lambda}^{\vartheta_4} S_{j}(\lambda) \right) (z^{\prime}, y^{\prime}) \, d\lambda.
\end{align*} 

In view of Lemma \ref{weighted-kernel-estimate-3}, the above term can be written as a finite linear combination of terms of the following form 
\begin{align} \label{Reduced-kernel-expression-T-1}
& {\mathbbm{1}}_{j,x}(y^{\prime})\frac{|x^{\prime}|^{b_1} {y^{\prime}_{k}}^{b_3}}{(|x'|+|y'|)^{2 L_2}} {y^{\prime}}^{\alpha_{1}} \int\limits_{[0,1]^{N_1} \times \Omega^{N_2} \times [0,1]^{|\tilde{\nu}|}} \int_{\mathbb{R}^{n_2}} e^{-i \lambda \cdot (z''-y'')} |\lambda|^{|\vartheta_4|} \, \mathfrak{A}_l(\lambda) \\ 
\nonumber & \quad \sum_{\mu} C_{\mu, \vec{c}} \left( \tau^{\frac{1}{2} \theta_1} \partial_\tau^{\theta_2} \psi_{j}\right) ((2\mu + \tilde{1} + \vec{c}(\omega)) |\lambda|) \Phi_{\mu}^{\lambda}(z^{\prime}) \Phi_{\mu + \tilde{\mu}}^{\lambda}(y^{\prime}) g(\omega) \, d\lambda \, d \omega \\ 
\nonumber & =  {\mathbbm{1}}_{j,x^{\prime}}(y^{\prime}) \frac{|x^{\prime}|^{b_1} {y^{\prime}_{k}}^{b_3}}{(|x'|+|y'|)^{2 L_2}}  \int\limits_{[0,1]^{N_1} \times \Omega^{N_2} \times [0,1]^{|\tilde{\nu}|}} \int_{\mathbb{R}^{n_2}} e^{-i \lambda \cdot (z''-y'')} \left( |\lambda| {y^{\prime}} \right)^{\alpha_{1}} |\lambda|^{|\vartheta_4| - |\alpha_1|} \, \mathfrak{A}_l(\lambda) \\ 
\nonumber & \quad \quad \sum_{\mu} C_{\mu, \vec{c}} \left(\tau^{\frac{1}{2} \theta_1} \partial_\tau^{\theta_2} \psi_{j}\right) ((2\mu + \tilde{1} + \vec{c}(\omega)) |\lambda|) \Phi_{\mu}^{\lambda}(z^{\prime}) \Phi_{\mu + \tilde{\mu}}^{\lambda}(y^{\prime}) g(\omega) \, d\lambda \, d \omega, 
\end{align} 
where $N_1, N_2, | \alpha_1 | \leq |\vartheta_4| - l$, $|\mathcal{\theta}_1| \leq |\mathcal{\theta}_2| \leq |\tilde{\nu}_2| + 2 |\vartheta_4| - 2 l$, $|\mathcal{\theta}_2| - \frac{|\mathcal{\theta}_1|}{2} = \frac{|\tilde{\nu}_2|}{2} + |\vartheta_4| - l - \frac{|\alpha_1|}{2} \geq (|\tilde{\nu}_2| + |\vartheta_4| - l) / 2$, and $\mathfrak{A}_l$ is a continuous function on $\mathbb{R}^{n_2} \setminus \{0\}$ which is homogeneous of degree $-l$. 

Using the ideas from Remark $4.8$ of \cite{Bagchi-Garg-1} on how to tackle terms of the form $\left( |\lambda| {y^{\prime}} \right)^{\alpha_{1}}$ (in fact, these terms behave like $\tau^{\frac{\alpha_{1}}{2}}$), we deduce from \eqref{Reduced-kernel-expression-T-1} 
\begin{align} \label{Op-norm-1} 
\left\| A^{1, x}_{k, +} \right\|_{op} & \lesssim 2^{j (2 L_2 - b_1 - b_3)/2} 2^{j |\alpha_1|/2} 2^{j(|\vartheta_4| - |\alpha_1| - l)} 2^{-j \left( |\theta_2| - \frac{|\theta_1|}{2} \right)} \\ 
\nonumber & = 2^{j b_4/2} 2^{j |\alpha_1|/2} 2^{j(|\vartheta_4| - |\alpha_1| - l)} 2^{-j \left( \frac{|\tilde{\nu}_2|}{2} + |\vartheta_4| - l - \frac{|\alpha_1|}{2} \right)} \\ 
\nonumber & = 2^{j b_4/2} 2^{-j 
|\tilde{\nu}_2|/2} \\ 
\nonumber & = 2^{-j |\nu_2|/2}. 
\end{align}

In view of \eqref{est:mauceri-type-term1-calc-op-norm-easy-part}, \eqref{Op-norm-1} and the condition $|j'-j| \leq C L $, we get from \eqref{est:mauceri-type-term1-calc-op-norm} that 
\begin{align*}
|B(x, 2^{-j/2})| \int_{\mathbb{R}^{n_1}} \int_{\mathbb{R}^{n_2}} \left| I^{1, \lambda, x''}_{k,+} (x^{\prime}, y^{\prime}) \right|^2 d\lambda \, dy^{\prime} \lesssim 2^{-j |\nu_2|} 2^{-j\left( |\nu_1 + \nu_2| + 2 L_2 \right)} = 2^{-j(2L_1+2L_2)} = 2^{-2jL}.
\end{align*}
 
\medskip \textbf{\underline{Analysis of the term $I^{1, \lambda, x''}_{k, -} (x^{\prime}, y^{\prime}) $ in \eqref{kernel-decom}}:} 

It can be estimated essentially same as the term $I^{1, \lambda, x''}_{k,+} (x^{\prime}, y^{\prime})$. 

\medskip \textbf{\underline{Analysis of the term $I^{2, \lambda, x''} (x^{\prime}, y^{\prime})$ in \eqref{kernel-decom}}:} 

Similar to the analysis of $I^{1,\lambda, x''}_{k,+} (x^{\prime}, y^{\prime})$, we rewrite \begin{align*}
I^{2,\lambda, x''} (x^{\prime}, y^{\prime}) = \int_{\mathbb{R}^{n_1}} T^{2, \lambda}(z^{\prime},y^{\prime}) \left(1 - {\mathbbm{1}}_{j,x^{\prime}}(z^{\prime}) \right) f^{2, \lambda, x^{\prime}}(z^{\prime}) \, dz', 
\end{align*}
where 
$$f^{2, \lambda, x}(z^{\prime}) = |\lambda|^{- |\vartheta_1 + \vartheta_4|} \left( \left( D^{\nu_1} \partial_{\lambda}^{\vartheta_3}\Theta(x'', \lambda) \right) \chi_{j'}(\lambda) \right) (x^{\prime}, z^{\prime}),$$ 
and 
$$T^{2, \lambda, x'} (z^{\prime},y^{\prime}) = {\mathbbm{1}}_{j,x^{\prime}} (y^{\prime}) \frac{|\lambda|^{|\vartheta_4|}}{(|x'|+|y'|)^{2 L_2}} \left( D^{\nu_2} \partial_{\lambda}^{\vartheta_4} S_{j}(\lambda) \right)(z^{\prime}, y^{\prime}).$$ 

As earlier, with $\gamma_1 + \gamma_2 = \nu_1$, it suffices to analyse 
\begin{align*}
f^{2, \lambda, x}(z^{\prime}) &= \left(|\lambda|^{-\left(\frac{|\gamma_1 + \gamma_2|}{2} + |\vartheta_1 + \vartheta_4| \right)} \left( \delta^{\gamma_1}(\lambda) \bar{\delta}^{\gamma_2}(\lambda) \partial_{\lambda}^{\vartheta_3} \Theta(x'', \lambda) \right) \chi_{j'}(\lambda) \right) (x^{\prime}, z^{\prime}) \\
& = \left(\mathcal{L}(\lambda)_{ \vartheta_3, \vartheta_1+\vartheta_4}^{ \gamma_1, \gamma_2} \Theta(x'', \lambda)\chi_{j'}(\lambda)\right)(x^{\prime}, z^{\prime}), 
\end{align*}
which implies that 
\begin{align*}
& \int_{\mathbb{R}^{n_1}} \int_{\mathbb{R}^{n_2}} \left| I^{2, \lambda, x''} (x^{\prime}, y^{\prime}) \right|^2 d\lambda \, dy^{\prime} \\ 
& \quad \lesssim \int_{\mathbb{R}^{n_1}} \int_{\mathbb{R}^{n_2}} \left| \int_{\mathbb{R}^{n_1}} T^{2, \lambda, x'}(z^{\prime}, y^{\prime}) \left(1 - {\mathbbm{1}}_{j,x^{\prime}} (z^{\prime}) \right) f^{2, \lambda, x}(z^{\prime}) \, dz^{\prime} \right|^2 d\lambda \, dy^{\prime} \\ 
& \quad = \int_{\mathbb{R}^{n_1 + n_1}} \left| \left( \left( A^{2, x} \right)^* F^{2, x} \right) (y) \right|^2 \, dy, 
\end{align*}
where $F^{2,x}$ is the function given by 
$$ F^{2,x}(z) = \int_{\mathbb{R}^{n_2}} e^{- i \lambda \cdot z^{\prime \prime}} \left(1 - {\mathbbm{1}}_{j,x^{\prime}}(z^{\prime})\right)  f^{2, \lambda, x} (z^{\prime}) \, d\lambda $$
and $ \left( A^{2,x} \right)^*$ is the adjoint of the operator $A^{2,x}$, with the kernel $A^{2,x}(z,y)$ of the operator $A^{2,x}$ given by 
$$ A^{2,x}(z,y) = \int_{\mathbb{R}^{n_2}} e^{ - i \lambda \cdot ( z^{\prime \prime} - y^{\prime \prime}) } {\mathbbm{1}}_{j,x^{\prime}}(y^{\prime}) \frac{1}{(|x'|+|y'|)^{2 L_2}} \left( |\lambda|^{|\vartheta_4|} D^{\nu_2} \partial_{\lambda}^{\vartheta_4}S_{j}(\lambda) \right) (z^{\prime}, y^{\prime}) \, d \lambda.$$

We therefore have 
\begin{align} \label{est:mauceri-type-term2-calc-op-norm}
\int_{\mathbb{R}^{n_1}} \int_{\mathbb{R}^{n_2}} \left| I^{2, \lambda, x''} (x^{\prime}, y^{\prime}) \right|^2 d\lambda \, dy^{\prime} \lesssim \left\| A^{2,x} \right\|^2_{op} \left\| F^{2,x} \right\|^2_2. 
\end{align}

The $L^2$-norm of the function $F^{2, x}$ comes from the assumption \eqref{def:grushin-kernel-Maucheri-hormander-cond2-L0-N0}. More precisely, 
\begin{align} \label{est:mauceri-type-term1-calc-op-norm-easy-part-2} 
\left\| F^{2,x} \right\|^2_2 & = \int_{\mathfrak{E}_{j} (x^{\prime})^{c}} \int_{\mathbb{R}^{n_2}} \left|\left(\mathcal{L}(\lambda)_{ \vartheta_3, \vartheta_1 + \vartheta_4}^{\gamma_1, \gamma_2} \Theta(x'', \lambda) \chi_{j'}(\lambda))\right) (x^{\prime}, z^{\prime})\right|^{2} d\lambda \, dy^{\prime} \\ 
\nonumber & \lesssim 2^{-j\left( |\gamma_1 + \gamma_2|+ 4 L_2 \right)} |B(x, 2^{-j/2})|^{-1}. 
\end{align}

Finally, to estimate the operator norm of $A^{2,x}$, we apply ideas same as that done to estimate the operator norm of $A^{1,x}_{k, +}$ during the analysis of the term $I^{1, \lambda, x''}_{k, -}$. For this, using \eqref{Reduced-kernel-expression-T-1} we get that $A^{2,x}(z,y)$ can be written as a finite linear combination of the terms of the following form 
\begin{align} \label{Reduced-kernel-expression-T-2}
& {\mathbbm{1}}_{j,x^{\prime}}(y^{\prime}) \frac{1}{(|x'|+|y'|)^{2 L_2}}  \int\limits_{[0,1]^{N_1} \times \Omega^{N_2} \times [0,1]^{|\nu|}} \int_{\mathbb{R}^{n_2}} e^{-i \lambda \cdot (z''-y'')} \left( |\lambda| {y^{\prime}} \right)^{\alpha_{1}} |\lambda|^{|\vartheta_4| - |\alpha_1|} \, \mathfrak{A}_l(\lambda) \\ 
\nonumber & \quad \quad \sum_{\mu} C_{\mu, \vec{c}} \left(\tau^{\frac{1}{2} \theta_1} \partial_\tau^{\theta_2} \psi_{j}\right) ((2\mu + \tilde{1} + \vec{c}(\omega)) |\lambda|) \Phi_{\mu}^{\lambda}(z^{\prime}) \Phi_{\mu + \tilde{\mu}}^{\lambda}(y^{\prime}) g(\omega) \, d\lambda \, d \omega, 
\end{align} 
where the indices and the conditions on them are same as to those in \eqref{Reduced-kernel-expression-T-1} with the understanding that here $b_1 = b_2 = 0$. 

From the expression \eqref{Reduced-kernel-expression-T-2}, we get 
\begin{align} \label{Op-norm-2} 
\left\| A^{2,x} \right\|_{op} \lesssim 2^{j L_2} 2^{j |\alpha_1|/2} 2^{j(|\vartheta_4| - |\alpha_1| - l)} 2^{-j (|\theta_2| - \frac{|\theta_1|}{2})} = 2^{j L_2} 2^{-j |\nu_2|/2}.
\end{align}

In view of \eqref{est:mauceri-type-term1-calc-op-norm-easy-part-2}, \eqref{Op-norm-2} and the condition $|j'-j|\leq C L $, we get from \eqref{est:mauceri-type-term2-calc-op-norm} that 
\begin{align*}
& |B(x, 2^{-j/2})| \int_{\mathbb{R}^{n_1}} \int_{\mathbb{R}^{n_2}} \left| I^{2, \lambda, x''} (x^{\prime}, y^{\prime}) \right|^2 d\lambda \, dy^{\prime} \\ 
& \quad \lesssim 2^{2j L_2} 2^{-j|\nu_2|} 2^{-j\left( |\gamma_1 + \gamma_2| + 4 L_2 \right)} = 2^{-j(2L_1+2L_2)} = 2^{-2jL}, 
\end{align*}
completing the estimation of term $I$ of \eqref{eq:first-break-lem:General-Kernel-estimate-lambda-singularity}. 

Term $II$ of \eqref{eq:first-break-lem:General-Kernel-estimate-lambda-singularity} could be estimated in an analogous manner, so we omit its proof. This completes the proof of Lemma \ref{lem:General-Kernel-estimate-lambda-singularity}. 
\end{proof}

\begin{lemma} \label{lem:General-Kernel-grad-estimate-x-lambda-singularity}
Suppose $\Theta$ satisfies the conditions \eqref{def:grushin-kernel-Maucheri-hormander-cond1-L0-N0} and \eqref{def:grushin-kernel-Maucheri-hormander-cond2-L0-N0}, for some $L_0 \in \mathbb{N}$ and $N_0 = 1$. Then for all $L \leq \frac{L_0 - 1}{2}$ and $q \in \mathbb{N}$, we have 
\begin{align} 
\sup_{x\in \mathbb{R}^{n_1+n_2}} |B(x, 2^{-j/2})| \int_{\mathbb{R}^{n_1+n_2}} d(x,y)^{4L} |X_{x} \Theta_{j}^{q}(x,y)|^2 \, dy & \lesssim_{L_0} 2^{-2j L}2^{j}, \label{General-square-grad-x-estimate-singularity-q}
\end{align}
\end{lemma}

Before we prove Lemma \ref{lem:General-Kernel-grad-estimate-x-lambda-singularity}, let us note that it immediately implies following analogous kernel estimates for $\Theta_{j}(x,y)$. 

\begin{lemma} \label{lem:General-Kernel-grad-estimate-x}
Suppose $\Theta$ satisfies the conditions \eqref{def:grushin-kernel-Maucheri-hormander-cond1-L0-N0} and \eqref{def:grushin-kernel-Maucheri-hormander-cond2-L0-N0}, for some $L_0 \in \mathbb{N}$ and $N_0 = 1$. Then for all $L \leq \frac{L_0 - 1}{2}$, we have 
\begin{align} 
\sup_{x\in \mathbb{R}^{n_1+n_2}} |B(x, 2^{-j/2})| \int_{\mathbb{R}^{n_1+n_2}} d(x,y)^{4L} |X_{x} \Theta_{j} (x,y)|^2 \, dy & \lesssim_{L_0} 2^{-2j L}2^{j}, \label{General-square-grad-x-estimate}
\end{align}
\end{lemma}

Lemma \ref{lem:General-Kernel-grad-estimate-x} can be deduced from Lemma \ref{lem:General-Kernel-grad-estimate-x-lambda-singularity} in the same way as Lemma \ref{lem:General-Kernel-estimate} from Lemma \ref{lem:General-Kernel-estimate-lambda-singularity}. The only change in the proof would be to argue that the gradients of $\Theta_{j}^{q} (x,y)$ converge pointwise to those of $\Theta_{j} (x,y)$ as $q \to \infty$. Since it is quite straightforward to verify, we leave the details and proceed to prove Lemma \ref{lem:General-Kernel-grad-estimate-x-lambda-singularity}. 

\begin{proof} [Proof of Lemma \ref{lem:General-Kernel-grad-estimate-x-lambda-singularity}] 
 As in the proof of Lemma \ref{lem:General-Kernel-estimate-lambda-singularity}, we divide the integral in the claimed estimate \eqref{General-square-grad-x-estimate-singularity-q} into two regions, namely, 
\begin{align} \label{eq:first-break-lem:General-Kernel-estimate-lambda-singularity-2}
I & = \left| B(x,2^{-j/2}) \right| \int_{\mathbb{R}^{n_2}} \int_{\mathfrak{E}_{j} (x^\prime)} |x'-y'|^{4L_1}\frac{|x''-y''|^{4 L_2}}{(|x'|+|y'|)^{4 L_2}} | X_x \Theta_{j}^{q}(x,y)|^2 \, dy^{\prime} \, dy^{\prime \prime}, \\ 
\nonumber \textup{and} \quad II & = \left| B(x,2^{-j/2}) \right| \int_{\mathbb{R}^{n_2}} \int_{\mathfrak{E}_{j} (x^{\prime})^c} |x'-y'|^{4L_1} |x''-y''|^{2 L_2} |X_x \Theta_{j}^{q}(x,y)|^2 \, dy^{\prime} \, dy^{\prime \prime}, 
\end{align}
with arbitrary $L_1$ and $L_2$ such that $L_1 + L_2=L$. As earlier, the sets $\mathfrak{E}_{j} (x^{\prime})$ are given by \eqref{def:set-Cj(x-prime)}. 

Once again, we only estimate term $I$ of \eqref{eq:first-break-lem:General-Kernel-estimate-lambda-singularity-2} and the estimation for term $II$ could be done in a similar manner. 

Note that 
\begin{align} \label{eq:first-break-lem:General-Kernel-estimate-lambda-singularity-2-derivative} 
|X_{x}\Theta_{j}^{q}(x,y)| \leq \sum_{k=1}^{n_1} \left| \frac{\partial}{\partial x_k^{\prime}} \Theta_{j}^{q}(x,y) \right| + \sum_{k=1}^{n_1}\sum_{l=1}^{n_2} \left| x_k^{\prime} \frac{\partial}{\partial x^{\prime \prime}_{l}} \Theta_{j}^{q}(x,y) \right|.
\end{align} 

\medskip \noindent \textbf{\underline{Estimation of the kernel corresponding to $\frac{\partial}{\partial x'_{k}}\Theta_{j}^{q}(x,y)$}:} 
We first consider terms of the type $\frac{\partial}{\partial x'_{k}}\Theta_{j}^{q}(x,y)$. It boils down to estimate terms of the form $|x^{\prime}-y^{\prime}|^{2L_1}|x^{\prime\prime}-y^{\prime\prime}|^{2L_2}\frac{\partial}{\partial x^{\prime}_{k}}\Theta_{j}^{q}(x,y)$ where $L_1+L_2=L$. Since $|x^{\prime\prime}-y^{\prime\prime}|$ commutes with $\frac{\partial}{\partial x^{\prime}_{k}}$, it is equivalent to estimate $|x^{\prime}-y^{\prime}|^{2L_1}\frac{\partial}{\partial x^{\prime}_{k}} \left(|x^{\prime\prime}-y^{\prime\prime}|^{2L_2}\Theta_{j}^{q}(x,y)\right).$ 

Since $\frac{\partial}{\partial x'_{k}}=\frac{1}{2}(A_{k}(\lambda) - A_{k}(\lambda)^{*})$, following the arguments leading to \eqref{est:Mauceri-type-lambda-singularity}, we are led to estimate 
\begin{align}\label{Plancherel-q-part-remove}
\int_{\mathbb{R}^{n_1}} \int_{\mathbb{R}^{n_2}}  \frac{\left| \left( D^{\alpha} A_{k}(\lambda) \partial^{\vartheta_2}_{\lambda} \Theta_{j}(x'', \lambda) \right) (x^\prime,y^\prime)\right|^{2}}{|\lambda|^{2|\vartheta_1|}} d\lambda \, dy^{\prime},
\end{align}
and an analogous term with $A_{k}(\lambda)$ replaced by $A_{k}(\lambda)^*$. 

Using also the relation $A_k(\lambda) M(\lambda) = |\lambda|^{\frac{1}{2}} \delta_{k}(\lambda) M(\lambda) + M(\lambda)  A_{k}(\lambda),$ we can perform calculations similar to the ones leading to \eqref{est:Mauceri-type-lambda-singularity-decompose} to argue that we need to analyse kernels of operators $|\lambda|^{-|\vartheta_1| + \frac{1}{2}} D^{\alpha} \delta_{k}(\lambda) \partial ^{\vartheta_2}_{\lambda} \Theta_{j}(x'', \lambda)$ and $|\lambda|^{-|\vartheta_1|} \left( D^{\alpha} \partial^{\vartheta_2}_{\lambda} \Theta_{j}(x'', \lambda) \right) A_{k}(\lambda)$, with $|\alpha| = 2 L_1$ and $|\vartheta_1 + \vartheta_2| = 2 L_2$, and we analyse kernels of these operators separately.

Operator $|\lambda|^{-|\vartheta_1| + \frac{1}{2}} D^{\alpha} \delta_{k}(\lambda) \partial_{\lambda}^{\vartheta_2} \Theta_{j}(x'', \lambda)$ is a finite linear combination of operators of the form
\begin{align*}
|\lambda|\left( |\lambda|^{-\left( \frac{|\gamma_1 + \gamma_2| + 1}{2} + |\vartheta_1| \right) } \delta^{\gamma_1}(\lambda) \bar{\delta}^{\gamma_2}(\lambda) \delta_{k}(\lambda) \partial ^{\vartheta_2}_{\lambda} \Theta_{j}(x'', \lambda) \right), \quad \text{with } |\gamma_1 + \gamma_2| = 2 L_1.
\end{align*}

Writing 
$$ K_1(x , y) = \int_{\mathbb{R}^{n_2}} e^{-i \lambda \cdot (x^{\prime\prime} - y^{\prime\prime})}  \left( |\lambda| \left(|\lambda|^{-\left( \frac{|\gamma_1| + |\gamma_2| + 1}{2} + |\vartheta_1| \right)} \delta^{\gamma_1 + e_k}(\lambda) \bar{\delta}^{\gamma_2}(\lambda) \partial ^{\vartheta_2}_{\lambda} \Theta_{j}(x'', \lambda) \right) \right) (x^{\prime},y^{\prime}) \, d\lambda, $$
one can essentially repeat the proof of Lemma  \ref{lem:General-Kernel-estimate-lambda-singularity} to show that  
\begin{align}\label{Grad-ker-1}
\left| B(x,2^{-j/2}) \right| \int_{\mathfrak{E}_{j} (x^{\prime})} \int_{\mathbb{R}^{n_2}} \left| \frac{K_1(x,y)}{(|x^{\prime}| + |y^{\prime}| )^{2L_2}} \right|^{2} \, dy^{\prime} \, dy^{\prime\prime} & \lesssim 2^{-2jL} 2^j.
\end{align}
Here the additional growth factor $2^j$ is a resultant of the extra $|\lambda|$ in the integral in the definition of $K_1 (x,y)$. 
 
On the other hand, the operator $|\lambda|^{-|\vartheta_1|} \left( D^{\alpha} \partial ^{\vartheta_2}_{\lambda}\Theta_{j}(x'', \lambda) \right) A_{k}(\lambda) $ is a finite linear combination of operators of the form 
\begin{align*}
\left(|\lambda|^{-\left(\frac{|\gamma_1 + \gamma_2|}{2} + |\vartheta_1|  \right)} \delta^{\gamma_1}(\lambda) \bar{\delta}^{\gamma_2}(\lambda) \partial ^{\vartheta_2}_{\lambda} \Theta_{j}(x'', \lambda) \right) A_k(\lambda), \quad \text{with } |\gamma_1 + \gamma_2| = 2 L_1.
\end{align*}
 
Let us write 
$$ K_2 (x, y) = \int_{\mathbb{R}^{n_2}} e^{-i \lambda \cdot (x^{\prime\prime} - y^{\prime\prime})} \left( \left(|\lambda|^{-\left(\frac{|\gamma_1 + \gamma_2|}{2} + |\vartheta_1|  \right)} \delta^{\gamma_1}(\lambda) \bar{\delta}^{\gamma_2}(\lambda) \partial ^{\vartheta_2}_{\lambda} \Theta_{j}(x'', \lambda) \right) A_k(\lambda) \right) (x^{\prime},y^{\prime}) \, d\lambda. $$ 

Once again, we can repeat the proof of Lemma \ref{lem:General-Kernel-estimate-lambda-singularity}. This time we have an extra $A_k(\lambda)$ in the integral of the kernel $K_2(x,y)$ which can in fact be clubbed in the calculations of $S_j$. Since $ A_{k}(\lambda) \Phi^{\lambda}_{\mu} = (2\mu_k |\lambda|)^{1/2} \Phi^{\lambda}_{\mu-e_{k}} $, an extra growth factor of $\tau^{1/2}$ will appear in the expressions analogous to \eqref{Reduced-kernel-expression-T-1} and \eqref{Reduced-kernel-expression-T-2}. So, we again have an extra growth of $2^{j/2}$ in the concerning operator norms. Overall, we get 
\begin{align} \label{grad-ker-2}
|B(x, 2^{-j/2})| \int_{\mathfrak{E}_{j} (x^{\prime})} \int_{\mathbb{R}^{n_2}} \left| \frac{K_2 (x,y)}{(|x^{\prime}| + |y^{\prime}| )^{2L_2} } \right|^{2} \, dy^{\prime} \, dy^{\prime\prime} & \lesssim 2^{-2jL} 2^j. 
\end{align}
 
This completes the proof of the claimed estimation of the kernel for $\frac{\partial}{\partial x'_{k}}\Theta_{j}^{q}(x,y)$. 

\medskip \noindent \textbf{\underline{Estimation of the kernel corresponding to $x^{\prime}_k \frac{\partial}{\partial x^{\prime \prime}_{l}}\Theta_{j}^{q}(x,y)$}:} 
Note that we can write 
\begin{align*}
& x^{\prime}_{k} \frac{\partial}{\partial x^{\prime\prime}_l} \Theta_j^{q} (x,y) \\ 
& = -i \int_{\mathbb{R}^{n_2}} e^{- i \lambda \cdot (x^{\prime\prime}-y^{\prime\prime})}\lambda_l x_{k}^{\prime} \Theta_{j}^q(x'', \lambda) (x^{\prime},y^{\prime}) \, d\lambda + \int_{\mathbb{R}^{n_2}} e^{- i \lambda \cdot (x^{\prime\prime}-y^{\prime\prime})} x_{k}^{\prime} \frac{\partial}{\partial x^{\prime\prime}_l} \Theta_{j}^q(x'', \lambda) (x^{\prime},y^{\prime}) \, d\lambda, 
\end{align*} 
and therefore, when we apply $|x^{\prime\prime}-y^{\prime\prime}|^{2L_2}$ to $x^{\prime}_{k}\frac{\partial}{\partial x^{\prime\prime}_l}\Theta_j^{q} (x,y)$, we get by Leibniz rule three types of terms, corresponding to 
\begin{itemize}
\item $\lambda_l x^{\prime}_k \left( \partial ^{\vartheta}_{\lambda} \Theta_{j}^q(x'', \lambda) \right) (x^{\prime}, y^{\prime})$, 

\item $x^{\prime}_k \left( \partial^{\vartheta - e_l}_{\lambda} \Theta_{j}^q(x'', \lambda) \right) (x^{\prime}, y^{\prime})$, 

\item $ x^{\prime}_k \left( \partial ^{\vartheta}_{\lambda} \frac{\partial}{\partial x^{\prime\prime}_l} \Theta_{j}^q(x'', \lambda) \right) (x^{\prime}, y^{\prime})$, 
\end{itemize}
with $|\vartheta|=2L_2$ . 
 
Let us first consider the term associated to 
$$\lambda_l x^{\prime}_k \left( \partial ^{\vartheta}_{\lambda} \Theta_{j}^q(x'', \lambda) \right) (x^{\prime},y^{\prime}) =  \frac{\lambda_l}{|\lambda|}|\lambda| x^{\prime}_k \left( \partial ^{\vartheta}_{\lambda} \Theta_{j}^q(x'', \lambda) \right) (x^{\prime}, y^{\prime}).$$
Since $|\lambda| x^{\prime}_k = \frac{1}{2} (A_{k}(\lambda) + A_{k}(\lambda))^* $, we have 
$$\frac{\lambda_l}{|\lambda|} |\lambda| x^{\prime}_k \left( \partial ^{\vartheta}_{\lambda} \Theta_{j}^q(x'', \lambda) \right) (x^{\prime}, y^{\prime}) = \frac{\lambda_l}{2|\lambda|} \left( A_{k}(\lambda) + A_{k}(\lambda)^* \right) \partial^{\vartheta}_{\lambda} \Theta_{j}^q(x'', \lambda) (x^{\prime}, y^{\prime}),$$ 
and these type of terms have already been dealt with in \eqref{Grad-ker-1} and \eqref{grad-ker-2}. Therefore, we get the desired result. 

Next, let us consider the term corresponding to $x^{\prime}_k \partial ^{\vartheta-e_l}_{\lambda} \Theta_{j}^{q} (x'', \lambda)$. For this term, note that because of the presence of the factor $x^{\prime}_k$, when we perform calculations leading to \eqref{kernel-decom}, we would actually have to take $(|x'|+|y'|)^{2 L_2 - 1}$ in the denominator of each of the terms $I^{1, \lambda, x''}_{k,+} (x^{\prime}, y^{\prime})$, $I^{1, \lambda, x''}_{k,-} (x^{\prime}, y^{\prime})$ and $I^{2, \lambda, x''} (x^{\prime}, y^{\prime})$. Afterwards, we repeat those calculations, and the resultant estimate would have $L_2 - \frac{1}{2}$ instead of $L_2$. This explains a maximum of extra growth of $2^j$. 

Finally, let us consider the term corresponding to $x^{\prime}_k \partial ^{\vartheta}_{\lambda} \frac{\partial}{\partial x^{\prime\prime}_l} \Theta_{j}^q(x'', \lambda)$. For this term, note that while we have already estimated terms of the form $\partial^{\vartheta}_{\lambda} \Theta_{j}^q(x'', \lambda)$, the presence of $x^{\prime}_k \frac{\partial}{\partial x^{\prime\prime}_l}$ amounts to $|\alpha_0| = |\beta_0| = 1$ in condition \eqref{def:grushin-kernel-Maucheri-hormander-cond1-L0-N0} which results in the growth factor as in other two cases.

Combining all the estimates above we obtain 
\begin{align*}
|B(x, 2^{-j/2})| \int_{\mathbb{R}^{n_1+n_2}} d(x,y)^{4L} |X_{x} \Theta_{j}(x,y)|^2 \, dy \lesssim  2^{-2jL} 2^{j}, 
\end{align*}
which is the claimed estimate \eqref{General-square-grad-x-estimate-singularity-q}. 
This completes the proof of the Lemma \ref{lem:General-Kernel-grad-estimate-x-lambda-singularity}. 
\end{proof}

Using Lemma \ref{lem:General-Kernel-grad-estimate-x} we can prove the following estimates. 

\begin{corollary} \label{Grad-ker}
Suppose $\Theta$ satisfies the conditions \eqref{def:grushin-kernel-Maucheri-hormander-cond1-L0-N0} and  \eqref{def:grushin-kernel-Maucheri-hormander-cond2-L0-N0} for $L_0 = \floor*{Q/2} + 3$ and $N_0 = 1$, then we have 
\begin{align} 
\int_{\mathbb{R}^{n_1 + n_2}} d(x,y)^{Q+1} |X_{x} \Theta_{j}(x,y)|^{2} \, dy & \lesssim 2^{j/2}. \label{Cor-grad-x-esti-higher-derivative} 
\end{align}
\end{corollary}
\begin{proof} 
Let us first assume that $Q \equiv 0~(mod ~ 4)$. In that case one can choose $L=\frac{Q}{4}$ and $L=\frac{Q+4}{4}$ in \eqref{General-square-grad-x-estimate} to conclude that 
\begin{align}\label{grad-even-even}
|B(x,2^{-j/2})| \int_{\mathbb{R}^{n_1 + n_2}}d(x,y)^{Q} |X_{x} \Theta_{j}(x,y)|^{2} \, dy & \lesssim 2^{-jQ/2} 2^{j}, \\ 
\nonumber \text{and } \quad |B(x,2^{-j/2})| \int_{\mathbb{R}^{n_1 + n_2}}d(x,y)^{Q+4} |X_{x} \Theta_{j}(x,y)|^{2} \, dy & \lesssim 2^{-j(Q+4)/4} 2^{j}.
\end{align}
	 
Interpolating the two estimates in \eqref{grad-even-even}, we get
\begin{align*}
|B(x,2^{-j/2})| \int_{\mathbb{R}^{n_1 + n_2}}d(x,y)^{Q+1} |X_{x}\Theta_{j}(x,y)|^{2} \, dy \lesssim 2^{-jQ/2} 2^{j/2}, 
\end{align*}
which (in view of the fact that $|B(x,2^{-j/2})| \gtrsim 2^{-jQ/2}$) immediately implies that 
\begin{align*}
\int_{\mathbb{R}^{n_1 + n_2}}d(x,y)^{Q+1} |X_{x}\Theta_{j}(x,y)|^{2} \, dy \lesssim 2^{j/2}. 
\end{align*}

When $Q \equiv 1~(mod ~ 4)$, we can repeat the above process with $L = \frac{Q-1}{4}$ and $L = \frac{Q+3}{4}$, whereas for $Q \equiv 2~(mod ~ 4)$, we can do the same with the help of $L = \frac{Q-2}{4}$ and $L = \frac{Q+2}{4}$, and when $Q \equiv 3~(mod ~ 4)$, we can just consider $L = \frac{Q+1}{4}$. 

This completes the proof of Corollary \ref{Grad-ker}. 
\end{proof}

Next, we discuss an analogue of Lemma \ref{lem:General-Kernel-grad-estimate-x-lambda-singularity} with gradients in $y$-variable. 

\begin{lemma} \label{lem:General-Kernel-grad-estimate-y-lambda-singularity}
Suppose $\Theta$ satisfies the conditions \eqref{def:grushin-kernel-Maucheri-hormander-cond1-L0-N0} and \eqref{def:grushin-kernel-Maucheri-hormander-cond2-L0-N0}, for some $L_0 \in \mathbb{N}$ and $N_0 = 0$. Then for all $L \leq \frac{L_0}{2}$ and $q \in \mathbb{N}$, we have 
\begin{align} 
\sup_{x\in \mathbb{R}^{n_1+n_2}} \left| B(x, 2^{-j/2}) \right| \int_{\mathbb{R}^{n_1+n_2}} d(x,y)^{4L} |X_{y} \Theta_{j}^{q}(x,y)|^2 \, dy \lesssim  2^{-2j L} 2^{j}. \label{General-square-grad-y-estimate-singularity-q}
\end{align}
\end{lemma}

\begin{proof}
 Following the terminology of the proof of Lemma \ref{lem:General-Kernel-grad-estimate-x-lambda-singularity}, for any $f \in C_c^{\infty} (\mathbb{R}^{n_1+n_2})$, we have 
\begin{align*}
& \int_{\mathbb{R}^{n_1}} \int_{\mathbb{R}^{n_2}} e^{-i\lambda\cdot x^{\prime\prime}} \frac{\partial}{\partial y^{\prime}_{k}} \left( \partial^{\vartheta_2}_{\lambda} \Theta_{j}^q(x'', \lambda) \right) (x^{\prime},y^{\prime}) f^{\lambda}(y^{\prime}) \, d\lambda \, dy' \\ 
& = - \int_{\mathbb{R}^{n_1}} \int_{\mathbb{R}^{n_2}} e^{-i\lambda\cdot x^{\prime\prime}} \left( \partial^{\vartheta_2}_{\lambda} \Theta_{j}^q(x'', \lambda) \right) (x^{\prime},y^{\prime})\frac{\partial f^{\lambda}}{\partial y'_{k}} (y) \, d\lambda \, dy^{\prime} \\ 
& = \frac{1}{2} \int_{\mathbb{R}^{n_1}} \int_{\mathbb{R}^{n_2}} e^{-i\lambda\cdot x^{\prime\prime}} \left( \partial^{\vartheta_2}_{\lambda} \Theta_{j}^q(x'', \lambda) \right) (x^{\prime},y^{\prime}) \left( \left( A_k(\lambda)^* - A_k(\lambda) \right) f^{\lambda} \right) (y^{\prime}) \, d\lambda \, dy^{\prime} \\ 
& = \frac{1}{2} \left\{ Op \left( \left( \partial^{\vartheta_2}_{\lambda} \Theta_{j}^q(x'', \lambda) \right) A_k(\lambda)^* \right) - Op \left( \left( \partial^{\vartheta_2}_{\lambda} \Theta_{j}^q(x'', \lambda) \right) A_{k}(\lambda) \right) \right\} f(x).
\end{align*}
But we have already dealt with these type of terms leading to estimate \eqref{grad-ker-2}. Hence we get the desired estimate for kernels corresponding to $X_{y} \Theta_{j}^{q}(x,y)$.

On the other hand, note that 
\begin{align*}
y^{\prime}_k \frac{\partial}{\partial y''_{l}}  \Theta_{j}^{q}(x,y) = i \int_{\mathbb{R}^{n_2}} e^{- i \lambda \cdot(x^{\prime \prime}-y^{\prime \prime})} (y^{\prime}_{k} \lambda_l ) \Theta_j^{q}(x^{\prime}, y^{\prime}) \, d\lambda,
\end{align*}
and therefore, when we apply $|x^{\prime}- y^{\prime}|^{2L_2}$ to $(y^{\prime}_{k} \lambda_l ) \Theta_j^{q}(x^{\prime}, y^{\prime})$, we get by Leibniz rule two types of terms, corresponding to 
\begin{itemize}
\item  $(y^{\prime}_{k} \lambda_l ) \left( \partial^{\vartheta}_{\lambda} \Theta_j^{q}(x'', \lambda) \right) (x^{\prime}, y^{\prime})$, 

\item $ y^{\prime}_{k} \left( \partial^{\vartheta-e_l}_{\lambda} \Theta_j^{q}(x'', \lambda) \right) (x^{\prime}, y^{\prime})$, 
\end{itemize} 
with $|\vartheta|=2L_2$. 

Let us first analyse the term $(y^{\prime}_{k} \lambda_l ) \left( \partial^{\vartheta}_{\lambda} \Theta_{j}^q(x'', \lambda) \right) (x^{\prime}, y^{\prime})$. For $f \in C_c^{\infty} (\mathbb{R}^{n_1+n_2})$, we have 
\begin{align*}
& \int_{\mathbb{R}^{n_1}} \int_{\mathbb{R}^{n_2}} e^{-i\lambda\cdot x^{\prime\prime}} y^{\prime}_{k} \lambda_l \left( \partial^{\vartheta}_{\lambda} \Theta_{j}^q(x'', \lambda) \right) (x^{\prime},y^{\prime}) f^{\lambda}(y^{\prime}) \, d\lambda \, dy' \\
& = \frac{1}{2} \int_{\mathbb{R}^{n_1}} \int_{\mathbb{R}^{n_2}} e^{-i\lambda\cdot x^{\prime\prime}} \left( \partial^{\vartheta}_{\lambda} \Theta_{j}^q(x'', \lambda) \right) (x^{\prime}, y^{\prime}) \frac{\lambda_l}{|\lambda|} \left( \left( A_k(\lambda)^{*} + A_k(\lambda) \right) f^{\lambda} \right) (y^{\prime}) \, d\lambda \, dy^{\prime} \\ 
& = \frac{1}{2} \left\{ Op \left( \left( \partial^{\vartheta}_{\lambda} \Theta_{j}^q(x'', \lambda) \right) \frac{\lambda_l}{|\lambda|} A_{k}(\lambda)^{*} \right) - Op \left( \left( \partial^{\vartheta}_{\lambda} \Theta_{j}^q(x'', \lambda) \right) \frac{\lambda_l}{|\lambda|}  A_{k}(\lambda) \right) \right\} f(x).
\end{align*} 

For operators with the expression as above, we can essentially repeat the analysis leading to \eqref{grad-ker-2}. 

Now, for the term $y^{\prime}_{k} \left( \partial^{\vartheta-e_l}_{\lambda} \Theta_j^{q}(x'', \lambda) \right) (x^{\prime}, y^{\prime})$, let us write
\begin{align*}
y^{\prime}_{k} \left( \partial^{\vartheta-e_l}_{\lambda} \Theta_j^{q}(x'', \lambda) \right) (x^{\prime}, y^{\prime}) & = - (x^{\prime}_{k} - y^{\prime}_{k}) \left( \partial^{\vartheta-e_l}_{\lambda} \Theta_j^{q} (x'', \lambda) \right) (x^{\prime}, y^{\prime}) \\ 
& \quad + x^{\prime}_{k} \left( \partial^{\vartheta-e_l}_{\lambda} \Theta_j^{q} (x'', \lambda) \right) (x^{\prime}, y^{\prime}). 
\end{align*}
Note that $(x^{\prime}_{k}-y^{\prime}_{k})$ corresponds to the action of the operator $D^{e_k}$ on the operator. Therefore, estimating the kernel associated to the term $(x^{\prime}_{k}-y^{\prime}_{k}) \left( \partial^{\vartheta-e_l}_{\lambda} \Theta_j^{q}(x'', \lambda) \right) (x^{\prime}, y^{\prime})$ is equivalent to estimating the kernel associated to the term $\left( D^{e_k} \partial^{\vartheta-e_l}_{\lambda} \Theta_j^{q}(x'', \lambda) \right) (x^{\prime}, y^{\prime})$, which can be done using the methodology of Lemma \ref{lem:General-Kernel-estimate}. Here we will not have any extra growth of $2^j$, because even though we have one less derivative in $\lambda$-variable but we also have one extra non-commutative derivative. We have already discussed the estimation of kernel associated to the term $x^{\prime}_{k} \left( \partial^{\vartheta-e_l}_{\lambda} \Theta_j^{q}(x'', \lambda) \right) (x^{\prime}, y^{\prime})$ in the previous Lemma. Combining all the above estimates we get \eqref{General-square-grad-y-estimate-singularity-q}, and this completes the proof of Lemma \ref{lem:General-Kernel-grad-estimate-y-lambda-singularity}. 
\end{proof}

As earlier, an analogue of Lemma \ref{lem:General-Kernel-grad-estimate-y-lambda-singularity} is true when one replaces $\Theta_{j}^{q}(x,y)$ by $\Theta_{j} (x,y)$. 

\subsection{\texorpdfstring{$L^{\infty}$}{}-weighted Plancherel (gradient) estimates} \label{subsec:mauceri-type-L-infty-weighted-kernel-estimates} In this subsection, we shall see that by assuming more (nearly double) number of derivatives in the conditions of the type \eqref{def:grushin-kernel-Maucheri-hormander-cond1-L0-N0} and \eqref{def:grushin-kernel-Maucheri-hormander-cond2-L0-N0}, one can get $L^{\infty}$-weighted Plancherel estimates. 

\begin{lemma} \label{lem:General-Sup-kernel-estimate}
Suppose $\Theta$ satisfies the conditions \eqref{def:grushin-kernel-Maucheri-hormander-cond1-L0-N0} and \eqref{def:grushin-kernel-Maucheri-hormander-cond2-L0-N0} for some $L_0 \in \mathbb{N}$ and $N_0 = 0$. Then for all $L\leq \frac{L_0}{2}$, we have the following estimate
\begin{align} \label{General-Sup-estimate} 
d(x,y)^{2L} |\Theta_{j} (x,y)| \lesssim_{L_0} 2^{-jL} \left| B(x, 2^{-j/2}) \right|^{-1/2} \left| B(y, 2^{-j/2}) \right|^{-1/2}. 
\end{align} 
\end{lemma}

Following arguments given in the previous subsection (see the explanation on how Lemma \ref{lem:General-Kernel-estimate} was deduced from Lemma \ref{lem:General-Kernel-estimate-lambda-singularity}), Lemma \ref{lem:General-Sup-kernel-estimate} is an immediate consequence of the following result. 

\begin{lemma} \label{lem:General-Sup-kernel-estimate-lambda-singularity}
Suppose $\Theta$ satisfies the conditions \eqref{def:grushin-kernel-Maucheri-hormander-cond1-L0-N0} and \eqref{def:grushin-kernel-Maucheri-hormander-cond2-L0-N0} for some $L_0 \in \mathbb{N}$ and $N_0 = 0$. Then for all $L\leq \frac{L_0}{2}$ and $q \in \mathbb{N}$, we have 
\begin{align} \label{General-Sup-estimate-singularity-q} 
d(x,y)^{2L} |\Theta_{j}^{q}(x,y)| \lesssim_{L_0} 2^{-jL} \left| B(x, 2^{-j/2}) \right|^{-1/2} \left| B(y, 2^{-j/2}) \right|^{-1/2}. 
\end{align} 
\end{lemma}
\begin{proof} 
We repeat the proof of Lemma \ref{lem:General-Kernel-estimate-lambda-singularity}, and break the estimate into two regions $I$ and $II$ given by 
\eqref{eq:first-break-lem:General-Kernel-estimate-lambda-singularity}. The first notable difference from the proof of Lemma \ref{lem:General-Kernel-estimate-lambda-singularity} is that we do not write $\Theta_{j}^q(x'', \lambda)$ as $\Upsilon^{q}(\lambda) \Theta_{j}(x'', \lambda)$. Instead, we stick to $\Theta_{j}^q(x'', \lambda)$ itself. Therefore, for term $I$, we need to analyse operators of the following form: 
$$ \left( D^{\alpha_1} \partial_{\lambda}^{\vartheta_1} \Theta^{q}(x'', \lambda) \right) \left( D^{\alpha_2} \partial_{\lambda}^{\vartheta_2}S_{j}(\lambda) \right), $$
where $|\alpha_1 + \alpha_2| = 2 L_1$ and $|\vartheta_1 + \vartheta_2| = 2 L_2$. 

As earlier, we take one such term and decompose it as follows: 
\begin{align} \label{est:Mauceri-type-lambda-singularity-decompose-pointwise1}
& \left( D^{\alpha_1} \partial_{\lambda}^{\vartheta_1}\Theta^{q}(x'', \lambda) \right) \left( D^{\alpha_2} \partial_{\lambda}^{\vartheta_2}S_{j}(\lambda) \right) \\ 
\nonumber & \quad = \sum_{j^{\prime}=0}^{\infty} \left(|\lambda|^{-|\vartheta_2|} D^{\alpha_1} \partial_{\lambda}^{\vartheta_1} \Theta^{q}(x'', \lambda) \right) \chi_{j^{\prime}}(\lambda) \left( |\lambda|^{|\vartheta_2|} D^{\alpha_2} \partial_{\lambda}^{\vartheta_2}S_{j}(\lambda) \right), 
\end{align}
and by Lemma \ref{lem:finite-shifts} we know that the above sum will survive only when $|j-j'| \leq C L$ for some constant $C > 0$ which depends only on $n_1$ and $n_2$. 

Note also that $ Op \left( \left( \partial_{\lambda}^{\tilde{\vartheta}} \Upsilon^{q}(\lambda) \right) \chi_{j^{\prime}} (\lambda) \right)$ is a Hilbert-Schmidt operator on $L^2 \left( \mathbb{R}^{n_1 + n_2} \right)$ for any $q \in \mathbb{N}$ and $\tilde{\vartheta} \in \mathbb{N}^{n_2}$. Now, with $\mathcal{D}_k^+$ and $\mathcal{D}_k^-$ as in \eqref{def:x-prime-sets-Mauceri-calc-1}, we initiate the analysis of term $I$ as in the proof of Lemma \ref{lem:General-Kernel-estimate-lambda-singularity}.

It suffices to estimate terms of the following form
\begin{align} \label{kernel-decom-2}
&\frac{{\mathbbm{1}}_{j,x^{\prime}}(y^{\prime})}{(|x'|+|y'|)^{2 L_2}} \int_{\mathbb{R}^{n_1+n_2}} Op \left( \left( |\lambda|^{-|\vartheta_2|} D^{\alpha_1} \partial_{\lambda}^{\vartheta_1}\Theta^{q}(x'', \lambda) \right) \chi_{j^{\prime}}(\lambda) \right) (x, z) \\
\nonumber & \quad \quad \quad \quad \quad \quad \quad \quad \quad \quad Op \left( |\lambda|^{|\vartheta_2|} D^{\alpha_2} \partial_{\lambda}^{\vartheta_2} S_{j}(\lambda) \right) (z, y) \, dz \\ 
\nonumber & \quad = : \sum_{k=1}^{n_1} \mathcal{I}^{1}_{k,+} (x,y) + \sum_{k=1}^{n_1} \mathcal{I}^{1}_{k,-} (x,y) + \mathcal{I}^{2} (x,y), 
\end{align}
where 
\begin{align*} 
\mathcal{I}^{1}_{k,+} (x,y) & = \frac{{\mathbbm{1}}_{j,x^{\prime}}(y^{\prime})}{(|x'|+|y'|)^{2 L_2}} \int_{\mathbb{R}^{n_2}} \int_{\mathfrak{E}_{j} (x') \cap \mathcal{D}_k^+} \frac{ Op \left( \left( |\lambda|^{-|\vartheta_2|} D^{\alpha_1} \partial_{\lambda}^{\vartheta_1} \Theta^{q}(x'', \lambda) \right) \chi_{j^{\prime}}(\lambda) \right) (x, z)}{(|x'|+z_k')^{2 L_2}} \\ 
& \quad \quad \quad \quad \quad \quad \quad \quad \quad \quad \quad \quad  (|x'|+z_k')^{2 L_2} Op \left( |\lambda|^{|\vartheta_2|} D^{\alpha_2} \partial_{\lambda}^{\vartheta_2}S_{j}(\lambda) \right) (z, y) \, dz' \, dz'', \\ 
\mathcal{I}^{1}_{k,-} (x,y) & = \frac{{\mathbbm{1}}_{j,x^{\prime}}(y^{\prime})}{(|x'|+|y'|)^{2 L_2}} \int_{\mathbb{R}^{n_2}} \int_{\mathfrak{E}_{j} (x') \cap \mathcal{D}_k^-} \frac{Op \left( \left( |\lambda|^{-|\vartheta_2|} D^{\alpha_1} \partial_{\lambda}^{\vartheta_1}\Theta^{q}(x'', \lambda) \right) \chi_{j^{\prime} }(\lambda) \right) (x, z)}{(|x'|-z_k')^{2 L_2}} \\ 
& \quad \quad \quad \quad \quad \quad \quad \quad \quad \quad \quad \quad  (|x'|-z_k')^{2 L_2} Op \left( |\lambda|^{|\vartheta_2|} D^{\alpha_2} \partial_{\lambda}^{\vartheta_2}S_{j}(\lambda) \right) (z, y) \, dz' \, dz'', \\ 
\mathcal{I}^{2} (x,y) & =  \frac{{\mathbbm{1}}_{j,x^{\prime}}(y^{\prime})}{(|x'|+|y'|)^{2 L_2}} \int_{\mathbb{R}^{n_2}} \int_{\mathfrak{E}_{j} (x')^c} Op \left( \left( |\lambda|^{-|\vartheta_2|} D^{\alpha_1}(\partial_{\lambda}^{\vartheta_1}\Theta^{q}(x'', \lambda) \right) \chi_{j^{\prime}}(\lambda) \right) (x, z) \\ 
& \quad \quad \quad \quad \quad \quad \quad \quad \quad \quad \quad \quad Op \left( |\lambda|^{|\vartheta_2|} D^{\alpha_2} \partial_{\lambda}^{\vartheta_2}S_{j}(\lambda) \right) (z, y) \, dz' \, dz''.
\end{align*} 

\medskip \textbf{\underline{Analysis of the term $\mathcal{I}^{1}_{k,+} (x,y) $ in \eqref{kernel-decom-2}}:} By Cauchy-Schwartz inequality  
\begin{align} \label{kernel-decom-2-Cauchy-Schwarz}
& \left| \mathcal{I}^{1}_{k,+} (x,y) \right| \\ 
\nonumber & \quad \leq \left( \int_{\mathbb{R}^{n_2}} \int_{\mathfrak{E}_{j} (x') \cap \mathcal{D}_k^+} \left| \frac{ Op \left( \left( |\lambda|^{-|\vartheta_2|} D^{\alpha_1} \partial_{\lambda}^{\vartheta_1}\Theta^{q}(x'', \lambda) \right) \chi_{j^{\prime}}(\lambda) \right) (x, z) }{(|x'|+z_k')^{2 L_2}}\right|^2 dz' \, dz'' \right)^{1/2} \\
\nonumber & \quad \quad \times \left( \int_{\mathbb{R}^{n_2}} \int_{\mathfrak{E}_{j} (x') \cap \mathcal{D}_k^+}  \left| \frac{(|x'|+z_k')^{2 L_2}}{(|x'|+|y'|)^{ 2L_2}} Op \left( |\lambda|^{|\vartheta_2|} D^{\alpha_2} \partial_{\lambda}^{\vartheta_2}S_{j}(\lambda) \right) (z, y) \right|^2 \, dz' \, dz'' \right)^{1/2}. 
\end{align}
 
It is at this stage that in the first term on the right hand side of the above inequality we perform analysis analogous to the one done in \eqref{est:Mauceri-type-lambda-singularity}, namely, we apply the Euclidean Plancherel theorem in $z''$-variable, and apply Leibniz's differentiation rule in $\lambda$-derivative on $\Upsilon^{q}(\lambda) \Theta (x'', \lambda)$, to get 
\begin{align}
& \int_{\mathbb{R}^{n_2}} \int_{\mathfrak{E}_{j} (x') \cap \mathcal{D}_k^+} \left| \frac{ |\lambda|^{-|\vartheta_2|} Op \left( \left( D^{\alpha_1} \partial_{\lambda}^{\vartheta_1} \Theta^{q}(x'', \lambda) \right) \chi_{j^{\prime}}(\lambda) \right) (x, z)}{(|x'|+z_k')^{2 L_2}} \right|^2 dz' \, dz'' \\ 
\nonumber & \lesssim \sum_{\vartheta_3 + \vartheta_4 = \vartheta_1}  \int_{\mathfrak{E}_{j} (x') } \int_{\mathbb{R}^{n_2}} \left| \frac{ |\lambda|^ {-|\vartheta_2|} \left( \partial^{\vartheta_3}_{\lambda}  \Upsilon^{q}(\lambda) \right) \left( \left( D^{\alpha_1} \partial_{\lambda}^{\vartheta_4 } \Theta (x'', \lambda) \right) \chi_{j^{\prime}}(\lambda) \right) (x^{\prime}, z^{\prime})}{(|x'|+|z'|)^{2 L_2}}\right|^2 d\lambda \, dz^{\prime} \\ 
\nonumber & \lesssim \sum_{\vartheta_3 + \vartheta_4 = \vartheta_1}  \int_{\mathfrak{E}_{j} (x')} \int_{\mathbb{R}^{n_2}} \left| \frac{ |\lambda|^ {-|\vartheta_2 + \vartheta_3|} \left( \left( D^{\alpha_1} \partial_{\lambda}^{\vartheta_4 } \Theta (x'', \lambda) \right) \chi_{j^{\prime}}(\lambda) \right) (x^{\prime}, z^{\prime})}{(|x'|+|z'|)^{2 L_2}}\right|^2  d\lambda \, dz^{\prime}. 
\end{align}
Then, as earlier, since $D^{\alpha_1} \partial_{\lambda}^{\vartheta_4} \Theta(x'', \lambda)$ can be written as a finite linear combination of operators of the form $|\lambda|^{-\frac{|\gamma_1 + \gamma_2|}{2}} \delta^{\gamma_1}(\lambda) \bar{\delta}^{\gamma_2}(\lambda) \partial_{\lambda}^{\vartheta_4} \Theta(x'', \lambda)$ with $\gamma_1 + \gamma_2 = \alpha_1$, it suffices to estimate 
$$ \int_{\mathfrak{E}_{j} (x') } \int_{\mathbb{R}^{n_2}} \left| \frac{ \left( |\lambda|^{-(|\vartheta_2 + \vartheta_3| + \frac{|\gamma_1 + \gamma_2|}{2})} \left( \delta^{\gamma_1}(\lambda) \bar{\delta}^{\gamma_2}(\lambda) \partial_{\lambda}^{\vartheta_4} \Theta(x'', \lambda) \right) \chi_{j^{\prime}}(\lambda) \right) (x^{\prime}, z^{\prime})}{(|x'|+|z'|)^{2 L_2}}\right|^2  d\lambda \, dz^{\prime}, $$ 
and therefore using condition \eqref{def:grushin-kernel-Maucheri-hormander-cond1-L0-N0}, we get 
\begin{align}\label{Given-condition}
& |B(x,2^{-j/2})| \int_{\mathfrak{E}_{j} (x')} \int_{\mathbb{R}^{n_2}} \left| \frac{\left( |\lambda|^{-(|\vartheta_2 + \vartheta_3|)} \left( D^{\alpha_1} \partial_{\lambda}^{\vartheta_4} \Theta(x'', \lambda) \right) \chi_{j^{\prime}}(\lambda) \right) (x^{\prime}, z^{\prime})}{(|x'|+|z'|)^{2 L_2}} \right|^2 d\lambda \, dz^{\prime} \\ 
\nonumber & \quad \lesssim 2^{-j(|\gamma_1 + \gamma_2| + 2 L_2)} = 2^{-j(|\alpha_1| + 2 L_2)}. 
\end{align}

On the other hand, instead of the operator norm of the type $A^{1, x}_{k, +}$ (as in \eqref{Op-norm-1}), this time we have to compute the weighted Plancherel estimate for the second term on the right hand side of the inequality \eqref{kernel-decom-2-Cauchy-Schwarz}. For the same, note that we have already seen in the proof of Lemma \ref{lem:General-Kernel-estimate-lambda-singularity} that ${\mathbbm{1}}_{j,x^{\prime}}(y^{\prime}) \frac{(|x'|+z_k')^{2 L_2}}{(|x'|+|y'|)^{2 L_2}} Op \left( |\lambda|^{|\vartheta_2|} D^{\alpha_2} \partial_{\lambda}^{\vartheta_2} S_{j}(\lambda) \right) (z, y)$ can be written as a finite linear combination of terms of the following form
\begin{align}\label{kernel-express-final}
& {\mathbbm{1}}_{j,x^{\prime}}(y^{\prime}) \frac{|x^{\prime}|^{b_1} {y^{\prime}_{k}}^{b_3}}{(|x'|+|y'|)^{2 L_2}} \int\limits_{[0,1]^{N_1} \times \Omega^{N_2} \times [0,1]^{|\tilde{\nu}|}} \int_{\mathbb{R}^{n_2}} e^{-i \lambda \cdot (z''-y'')} \left( |\lambda| {y^{\prime}} \right)^{\alpha_3} |\lambda|^{|\vartheta_2| - |\alpha_3|} \, \mathfrak{A}_l(\lambda) \\ 
\nonumber & \quad \sum_{\mu} C_{\mu, \vec{c}} \left(\tau^{\frac{1}{2} \theta_1} \partial_\tau^{\theta_2} \psi_{j}\right) ((2\mu + \tilde{1} + \vec{c}(\omega)) |\lambda|) \Phi_{\mu}^{\lambda}(z^{\prime}) \Phi_{\mu + \tilde{\mu}}^{\lambda}(y^{\prime}) g(\omega) \, d\lambda \, d \omega, 
\end{align}
where $N_1, N_2, | \alpha_3 | \leq |\vartheta_2| - l$, $ |\mathcal{\theta}_1| \leq |\mathcal{\theta}_2| \leq |\tilde{\nu}| + 2 |\vartheta_2| - 2 l$, $|\mathcal{\theta}_2| - \frac{|\mathcal{\theta}_1|}{2} = \frac{|\tilde{\nu}|}{2} + |\vartheta_2| - l - \frac{|\alpha_3|}{2} \geq (|\tilde{\nu}| + |\vartheta_2| - l) / 2$, and $\mathfrak{A}_l$ is a continuous function on $\mathbb{R}^{n_2} \setminus \{0\}$ which is homogeneous of degree $-l$. Also, $b_1$, $b_3$ and $\tilde{\nu}$ should be understood as in \eqref{Reduced-kernel-expression-T-1}. 

Using the expression of \eqref{kernel-express-final}, we get by the Plancherel theorem in $z''$-variable 
\begin{align}\label{plancherel-Sj}
& |B(y, 2^{-j/2})| \int_{\mathbb{R}^{n_2}} \int_{\mathfrak{E}_{j} (x') \cap \mathcal{D}_k^+} \left| {\mathbbm{1}}_{j,x^{\prime}}(y^{\prime}) \frac{(|x'|+z_k')^{2 L_2}}{(|x'|+|y'|)^{ 2 L_2}} Op \left( |\lambda|^{|\vartheta_2|} D^{\alpha_2}  \partial_{\lambda}^{\vartheta_2} S_{j}(\lambda) \right) (z, y) \right|^2 dz' \, dz'' \\ 
\nonumber & \quad = |B(y, 2^{-j/2})| \int_{\mathbb{R}^{n_2}} \int_{\mathfrak{E}_{j} (x') \cap \mathcal{D}_k^+} \left| {\mathbbm{1}}_{j,x^{\prime}}(y^{\prime}) \frac{(|x'|+z_k')^{2 L_2}}{(|x'|+|y'|)^{ 2L_2}} \left( |\lambda|^{|\vartheta_2|} D^{\alpha_2}  \partial_{\lambda}^{\vartheta_2}S_{j}(\lambda) \right) (z', y') \right|^2 d\lambda \, dz \\ 
\nonumber & \quad \lesssim 2^{-j |\alpha_2|}. 
\end{align}

Finally, estimates \eqref{Given-condition} and \eqref{plancherel-Sj} together imply that 
\begin{align*}
\left| \mathcal{I}^{1}_{k,+} (x, y) \right| & \lesssim 2^{-j(|\alpha_1|+2L_2)/2} |B(x,2^{-j/2})|^{-1/2} 2^{-j|\alpha_2|/2} |B(y, 2^{-j/2})|^{-1/2}\\
& = 2^{-jL} |B(x,2^{-j/2})|^{-1/2} |B(y, 2^{-j/2})|^{-1/2}.
\end{align*}

The analysis of the terms $\mathcal{I}^{1}_{k,-} (x,y)$, $\mathcal{I}^{2} (x,y)$, and of the term $II$ could be done using similar ideas, and we leave those details. This completes the proof of Lemma \ref{lem:General-Sup-kernel-estimate-lambda-singularity}. 
\end{proof}

\begin{lemma} \label{General-Sup-kernel-grad-estimate-x-variable}
Suppose $\Theta$ satisfies the conditions \eqref{def:grushin-kernel-Maucheri-hormander-cond1-L0-N0} and \eqref{def:grushin-kernel-Maucheri-hormander-cond2-L0-N0} for some $L_0 \in \mathbb{N}$ and $N_0 = 1$. Then for all $L\leq \frac{L_0 - 1}{2}$, we have 
\begin{align*}
d(x,y)^{2L} |X_{x} \Theta_{j}(x,y)| & \lesssim_{L_0} 2^{-jL} 2^{j/2} \left| B(x, 2^{-j/2}) \right|^{-1/2} \left| B(y, 2^{-j/2}) \right|^{-1/2}.
\end{align*}
\end{lemma}
\begin{proof}
One can prove Lemma \ref{General-Sup-kernel-grad-estimate-x-variable} following the the methodology of proofs of Lemmas \ref{lem:General-Sup-kernel-estimate-lambda-singularity} and  \ref{lem:General-Kernel-grad-estimate-x-lambda-singularity} , so we omit it's proof. 
\end{proof}

Similar estimates could be proved for $y$-gradients with $N_0 = 0$ itself. More precisely, 
\begin{lemma} \label{General-Sup-kernel-grad-estimate-y-variable}
Suppose $\Theta$ satisfies the conditions \eqref{def:grushin-kernel-Maucheri-hormander-cond1-L0-N0} and \eqref{def:grushin-kernel-Maucheri-hormander-cond2-L0-N0} for some $L_0 \in \mathbb{N}$ and $N_0 = 0$. Then for all $L\leq \frac{L_0}{2}$, we have 
\begin{align*}
d(x,y)^{2L} |X_{y} \Theta_{j}(x,y)| & \lesssim_{L_0} 2^{-jL} 2^{j/2} \left| B(x, 2^{-j/2}) \right|^{-1/2} \left| B(y, 2^{-j/2}) \right|^{-1/2}. 
\end{align*}
\end{lemma}

\subsection{Proofs of Theorems \ref{thm:Mauceri-kernel-a=0-weak-11}, \ref{thm:Mauceri-kernel-a=0} and \ref{thm:Mauceri-kernel-a=0-more-derivative}} \label{subsec:proofs-mauceri-type-theorems} 

In this subsection we explain how Theorems \ref{thm:Mauceri-kernel-a=0-weak-11}, \ref{thm:Mauceri-kernel-a=0} and \ref{thm:Mauceri-kernel-a=0-more-derivative} would follow from the weighted kernel estimates that we have established in the previous two subsections. But, before doing that, let us remark that here our assumptions on the choice of $L_0$ are $1$ or $2$ more than those required for results stated in Subsection \ref{subsec:results-for-Grsuhin-pseudo}. As also pointed out in Subsection \ref{subsec:intro-methodology-proofs}, the extra assumption is needed due to the method of our proof. For example, in order to apply estimate \eqref{General-square-estimate} of Lemma \ref{lem:General-Kernel-estimate}, we have to take $L \in \mathbb{N}$. This is indeed required as we have already seen that our proof of estimates of the type \eqref{General-square-estimate} is built on direct calculations pertaining to the action of the distance function $d(x,y)$, which forces us to work with multiples of 4 in powers of $d(x,y)$. Similar remark is applicable for other weighted kernel estimates too.

\begin{proof}[Proof of Theorem \ref{thm:Mauceri-kernel-a=0-weak-11}]
In view of Theorem \ref{thm:weak-type-bound-operator}, it suffices to prove that the operator $\Theta$ satisfies estimates \eqref{cond:General-hypo-sup} and \eqref{cond:General-hypo-y-grad-sup} for $R_0= Q + 2$. So we decompose the operator $\Theta$ as $\Theta = \sum_{j} \Theta_j$, where $\Theta_j$'s are defined in \eqref{def:operator-countable-break-smooth}. 

Now we have from Lemma \ref{lem:General-Sup-kernel-estimate} that for all natural numbers $L \leq (Q+2) / 2$, 
\begin{align*} 
d(x,y)^{2L} |\Theta_{j} (x,y)| \lesssim 2^{-jL} \left| B(x, 2^{-j/2}) \right|^{-1/2} \left| B(y, 2^{-j/2}) \right|^{-1/2}. 
\end{align*}
Therefore, by convexity, for all real numbers $0 \leq \mathfrak{r}\leq Q +2$, we get 
\begin{align*} 
d(x,y)^{\mathfrak{r}} |\Theta_{j} (x,y)| \lesssim 2^{-j\mathfrak{r}/2} \left| B(x, 2^{-j/2}) \right|^{-1/2} \left| B(y, 2^{-j/2}) \right|^{-1/2}, 
\end{align*}
which implies that condition \eqref{cond:General-hypo-sup} is satisfied for $R_0= Q + 2$. 

Similarly we get from Lemma \ref{General-Sup-kernel-grad-estimate-y-variable} that for all real numbers $0 \leq \mathfrak{r}\leq Q +2$, we have
\begin{align*}
d(x,y)^{\mathfrak{r}} |X_{y} \Theta_{j}(x,y)| & \lesssim 2^{-j\mathfrak{r}/2} 2^{j/2} \left| B(x, 2^{-j/2}) \right|^{-1/2} \left| B(y, 2^{-j/2}) \right|^{-1/2}. 
\end{align*}
which imply that conditions \eqref{cond:General-hypo-y-grad-sup} is satisfied for $R_0= Q + 2$. 

Hence, keeping Remark \ref{choice:partition-of-unity} in mind, Theorem \ref{thm:Mauceri-kernel-a=0-weak-11} follows from Theorem \ref{thm:weak-type-bound-operator}. 
\end{proof}

\begin{proof}[Proof of Theorem \ref{thm:Mauceri-kernel-a=0}]
We are given that conditions \eqref{def:grushin-kernel-Maucheri-hormander-cond1-L0-N0} and \eqref{def:grushin-kernel-Maucheri-hormander-cond2-L0-N0} hold for $L_0= \floor*{Q/2} + 3$ and $N_0$ =1, we can make use of Lemmas \ref{lem:General-Kernel-estimate} and \ref{lem:General-Kernel-grad-estimate-x} to get that for all real numbers $0 \leq \mathfrak{r}\leq \floor*{Q/2} +2$, 
\begin{align*}
|B(x, 2^{-j/2})| \int_{\mathbb{R}^{n_1 + n_2}} d(x,y)^{2\mathfrak{r}} |\Theta_{j}(x,y)|^2 \, dy & \lesssim 2^{-j\mathfrak{r}},\\
|B(x, 2^{-j/2})| \int_{\mathbb{R}^{n_1+n_2}} d(x,y)^{2\mathfrak{r}} |X_{x} \Theta_{j} (x,y)|^2 \, dy & \lesssim 2^{-j \mathfrak{r}}2^{j}.
\end{align*}
Thus, we see that the operator $\Theta$ satisfies conditions \eqref{cond:General-hypo} and \eqref{cond:General-hypo-grad} for $R_0= \floor*{Q/2} +2$. With this, Theorem \ref{thm:Mauceri-kernel-a=0} follows from Theorem \ref{thm:main-sparse} and Remark  \ref{choice:partition-of-unity}.
\end{proof}

\begin{proof}[Proof of Theorem \ref{thm:Mauceri-kernel-a=0-more-derivative}]
Since we have assumed conditions \eqref{def:grushin-kernel-Maucheri-hormander-cond1-L0-N0} and \eqref{def:grushin-kernel-Maucheri-hormander-cond2-L0-N0} for $L_0= Q + 3$ and $N_0$ =1, Lemmas \ref{lem:General-Sup-kernel-estimate},   \ref{General-Sup-kernel-grad-estimate-x-variable} and \ref{General-Sup-kernel-grad-estimate-y-variable} are applicable, and we get that for all real numbers $0 \leq \mathfrak{r}\leq Q +2$, 
\begin{align*}
d(x,y)^{\mathfrak{r}} |\Theta_{j} (x,y)| & \lesssim 2^{-j\mathfrak{r}/2} \left| B(x, 2^{-j/2}) \right|^{-1/2} \left| B(y, 2^{-j/2}) \right|^{-1/2}, \\
d(x,y)^{\mathfrak{r}} |X_{x} \Theta_{j}(x,y)| & \lesssim 2^{-j\mathfrak{r}/2} 2^{j/2} \left| B(x, 2^{-j/2}) \right|^{-1/2} \left| B(y, 2^{-j/2}) \right|^{-1/2},\\ 
d(x,y)^{\mathfrak{r}} |X_{y} \Theta_{j}(x,y)| & \lesssim 2^{-j\mathfrak{r}/2} 2^{j/2} \left| B(x, 2^{-j/2}) \right|^{-1/2} \left| B(y, 2^{-j/2}) \right|^{-1/2}.
\end{align*}
That is, $\Theta$ satisfies conditions \eqref{cond:General-hypo-sup}, \eqref{cond:General-hypo-y-grad-sup}, and \eqref{cond:General-hypo-grad-sup} for $R_0 = Q + 2$, and therefore Theorem \ref{thm:Mauceri-kernel-a=0-more-derivative} follows from Theorem \ref{thm:main-sparse-more-derivative} and Remark \ref{choice:partition-of-unity}.
\end{proof} 

\section{Proofs of results of Subsection \ref{subsec:main-results-joint-functional-calculus}} \label{Sec:Shifted-Grushin-pseudo-multipliers} 

In this section, we shall provide proofs of Theorems \ref{thm:joint-shift-multi-weak-type}, \ref{thm:joint-shift-multi-sparse-less-derivative} and \ref{thm:joint-shift-multi-sparse-more-derivative}. We shall actually show that the operator family $\Theta(x'', \lambda)$ given by \eqref{def:theta-lambda-example-with-shifts} satisfies the following conditions for appropriate values of $L_0$ and $N_0$: 
\newcommand{\KCKCaaREV}{
\begin{align} \label{def:grushin-kernel-Maucheri-hormander-cond1-L0-N0-for-symbols} 
& |B(x, 2^{-j/2})| \int_{\mathfrak{E}_{j} (x^{\prime})} \int_{\mathbb{R}^{n_2}}
\left| \frac{ \left( \left( \widetilde{\mathcal{L}(\lambda)}_{\vartheta_1, \vartheta_2}^{\nu} {x'}^{\alpha_0} \partial^{\beta_0}_{x''} \Theta(x'', \lambda) \right) \chi_{j}(\lambda)\right)(x^{\prime}, y^{\prime})}{(|x^{\prime}|+|y^{\prime}|)^{|\vartheta_1 + \vartheta_2|}} \right|^2 d\lambda \, dy^{\prime} \tag{$\widetilde{KC1}$-$L_0$-$N_0$} \\ 
\nonumber & \quad \lesssim_{L_0, N_0} 2^{j |\beta_0|} 2^{-j\left( |\nu| + |\vartheta_1 + \vartheta_2| \right)}, \\ 
\nonumber & \textup{for all } |\nu| + |\vartheta_1 + \vartheta_2| + |\beta_0| \leq L_0 \, \textup{ and } |\alpha_0|= |\beta_0| \leq N_0, \\ 
\label{def:grushin-kernel-Maucheri-hormander-cond2-L0-N0-for-symbols}
& |B(x, 2^{-j/2})| \int_{\mathfrak{E}_{j} (x^{\prime})^c} \int_{\mathbb{R}^{n_2}} \left| \left( \left( \widetilde{\mathcal{L}(\lambda)}_{\vartheta_1, \vartheta_2}^{\nu} {x'}^{\alpha_0} \partial^{\beta_0}_{x''} \Theta(x'', \lambda) \right)  \chi_{j}(\lambda) \right) (x^{\prime}, y^{\prime}) \right|^2 d\lambda \, dy' \tag{$\widetilde{KC2}$-$L_0$-$N_0$} \\ 
\nonumber & \quad \lesssim_{L_0, N_0} 2^{j |\beta_0|} 2^{-j\left( |\nu| + 2 |\vartheta_1 + \vartheta_2| \right)}, \\ 
\nonumber & \textup{for all } |\nu| + 2|\vartheta_1 + \vartheta_2| + |\beta_0| \leq L_0 \, \textup{ and } |\alpha_0|= |\beta_0| \leq N_0.
\end{align}
}

\begin{KCKCaaREV}
\end{KCKCaaREV}

where $\widetilde{\mathcal{L}(\lambda)}_{\vartheta_1, \vartheta_2}^{\nu} \Theta(x'', \lambda)$ is the differential operator given by 
$$ \widetilde{\mathcal{L}(\lambda)}_{\vartheta_1, \vartheta_2}^{\nu} \Theta(x'', \lambda) =  |\lambda|^{- |\vartheta_2|} D^{\nu} \partial_{\lambda}^{\vartheta_1} \Theta(x'', \lambda).$$ 

Let us explain the relation between conditions \eqref{def:grushin-kernel-Maucheri-hormander-cond1-L0-N0} and \eqref{def:grushin-kernel-Maucheri-hormander-cond1-L0-N0-for-symbols}, and conditions \eqref{def:grushin-kernel-Maucheri-hormander-cond2-L0-N0} and \eqref{def:grushin-kernel-Maucheri-hormander-cond2-L0-N0-for-symbols}, and the reason behind introducing these seemingly new conditions. 

Recall that $D_j$ denotes the commutator by the multiplication operator $x'_j$, that is, $D_j (M) = [x_j, M]$. It follows that $D_j = |\lambda|^{-1/2} \left( \delta_j (\lambda) + \bar{\delta}_j (\lambda) \right),$ and therefore $D^{\nu}$ can be expressed as a finite linear combination of operators of the form $|\lambda|^{-|\gamma_1 + \gamma_2|/2} \delta^{\gamma_1}(\lambda) \bar{\delta}^{\gamma_2}(\lambda)$ with $\gamma_1 + \gamma_2 = \nu$, implying that condition \eqref{def:grushin-kernel-Maucheri-hormander-cond1-L0-N0} imply those of the type \eqref{def:grushin-kernel-Maucheri-hormander-cond1-L0-N0-for-symbols} and condition \eqref{def:grushin-kernel-Maucheri-hormander-cond2-L0-N0} imply those of the type \eqref{def:grushin-kernel-Maucheri-hormander-cond2-L0-N0-for-symbols}. Also, a peek into the proofs of Section \ref{sec:mauceri-type-weighted-kernel-estimates} (see, for example, the analysis of the term $I^{1, \lambda, x''}_{k,+} (x^{\prime}, y^{\prime}) $ in \eqref{kernel-decom}) reveals that the proofs actually make use of the assumptions of the type \eqref{def:grushin-kernel-Maucheri-hormander-cond1-L0-N0-for-symbols} and \eqref{def:grushin-kernel-Maucheri-hormander-cond2-L0-N0-for-symbols} only. 

Now, in the particular case where the operator family $\Theta(x'', \lambda)$ is as in \eqref{def:theta-lambda-example-with-shifts}, we must work with conditions \eqref{def:grushin-kernel-Maucheri-hormander-cond1-L0-N0-for-symbols} and \eqref{def:grushin-kernel-Maucheri-hormander-cond2-L0-N0-for-symbols} only, otherwise the operators  $\delta^{\gamma_1}(\lambda) \bar{\delta}^{\gamma_2}(\lambda)$ would fall on the symbol function $m(x, \tau, \kappa)$ too, and thus producing large number of unwanted spatial derivatives of $m$. 

Summarising, with an understanding that Theorems \ref{thm:Mauceri-kernel-a=0-weak-11}, \ref{thm:Mauceri-kernel-a=0} and \ref{thm:Mauceri-kernel-a=0-more-derivative} hold true with conditions \eqref{def:grushin-kernel-Maucheri-hormander-cond1-L0-N0-for-symbols}, \eqref{def:grushin-kernel-Maucheri-hormander-cond2-L0-N0-for-symbols} in place of \eqref{def:grushin-kernel-Maucheri-hormander-cond1-L0-N0} and \eqref{def:grushin-kernel-Maucheri-hormander-cond2-L0-N0}, we have 
\begin{itemize}
\item Theorem \ref{thm:joint-shift-multi-weak-type} follows from Theorem \ref{thm:Mauceri-kernel-a=0-weak-11} if we take $L_0= Q +2$ and $N_0=0$.

\item Theorem \ref{thm:joint-shift-multi-sparse-less-derivative} follows from Theorem \ref{thm:Mauceri-kernel-a=0} if we take $L_0= \floor*{Q/2} + 3$ and $N_0=1$.

\item Theorem \ref{thm:joint-shift-multi-sparse-more-derivative} follows from Theorem \ref{thm:Mauceri-kernel-a=0-more-derivative} if we take $L_0= Q + 3$ and $N_0=1$. 
\end{itemize}

\medskip \noindent \textbf{\underline{$\Theta(x'', \lambda)$ satisfies the condition \eqref{def:grushin-kernel-Maucheri-hormander-cond1-L0-N0-for-symbols}}:} We shall show that if for some $L_0, N_0 \in \mathbb{N}$, the symbol function $m$ satisfies condition \eqref{assumption:decay-grushin-joint-symb} for all $|\alpha| + |\beta| + |\Gamma| \leq L_0$, $ |\Gamma| \leq  N_0$ and the cancellation condition \eqref{def:grushin-symb-vanishing-0-condition} for $|\beta'| \leq L_0$, then the family $\Theta(x'', \lambda)$ given by \eqref{def:theta-lambda-example-with-shifts} satisfies the condition \eqref{def:grushin-kernel-Maucheri-hormander-cond1-L0-N0-for-symbols}. 

With $\psi_j$ as in \eqref{def:spectral-break} and $\Upsilon^{q}$ given by \eqref{def:spectral-break-kappa-variable}, we define for every $j, q \in \mathbb{N}$ 
\begin{align*}
m^q_{j}(x, \tau, \kappa)=  m(x, \tau, \kappa) \, \psi_j(|(\tau, \kappa)|) \Upsilon^q (\kappa) = m_j(x, \tau, \kappa) \Upsilon^q (\kappa), 
\end{align*}
then in view of Lemma \ref{lem:finite-shifts} we have that 
\begin{align*}
\left( \widetilde{\mathcal{L}(\lambda)}_{\vartheta_1, \vartheta_2}^{\nu} {x'}^{\alpha_0} \partial^{\beta}_{x''} \Theta(x'', \lambda) \right) \chi_{j}(\lambda) = \sum_{|j - j'| \leq C L_0} \left( \widetilde{\mathcal{L}(\lambda)}_{\vartheta_1, \vartheta_2}^{ \nu} {x'}^{\alpha_0} \partial^{\beta}_{x''} \Theta_{j'} (x'', \lambda) \right) \chi_{j}(\lambda), 
\end{align*}
and therefore it suffices to work with $\left( \widetilde{\mathcal{L}(\lambda)}_{\vartheta_1, \vartheta_2}^{ \nu} {x'}^{\alpha_0} \partial^{\beta}_{x''} \Theta_{j} (x'', \lambda) \right) \chi_{j}(\lambda)$. 

Moreover, as already seen in Section \ref{subsec:mauceri-type-L2-weighted-kernel-estimates}, it suffices to prove estimates for $\Theta^q_{j}$ with bounds uniform in $q \in \mathbb{N}$. So, we pursue these estimates. 

Note that for $\tilde{c}$ coming from a fixed compact set, we have $|(2\mu + \tilde{c} |\lambda|, \lambda)| \sim |(2\mu + \tilde{c}) |\lambda| |$. Now, by Lemma \ref{weighted-kernel-estimate-3} and the orthogonality of (scaled) Hermite functions, the kernel of the operator 
$Op \left( \left( \widetilde{\mathcal{L}(\lambda)}_{\vartheta_1, \vartheta_2}^{ \nu} {x'}^{\alpha_0} \partial^{\beta}_{x''} \Theta_{j}^{q}(x'', \lambda) \right) \chi_{j}(\lambda)\right)$ can be written as finite linear combination of terms of the forms 
\begin{align} \label{eq:cancel-calc-kc1lo-1} 
& {x^{\prime}}^{\alpha_{1}} \int_{[0,1]^{N_1} \times \Omega^{N_2} \times [0,1]^{|\nu|}} g(\omega) \int_{\mathbb{R}^{n_2}} e^{-i \lambda \cdot (x^{\prime \prime} - y^{\prime \prime})} \mathfrak{A}_l(\lambda) |\lambda|^{-|\vartheta_2|} \, {x'}^{\alpha_0} \\ 
\nonumber & \quad \sum_{2^j \leq (2|\mu|+n_1)|\lambda| < 2^{j+1}} C_{\mu, \tilde{c}, \vec{c}} \left( \tau^{\frac{1}{2} \theta_1} \partial_\tau^{\theta_2} \partial_{\kappa}^{\beta_2} \partial^{\beta}_{x''} m_j^q\right) (x,(2\mu + \tilde{c} + \vec{c}(\omega)) |\lambda|,\lambda) \Phi_{\mu + \tilde{\mu}}^{\lambda}(x^{\prime}) \Phi_{\mu}^{\lambda}(y^{\prime}) \, d \lambda \, d\omega, 
\end{align} 
where $N_1, N_2 \leq |\vartheta_1| - (|\beta_2| + l)$, $| \alpha_1 | \leq |\vartheta_1| - (|\beta_2|+l)$, $ |\mathcal{\theta}|_1 \leq |\mathcal{\theta}_2| \leq |\nu| + 2 |\vartheta_1| - 2  (|\beta_2|+l)$, $|\mathcal{\theta}_2| - \frac{|\mathcal{\theta}_1|}{2} = \frac{|\nu|}{2} + |\vartheta_1| -  (|\beta_2|+l)- \frac{|\alpha_1|}{2} $, $C_{\mu, \tilde{c}, \vec{c}}$ is a bounded function of $\mu$ and $\vec{c}$, and $\mathfrak{A}_l$ is a continuous function on $\mathbb{R}^{n_2} \setminus \{0\}$ which is homogeneous of degree $-l$.

By Leibniz formula in $\kappa$ variable, for $\beta_3 \leq \beta_2$, it suffices to consider the kernel of the form 
\begin{align*} 
& {x^{\prime}}^{\alpha_{1}}\int_{[0,1]^{N_1} \times \Omega^{N_2} \times [0,1]^{|\nu|}} g(\omega) \int_{\mathbb{R}^{n_2}} e^{-i \lambda \cdot (x^{\prime \prime} - y^{\prime \prime})} \mathfrak{A}_l(\lambda) |\lambda|^{-|\vartheta_2|} \partial^{\beta_2-\beta_3}_{\lambda} \Upsilon^{q}(\lambda)\\
\nonumber & \sum_{2^j \leq (2|\mu|+n_1)|\lambda| < 2^{j+1}} {x'}^{\alpha_0} C_{\mu, \tilde{c}, \vec{c}} \left(\tau^{ \frac{1}{2} \theta_1} \partial_\tau^{\theta_2}\partial_{\kappa}^{\beta_3} \partial^{\beta}_{x''} m_j\right) (x,(2\mu + \tilde{c} + \vec{c}(\omega)) |\lambda|,\lambda) \Phi_{\mu + \tilde{\mu}}^{\lambda}(x^{\prime}) \Phi_{\mu}^{\lambda}(y^{\prime}) \, d \lambda \, d\omega.
\end{align*}

From here on we drop the integration on the compact set $[0,1]^{N_1} \times \Omega^{N_2} \times [0,1]^{|\nu|}$ as it will become clear from the calculations that the estimate that we get are uniform in $\omega$ (for fixed $N_1$, $N_2$ and $\nu$). Now, 
\begin{align}\label{Pseudo-shift-first-condition} 
& \int_{\mathfrak{E}_{j} (x^{\prime})} \int_{\mathbb{R}^{n_2}}
\left| \frac{ \left( \left( \widetilde{\mathcal{L}(\lambda)}_{\vartheta_1, \vartheta_2}^{ \nu} {x'}^{\alpha_0} \partial^{\beta_0}_{x''} \Theta^q_j (x'', \lambda) \right) \chi_{j}(\lambda)\right)(x^{\prime}, y^{\prime})}{(|x^{\prime}|+|y^{\prime}|)^{|\vartheta_1 + \vartheta_2|}} \right|^2 d\lambda \, dy^{\prime} \\ 
\nonumber & \lesssim 2^{j(|\vartheta_1 + \vartheta_2|-|\alpha_1|)}  \int_{\mathbb{R}^{n_2}}
\left| \mathfrak{A}_l(\lambda)\right|^{2}  |\lambda|^{-2|\vartheta_2|} \left| \partial^{\beta_2 - \beta_3}_{\lambda} \Upsilon^{q}(\lambda)\right|^2 \\
\nonumber & \quad \quad \quad \quad \quad \quad  \sum_{\mu} \left| {x'}^{\alpha_0} \left(\tau^{ \frac{1}{2} \theta_1}  \partial_\tau^{\theta_2}\partial_{\kappa}^{\beta_3}  \partial^{\beta_0}_{x''} m_j\right) (x,(2\mu + \tilde{c} + \vec{c}(\omega)) |\lambda|,\lambda)  \right|^2 \left| \Phi_{\mu + \tilde{\mu}}^{\lambda}(x^{\prime}) \right|^2 d\lambda. 
\end{align}
Using condition \eqref{def:grushin-symb-vanishing-0-condition}, we get as an application of Taylor's theorem that 
$$ \partial_\tau^{\theta_2} \partial_{\kappa}^{\beta_3}\partial^{\beta_0}_{x''} m_j (x, \tau,\kappa) = \sum_{\tilde{\beta}_3 \geq \beta_3 : |\tilde{\beta}_3| = |\beta_2| + l + |\vartheta_2|} \frac{ \kappa^{\tilde{\beta}_3 - \beta_3} }{ \left( \tilde{\beta}_3 - \beta_3 \right)!} \int_0^1 \partial_\tau^{\theta_2} \partial_{\kappa}^{\tilde{\beta}_3} \partial^{\beta_0}_{x''} m_j (x, \tau, t\kappa) \, dt, $$
which implies that 
\begin{align*}
\left| \partial_\tau^{\theta_2} \partial_{\kappa}^{\beta_3}\partial^{\beta_0}_{x''} m_j (x, \tau,\kappa) \right|^2 \lesssim \left( |\kappa|^{|\beta_2| + l+ |\vartheta_2| - |\beta_3|} \right)^2  \sum_{\tilde{\beta}_3 \geq \beta_3 : |\tilde{\beta}_3| = |\beta_2| + l+ |\vartheta_2|}   \int_0^1 \left|  \partial_\tau^{\theta_2}\partial_{\kappa}^{\tilde{\beta}_3} \partial^{\beta_0}_{x''} m_j (x, \tau, t\kappa) \right|^2 \, dt.
\end{align*}

Now, since 
$$ \sup_{q \in \mathbb{N}} \sup_{\kappa \in \mathbb{R}^{n_2} \setminus \{0\}} |\kappa|^{|\beta_2| + l - |\beta_3|} \left| \mathfrak{A}_l(\kappa) \right| \left| \partial^{\beta_2-\beta_3}_{\kappa} \Upsilon^{q}(\kappa) \right| \leq C_{\beta_2, \beta_3, l} < \infty,$$ 
therefore \eqref{Pseudo-shift-first-condition} is dominated by $2^{j(|\vartheta_1 + \vartheta_2|-|\alpha_1|)}$ times a finite sum of terms of the form 
\begin{align*}
& \int_0^1 \int_{\mathbb{R}^{n_2}} \sum_{\mu} \left| \left(\tau^{ \frac{1}{2} \theta_1} \partial_\tau^{\theta_2}\partial_{\kappa}^{\tilde{\beta}_3} \left( {x'}^{\alpha_0} \partial^{\beta_0}_{x''} m_j \right) \right) (x,(2 \mu + \tilde{c} + \vec{c}(\omega)) |\lambda|, t \lambda) 
\right|^2 \left| \Phi_{\mu + \tilde{\mu}}^{\lambda}(x^{\prime}) \right|^2 d\lambda \, dt \\
& = \int_0^1 \int_{\mathbb{R}^{n_2}} \sum_{\mu} \left| \left(\tau^{ \frac{1}{2} \theta_1} \partial_\tau^{\theta_2}\partial_{\kappa}^{\tilde{\beta}_3} \left( {x'}^{\alpha_0} \partial^{\beta_0}_{x''} m_j \right) \right) (x,(2 \mu - 2 \tilde{\mu} + \tilde{c} + \vec{c}(\omega)) |\lambda|, t \lambda) 
\right|^2 \left| \Phi_{\mu }^{\lambda}(x^{\prime}) \right|^2 d\lambda \, dt \\
& \lesssim \int_{\mathbb{R}^{n_2}} \sum_{\mu} \left| \sup_{t \in [0,1]} \left| \left(\tau^{ \frac{1}{2} \theta_1} \partial_\tau^{\theta_2} \partial_{\kappa}^{\tilde{\beta}_3} \left( {x'}^{\alpha_0} \partial^{\beta_0}_{x''} m_j \right) \right) (x,(2 \mu - 2 \tilde{\mu} + \tilde{c} + \vec{c}(\omega)) |\lambda|, t \lambda) \right| \right|^2 \left| \Phi_{\mu }^{\lambda}(x^{\prime}) \right|^2 d\lambda \\ 
& \lesssim \int_{\mathbb{R}^{n_2}} \sum_{\mu} \left| \sup_{\kappa \in \mathbb{R}^{n_2}} \left| \left(\tau^{ \frac{1}{2} \theta_1} \partial_\tau^{\theta_2} \partial_{\kappa}^{\tilde{\beta}_3} \left( {x'}^{\alpha_0} \partial^{\beta_0}_{x''} m_j \right) \right) (x,(2 \mu - 2 \tilde{\mu} + \tilde{c} + \vec{c}(\omega)) |\lambda|, \kappa) \right| \right|^2 \left| \Phi_{\mu }^{\lambda}(x^{\prime}) \right|^2 d\lambda \\ 
& \lesssim \sup_{x_0} \int_{\mathbb{R}^{n_2}} \sum_{\mu} \left| \widetilde{m}_j^{\theta_1, \theta_2, \tilde{\beta}_3, \beta_0, \tilde{\alpha}} (x_0, (2 \mu - 2 \tilde{\mu} + \tilde{c} + \vec{c}(\omega)) |\lambda|) \right|^2 \left| \Phi_{\mu }^{\lambda}(x^{\prime}) \right|^2 d\lambda, 
\end{align*}
where in the final line we have taken 
$$ \widetilde{m}_j^{\theta_1, \theta_2, \tilde{\beta}_3, \beta_0,\tilde{\alpha}} (x_0, \tau) = \left| \sup_{\kappa \in \mathbb{R}^{n_2}} \left| \left(\tau^{ \frac{1}{2} \theta_1} \partial_\tau^{\theta_2} \partial_{\kappa}^{\tilde{\beta}_3} \left( {x'}^{\alpha_0} \partial^{\beta_0}_{x''} m_j \right) \right) (x_0, \tau, \kappa) \right| \right|.$$

It is clear that as a function of $\tau$-variable, support of $\widetilde{m}_j^{\theta_1, \theta_2, \tilde{\beta}_3, \beta_0,\tilde{\alpha}} (x_0, \tau) $ is contained in $[-2^{j+2}, 2^{j+2}]^{n_1}$. With that, we can make use of the calculations of Lemmas 10 and 11 of \cite{MartiniMullerGrushinRevistaMath} where the authors worked with estimates of Hermite functions to handle shifts of the form appearing in our case too, and then we can conclude from the final estimate above that \eqref{Pseudo-shift-first-condition} is dominated by a finite sum of terms of the form 
\begin{align*}
& |B(x, 2^{-j/2})|^{-1} 2^{j(|\vartheta_1 + \vartheta_2|-|\alpha_1|)} \sup_{x_0 \in \mathbb{R}^{n_1 + n_2}} \left\| \widetilde{m}_j^{\theta_1, \theta_2, \tilde{\beta}_3, \beta_0, \tilde{\alpha}} (x_0, \cdot) \right\|^2_{L^{\infty}} \\
& \lesssim |B(x, 2^{-j/2})|^{-1} 2^{j(|\vartheta_1 + \vartheta_2|-|\alpha_1|)} \sup_{x_0 \in \mathbb{R}^{n_1 + n_2}} \left\| \sup_{\kappa \in \mathbb{R}^{n_2}} \left| \left(\tau^{ \frac{1}{2} \theta_1} \partial_\tau^{\theta_2} \partial_{\kappa}^{\tilde{\beta}_3} \left( {x'}^{\alpha_0} \partial^{\beta_0}_{x''} m_j \right) \right) (x_0, \tau, \kappa) \right| \right\|^2_{L^{\infty} (d \tau)} \\ 
& \lesssim |B(x, 2^{-j/2})|^{-1} 2^{-j(|\nu| + |\vartheta_1 + \vartheta_2|)} 2^{ j |\beta_0|} 
\end{align*}
where the last inequality follows from the conditions of \eqref{eq:cancel-calc-kc1lo-1}. 

This completes the proof of our claim on $\Theta(x'', \lambda)$ satisfying condition \eqref{def:grushin-kernel-Maucheri-hormander-cond1-L0-N0-for-symbols}. 

\medskip \noindent \textbf{\underline{$\Theta(x'', \lambda)$ satisfies the condition \eqref{def:grushin-kernel-Maucheri-hormander-cond2-L0-N0-for-symbols}}:}
In an analogous manner, one can show that if for some $L_0, N_0 \in \mathbb{N}$, the symbol function $m$ satisfies condition \eqref{assumption:decay-grushin-joint-symb} for all $|\alpha| + |\beta| + |\Gamma| \leq L_0$, $ |\Gamma| \leq  N_0$ and the cancellation condition \eqref{def:grushin-symb-vanishing-0-condition} for $|\beta'| \leq L_0$, then the family $\Theta(x'', \lambda)$ satisfies the condition \eqref{def:grushin-kernel-Maucheri-hormander-cond2-L0-N0-for-symbols}, so we omit the details. 

\section*{Acknowledgements} 
SB and RG were supported in parts from their individual INSPIRE Faculty Fellowships from DST, Government of India. RB was supported by the Senior Research Fellowship from CSIR, Government of India. AG was supported by Centre for Applicable Mathematics, TIFR.

\providecommand{\bysame}{\leavevmode\hbox to3em{\hrulefill}\thinspace}
\providecommand{\MR}{\relax\ifhmode\unskip\space\fi MR }
\providecommand{\MRhref}[2]{%
  \href{http://www.ams.org/mathscinet-getitem?mr=#1}{#2}
}
\providecommand{\href}[2]{#2}


\begin{thebibliography}{BBGG23}

\bibitem[Bag21]{BagchiFourierMultipliersHeisenbergStudia}
Sayan Bagchi, \emph{Fourier multipliers on the {H}eisenberg group revisited},
  Studia Math. \textbf{260} (2021), no.~3, 241--272. \MR{4296727}

\bibitem[BBGG22]{BBGG-2}
Sayan Bagchi, Riju Basak, Rahul Garg, and Abhishek Ghosh, \emph{Sparse bounds
  for pseudo-multipliers associated to {G}rushin operators, {II}},
  {https://arxiv.org/abs/2202.10269} (2022).

\bibitem[BBGG23]{BBGG-1}
\bysame, \emph{Sparse bounds for pseudo-multipliers associated to {G}rushin
  operators, {I}}, J. Fourier Anal. Appl. \textbf{29} (2023), no.~3, Paper No.
  27, 38. \MR{4585471}

\bibitem[BG23]{Bagchi-Garg-1}
Sayan Bagchi and Rahul Garg, \emph{On {$L^2$}-boundedness of pseudo-multipliers
  associated to the grushin operator}, {https://arxiv.org/abs/2111.10098v3}
  (2023).

\bibitem[BT15]{BagchiThangaveluHermitePseudo}
Sayan Bagchi and Sundaram Thangavelu, \emph{On {H}ermite pseudo-multipliers},
  J. Funct. Anal. \textbf{268} (2015), no.~1, 140--170. \MR{3280055}

\bibitem[CCM22]{Martini-IMRN-2022}
Valentina Casarino, Paolo Ciatti, and Alessio Martini, \emph{Weighted spectral
  cluster bounds and a sharp multiplier theorem for ultraspherical {G}rushin
  operators}, Int. Math. Res. Not. IMRN (2022), no.~12, 9209--9274.
  \MR{4436206}

\bibitem[DJ16]{Dziubanski-Jotsaroop-Hardy-BMO-Grushin}
Jacek Dziuba\'{n}ski and K.~Jotsaroop, \emph{On {H}ardy and {BMO} spaces for
  {G}rushin operator}, J. Fourier Anal. Appl. \textbf{22} (2016), no.~4,
  954--995. \MR{3528406}

\bibitem[DM22]{Martini-JFAA-2022}
Gian~Maria Dall'Ara and Alessio Martini, \emph{An optimal multiplier theorem
  for {G}rushin operators in the plane, {II}}, J. Fourier Anal. Appl.
  \textbf{28} (2022), no.~2, Paper No. 32, 29. \MR{4402142}

\bibitem[DMM79]{De-Michele-Mauceri-Heisenberg79}
Leonede De~Michele and Giancarlo Mauceri, \emph{{$L^{p}$} multipliers on the
  {H}eisenberg group}, Michigan Math. J. \textbf{26} (1979), no.~3, 361--371.
  \MR{544603}

\bibitem[DSY11]{DuongSikoraYanJFA2011}
Xuan~Thinh Duong, Adam Sikora, and Lixin Yan, \emph{Weighted norm inequalities,
  {G}aussian bounds and sharp spectral multipliers}, J. Funct. Anal.
  \textbf{260} (2011), no.~4, 1106--1131. \MR{2747017}

\bibitem[Epp96]{EppersonHermitePseudo}
Jay Epperson, \emph{Hermite multipliers and pseudo-multipliers}, Proc. Amer.
  Math. Soc. \textbf{124} (1996), no.~7, 2061--2068. \MR{1343690}

\bibitem[Gru70]{Grushin70}
V.~V. Gru{\v{s}}in, \emph{A certain class of hypoelliptic operators}, Mat. Sb.
  (N.S.) \textbf{83 (125)} (1970), 456--473. \MR{0279436}

\bibitem[JST13]{JotsaroopSanjayThangaveluRieszTransformsGrushin}
K.~Jotsaroop, P.~K. Sanjay, and S.~Thangavelu, \emph{Riesz transforms and
  multipliers for the {G}rushin operator}, J. Anal. Math. \textbf{119} (2013),
  255--273. \MR{3043153}

\bibitem[Lin03]{Chin-Cheng-Lin-Heisenberg}
Chin-Cheng Lin, \emph{H\"{o}rmander's {$H^p$} multiplier theorem for the
  {H}eisenberg group}, J. London Math. Soc. (2) \textbf{67} (2003), no.~3,
  686--700. \MR{1967700}

\bibitem[LO20]{Lerner-Ombrosi-pointwaise-sparse2020}
Andrei~K. Lerner and Sheldy Ombrosi, \emph{Some remarks on the pointwise sparse
  domination}, J. Geom. Anal. \textbf{30} (2020), no.~1, 1011--1027.
  \MR{4058547}

\bibitem[Lor21]{Lorist-pointwaise-sparse2021}
Emiel Lorist, \emph{On pointwise {$\ell^r$}-sparse domination in a space of
  homogeneous type}, J. Geom. Anal. \textbf{31} (2021), no.~9, 9366--9405.
  \MR{4302224}

\bibitem[Mar17]{MartiniJointFunctionalCalculiMathZ}
Alessio Martini, \emph{Joint functional calculi and a sharp multiplier theorem
  for the {K}ohn {L}aplacian on spheres}, Math. Z. \textbf{286} (2017),
  no.~3-4, 1539--1574. \MR{3671588}

\bibitem[Mau80]{MauceriWeylTransformJFA80}
Giancarlo Mauceri, \emph{The {W}eyl transform and bounded operators on
  {$L^{p}({\bf R}^{n})$}}, J. Functional Analysis \textbf{39} (1980), no.~3,
  408--429. \MR{600625}

\bibitem[MM14]{MartiniMullerGrushinRevistaMath}
Alessio Martini and Detlef M\"{u}ller, \emph{A sharp multiplier theorem for
  {G}rushin operators in arbitrary dimensions}, Rev. Mat. Iberoam. \textbf{30}
  (2014), no.~4, 1265--1280. \MR{3293433}

\bibitem[MS12]{MartiniSikoraGrushinMRL}
Alessio Martini and Adam Sikora, \emph{Weighted {P}lancherel estimates and
  sharp spectral multipliers for the {G}rushin operators}, Math. Res. Lett.
  \textbf{19} (2012), no.~5, 1075--1088. \MR{3039831}

\bibitem[RS08]{RobinsonSikoraDegenerateEllipticOperatorsGrushinTypeMathZ2008}
Derek~W. Robinson and Adam Sikora, \emph{Analysis of degenerate elliptic
  operators of {G}ru\v{s}in type}, Math. Z. \textbf{260} (2008), no.~3,
  475--508. \MR{2434466}

\bibitem[TDOS02]{DuongOuhabazSikoraWeightedPlancherel2002JFA}
Xuan Thinh~Duong, El~Maati Ouhabaz, and Adam Sikora, \emph{Plancherel-type
  estimates and sharp spectral multipliers}, J. Funct. Anal. \textbf{196}
  (2002), no.~2, 443--485. \MR{1943098}

\bibitem[Wei01]{Lutz-Weis-R-Boundedness-2001}
Lutz Weis, \emph{Operator-valued {F}ourier multiplier theorems and maximal
  {$L_p$}-regularity}, Math. Ann. \textbf{319} (2001), no.~4, 735--758.
  \MR{1825406}

\end{thebibliography}
\end{document}